\documentclass[12pt]{amsart}
\usepackage{bbm}
\usepackage[pdftex]{graphicx}
\usepackage{amscd, amssymb, dsfont, upgreek, mathrsfs}
\input{xy}
\xyoption{all}
\newtheorem{Thm}{Theorem}
\newtheorem{Conj}[Thm]{Conjecture}
\newtheorem{Prop}[Thm]{Proposition}
\newtheorem{Def}[Thm]{Definition}
\newtheorem{Def/Thm}[Thm]{Definition/Theorem}
\newtheorem{Cor}[Thm]{Corollary}
\newtheorem{Lemma}[Thm]{Lemma}

\theoremstyle{remark}
\newtheorem{Rmk}[Thm]{Remark}

\newcommand{\F}{{\mathsf{M}}}

\newcommand{\ot }{\otimes}
\newcommand{\ra }{\rightarrow}

\newcommand{\lra }{\longrightarrow}
\newcommand{\Spec}{{\mathrm{Spec}}}

\newcommand{\lann}{\langle\langle}
\newcommand{\rann}{\rangle\rangle}
\newcommand{\lannn}{\left\langle\left\langle}
\newcommand{\rannn}{\right\rangle\right\rangle}

\newcommand{\G}{{\bf G}}

\newcommand{\PP }{{\mathbb P}}

\newcommand{\QQ }{{\mathbb Q}}
\newcommand{\CC }{{\mathbb C}}
\newcommand{\ZZ }{{\mathbb Z}}

\newcommand{\vir}{\mathrm{vir}}

\newcommand{\DD}{\mathsf{D}}

\newcommand{\T}{{\mathsf{T}}}

\newcommand{\lan}{\langle}
\newcommand{\ran}{\rangle}

\newcommand{\dsI}{\mathds{I}}

\newcommand{\pP}{\mathsf{P}}

\newcommand{\ppl}{{\mathsf{P}}\left[}
\newcommand{\ppr}{\right]}

\newcommand{\rarr}{\longrightarrow}

\def \proj {{\mathbb{P}}}
\newcommand{\com}{{\mathbb{C}}}

\newcommand{\oh}{{\mathcal O}}

\begin{document}

\title[Equivariant holomorphic anomaly equation]
{Equivariant holomorphic anomaly equation}

\author{Hyenho Lho}
\address{Department of Mathematics, ETH Z\"urich}
\email {hyenho.lho@math.ethz.ch}
\date{July 2018}.

\begin{abstract} 
In \cite{LP} the fundamental relationship between stable quotient
invariants and the B-model for local $\PP^2$
in all genera was studied under some specialization of equivariant variables. We generalize the argument of \cite{LP} to full equivariant settings without the specialization. Our main results are the proof of holomorphic anomaly equations for the equivariant Gromov-Witten theories of local $\PP^2$ and local $\PP^3$. We also state the generalization to full equivariant formal quintic theory of the result in \cite{LP2}.

 \end{abstract}

\maketitle

\setcounter{tocdepth}{1} 
\tableofcontents

\setcounter{section}{-1}

\section{Introduction}

\subsection{Equivariant local $\PP^n$ theories.}\label{twth}
Equivariant local $\PP^n$ theories can be constructed as follows. Let the algebraic torus
$$\mathsf{T}_{n+1}=(\CC^*)^{n+1}$$
act with the standard linearization on $\PP^n$ with weights $\lambda_0,\dots,\lambda_n$ on the vector space $H^0(\PP^n,\mathcal{O}_{\PP^n}(1))$.
 Let $\overline{M}_g(\PP^n,d)$ be the moduli space of stable maps to $\PP^n$ equipped with the canonical $\mathsf{T}_{n+1}$-action, and let
 $$\mathsf{C}\rightarrow\overline{M}_g(\PP^n,d)\,,\,\,f:\mathcal{C}\rightarrow\PP^n\,,\,\,\mathsf{S}=f^*\mathcal{O}_{\PP^n}(-1)\rightarrow\mathsf{C}$$
be the standard universal structures. 

The equivariant Gromov-Witten invariants of the local $\PP^n$ are defined via the equivariant integrals
\begin{align}\label{GW}
    N_{g,d}^{\mathsf{GW}}=\int_{[\overline{M}_g(\PP^n,d)]^{\text{vir}}}e\Big(-R\pi_*f^* \mathcal{O}_{\PP^n}(-n-1)\Big)\,.
\end{align}
The integral \eqref{GW} defines a rational function in $\lambda_i$
$$N_{g,d}^{\mathsf{GW}}\in \QQ(\lambda_0,\dots,\lambda_n)\,.$$

Over the moduli space of stable quotients, there is a universal
 curve
\begin{equation}\label{ggtt}
\pi: \mathcal{C} \rightarrow \overline{Q}_{g}(\PP^n,d)
\end{equation}
with a universal quotient
$$0 \rarr \mathsf{S} \rarr \com^N \otimes \oh_{\mathcal{C}} \stackrel{q_U}{\rarr} 
\mathsf{Q}\rarr 0.$$
The equivariant stable quotient invariants of the local $\PP^n$ are defined via the equivariant integrals
\begin{align}\label{SQ}
    N^{\mathsf{SQ}}_{g,d}=\int_{[\overline{Q}_g(\PP^n,d)]^{\text{vir}}}e\Big(-R\pi_*\mathsf{S}\Big)\,.
\end{align}
The integral \eqref{SQ} also defines a rational function in $\lambda_i$
$$N_{g,d}^{\mathsf{SQ}}\in \QQ(\lambda_0,\dots,\lambda_n)\,.$$
We refer the reader to \cite[Section 1]{LP} for a more leisurely treatment of stable quotients.

In \cite{LP} it was observed that the analysis of $I$-function in \cite{ZaZi} plays important role in the study of local $\PP^n$ theories. But the result in \cite{ZaZi} holds only after the specialization to $(n+1)$-th root of unity $\zeta_{n+1}$,
\begin{align*}
    \lambda_i=\zeta_{n+1}^i\,.
\end{align*}
In order to generalize the results in \cite{LP} to full equivariant theories, one needs the analogous generalization of the results in \cite{ZaZi} to full equivariant settings. This will be studied in Appendix.

\subsection{Holomorphic anomaly for $K\PP^2$.}\label{holp2} We state the precise form of the holomorphic anomaly equations for local $\PP^2$. Denote by $K\PP^2$ the total space of the canonical bundle over $\PP^2$.
 Let $H\in H^2(K\PP^2,\QQ)$ be the hyperplane class obtained from $\PP^2$, and let
 
\begin{align*}
    \mathcal{F}^{\mathsf{GW}}_{g,m}(Q)=\lan H,\dots,H\ran^{\mathsf{GW}}_{g,m}=\sum^{\infty}_{d=0}Q^d\int_{[\overline{M}_{g,m}(K\PP^2,d)]^{\text{vir}}}\prod^m_{i=1}\text{ev}^*_i(H)\,,\\
    \mathcal{F}^{\mathsf{SQ}}_{g,m}(q)=\lan H,\dots,H\ran^{\mathsf{SQ}}_{g,m}=\sum^{\infty}_{d=0}Q^d\int_{[\overline{Q}_{g,m}(K\PP^2,d)]^{\text{vir}}}\prod^m_{i=1}\text{ev}^*_i(H)\,.
\end{align*} 
be the Gromov-Witten and stable quotient series respectively (involving the evaluation morphisms at the markings). The relationship between the Gromov-Witten and stable quotient invariants of $K\PP^2$ is proven in \cite{CKg} in case $2g-2+n > 0$:
\begin{align}\label{3456}
    \mathcal{F}^{\mathsf{GW}}_{g,m}(Q(q))=\mathcal{F}^{\mathsf{SQ}}_{g,m}(q)\,,
\end{align}
where $Q(q)$ is the mirror map,
$$I_1^{K\PP^2}(q)=\text{log}(q)+3\sum^{\infty}_{d=1}(-q)^d\frac{(3d-1)!}{(d!)^3}\,,$$

$$Q(q)=\text{exp}\left(I_1^{K\PP^2}(q)\right)=q\cdot \text{exp}\left(3\sum^{\infty}_{d=1}(-q)^d\frac{(3d-1)!}{(d!)^3}\right)\,.$$

To state the holomorphic anomaly equations, we need the following additional series in $q$.

\begin{align*}
    L(q)&=(1+27q)^{-\frac{1}{3}}=1-9q+162q^2+\dots\,,\\
    C_1(q)&=q\frac{d}{dq}I_1^{K\PP^2}\,,\\
    A_2(q)&=\frac{1}{L^3}\left(3 \frac{q\frac{d}{dq}C_1}{C_1}+1-\frac{L^3}{2}\right)\,.
\end{align*}
We also need new series $L_i(q)$ defined by roots of following degree $3$ polynomial in $\mathcal{L}$ for $i=0,1,2$:

$$(1+27q)\mathcal{L}^3-(\lambda_0+\lambda_1+\lambda_2)\mathcal{L}^2+(\lambda_0\lambda_1+\lambda_1\lambda_2+\lambda_2\lambda_0)\mathcal{L}-\lambda_0\lambda_2\lambda_3\,,$$
with initial conditions,
$$L_i(0)=\lambda_i\,.$$
Let $f_2$ be the polynomial of degree $2$ in variable $x$ over $\CC(\lambda_0,\lambda_1,\lambda_2)$ defined by
$$f_2(x):=(\lambda_0+\lambda_1+\lambda_2)x^2-2(\lambda_0\lambda_1+\lambda_1\lambda_2+\lambda_2\lambda_0)x+3\lambda_0\lambda_1\lambda_2\,.$$
The ring $$\mathds{G}_2:=\CC(\lambda_0,\lambda_1,\lambda_2)[L_0^{\pm 1},L_1^{\pm 1},L_2^{\pm 1},f(L_0)^{-\frac{1}{2}},f(L_1)^{-\frac{1}{2}},f(L_2)^{-\frac{1}{2}}]$$ will play a basic role in our paper. Consider the free polynomial rings in the variables $A_2$ and $C_1^{-1}$ over $\mathds{G}_2$,
\begin{align}\label{LOR}
    \mathds{G}_2[A_2]\,\,,\,\,\,\mathds{G}_2[A_2,C_1^{-1}]\,.
\end{align}
We have canonical maps
\begin{align}\label{LOM2}
    \mathds{G}_2[A_2]\rightarrow \CC(\lambda_0,\lambda_1,\lambda_2)[[q]]\,\,,\,\,\,\mathds{G}_2[A_2,C_1^{-1}]\rightarrow \CC(\lambda_0,\lambda_1,\lambda_2)[[q]]
\end{align}
given by assigning the above defined series $A_2(q)$ and $C_1^{-1}$ to the variables $A_2$ and $C_1^{-1}$ respectively. Therefore we can consider elements of the rings \eqref{LOR} either as free polynomials in the variables $A_2$ and $C_1^{-1}$ or as series in $q$.

Let $F(q)\in \CC(\lambda_0,\lambda_1,\lambda_2)[[q]]$ be a series in $q$. When we write
$$F(q)\in \mathds{G}_2[A_2]\,,$$
we mean there is a cononical lift $F\in\mathds{G}_2[A_2]$ for which
$$F\rightarrow F(q)\in\CC(\lambda_0,\lambda_1,\lambda_2)[[q]]$$
under the map \eqref{LOM2}. The symbol $F$ without the argument $q$ is the canonical lift. The notation
$$F(q)\in\mathds{G}_2[A_2,C_1^{-1}]$$
is parallel.

Let $T$ be the standard coordinate mirror to $t=\text{log}(q)$,
$$T=I^{K\PP^2}_1(q)\,.$$
Then $Q(q)=\text{exp}(T)$ is the mirror map.
\begin{Conj}\label{ooo}
 For the stable quotient invariants of $K\PP^2$,
 \begin{itemize}
  \item[(i)] $\mathcal{F}^{\text{SQ}}_g(q)\in \mathds{G}_2[A_2]$ for $g\ge2$,
  \item[(ii)] $\mathcal{F}^{\mathsf{SQ}}_g$ is of degree at most $3g-3$ with respect to $A_2$,
  \item[(iii)] $\frac{\partial^k \mathcal{F}^{\mathsf{SQ}}}{\partial T^k}(q)\in \mathds{G}_2[A_2,C_1^{-1}]$ for $g\ge 1$ and $k\ge 1$,
  \item[(iv)] $\frac{\partial^k\mathcal{F}^{\mathsf{SQ}_g}}{\partial T^k}$ is homogeneous of degree $k$ with respect to $C_1^{-1}$.
 \end{itemize}
\end{Conj}
\noindent Here, $\mathcal{F}^{\mathsf{SQ}}_g=\mathcal{F}^{\mathsf{SQ}}_{g,0}$.

\begin{Conj}\label{HAE}
The holomorphic anomaly equations for the stable quotient invariants of $K\PP^2$ hold for $g\ge 2$:
$$\frac{1}{C_1^2}\frac{\partial \mathcal{F}^{\mathsf{SQ}}_g}{\partial A_2}=\frac{1}{2}\sum^{g-1}_{i=1} \frac{\partial\mathcal{F}^{\mathsf{SQ}}_{g-i}}{\partial T}\frac{\partial\mathcal{F}^{\mathsf{SQ}}_{i}}{\partial T}+\frac{1}{2}\frac{\partial^2\mathcal{F}^{\mathsf{SQ}}_{g-1}}{\partial T^2}\,.$$
\end{Conj}

The derivative of $\mathcal{F}^{\mathsf{SQ}}_g$ with respect to $A_2$ in the above equation is well-defined since
$$\mathcal{F}^{\mathsf{SQ}}_g\in \mathds{G}_2[A_2]$$
by part (i) of Conjecture \ref{HAE}. By parts (ii) and (iii),
$$\frac{\partial \mathcal{F}^{\mathsf{SQ}}_{g-i}}{\partial T}\frac{\partial \mathcal{F}^{\mathsf{SQ}}_{i}}{\partial T}\,\,,\,\,\frac{\partial^2 \mathcal{F}^{\mathsf{SQ}}_{g-1}}{\partial T^2}\in \mathds{G}_2[A_2,C_1^{-1}]$$
are both of degree $2$ in $C_1^{-1}$. Hence, the holomorphic anomaly equation of Conjecture \ref{HAE} may be viewed as holding in $\mathds{G}[A_2]$ since the factors of $C_1^{-1}$ on both sides cancel. If we use the specializations by primitive third root of unity $\zeta$
$$\lambda_i=\zeta^i\,,$$
the holomorphic anomaly equations here for $K\PP^2$ recover the precise form presented in \cite[(4.27)]{ASYZ} via B-model physics.

Conjecture \ref{HAE} determine $\mathcal{F}_g^{\mathsf{SQ}}\in\mathds{G}_2[A_2]$ uniquely as a polynomial in $A_2$ up to a constant term in $\mathds{G}_2$. The degree of the constant term can be bounded. Therefore Conjecture \ref{HAE} determine $\mathcal{F}_g^{\mathsf{SQ}}$ from the lower genus theory together with a finite amount of date.

We will prove the following special cases of the conjectures in Section \ref{hafp}.

\begin{Thm}\label{MT1}
Conjecture \ref{ooo} holds for the choices of $\lambda_0,\lambda_1,\lambda_2$ such that
$$\lambda_i \ne \lambda_j \,\, \text{for}\,\,i \ne j\,,\\$$
$$(\lambda_0\lambda_1+\lambda_1\lambda_2+\lambda_2\lambda_0)^2-3\lambda_0\lambda_1\lambda_2(\lambda_0+\lambda_1+\lambda_2)=0\,.$$
\end{Thm}

\begin{Thm}\label{MT2}
Conjecture \ref{HAE} holds for the choices of $\lambda_0,\lambda_1,\lambda_2$ such that
$$\lambda_i \ne \lambda_j \,\, \text{for}\,\,i \ne j\,,\\$$
$$(\lambda_0\lambda_1+\lambda_1\lambda_2+\lambda_2\lambda_0)^2-3\lambda_0\lambda_1\lambda_2(\lambda_0+\lambda_1+\lambda_2)=0\,.$$
\end{Thm}

\subsection{Holomorphic anomaly equations for $K\PP^3$.}\label{holp22} We state the precise form of the holomorphic anomaly equations for local $\PP^3$. Since the study of local $\PP^3$ will be parallel to the study of local $\PP^2$, we will sometime use the same notations for local $\PP^2$ and local $\PP^3$. Since the study of two theories are logically independent in our paper, the indication of each notation will be clear from the context. Denote by $K\PP^3$ the total space of the canonical bundle over $\PP^3$.
 Let $H\in H^2(K\PP^3,\QQ)$ be the hyperplane class obtained from $\PP^3$, and let
 
\begin{align*}
    \mathcal{F}^{\mathsf{GW}}_{g,m}[a,b](Q)=\lan \tau_0(H)^a\tau_0(H^2)^b\ran^{\mathsf{GW}}_{g,m}=\sum^{\infty}_{d=0}Q^d\int_{[\overline{M}_{g,m}(K\PP^3,d)]^{\text{vir}}}\prod^a_{i=1}\text{ev}^*_i(H)\prod^{a+b}_{i=a+1}\text{ev}^*_i(H^2)\,,\\
    \mathcal{F}^{\mathsf{SQ}}_{g,m}[a,b](q)=\lan\tau_0(H)^a\tau_0(H^2)^b \ran^{\mathsf{SQ}}_{g,m}=\sum^{\infty}_{d=0}q^d\int_{[\overline{Q}_{g,m}(K\PP^3,d)]^{\text{vir}}}\prod^a_{i=1}\text{ev}^*_i(H)\prod^{a+b}_{i=a+1}\text{ev}^*_i(H^2)\,.
\end{align*} 
be the Gromov-Witten and stable quotient series respectively with $a+b=m$ (involving the evaluation morphisms at the markings). The relationship between the Gromov-Witten and stable quotient invariants of $K\PP^3$ is proven in \cite{CKg} in case $2g-2+n > 0$:
\begin{align}\label{34567}
    \mathcal{F}^{\mathsf{GW}}_{g,m}[a,b](Q(q))=\mathcal{F}^{\mathsf{SQ}}_{g,m}[a,b](q)\,,
\end{align}
where $Q(q)$ is the mirror map,
$$I_1^{K\PP^3}(q)=\text{log}(q)+4\sum^{\infty}_{d=1}q^d\frac{(4d-1)!}{(d!)^4}\,,$$

$$Q(q)=\text{exp}\left(I_1^{K\PP^3}(q)\right)=q\cdot \text{exp}\left(4\sum^{\infty}_{d=1}q^d\frac{(4d-1)!}{(d!)^4}\right)\,.$$

To state the holomorphic anomaly equations, we need the following additional series in $q$.

\begin{align*}
    L(q)&=(1-4^4q)^{-\frac{1}{4}}=1+64q+10240q^2+\dots\,,\\
    C_1(q)&=q\frac{d}{dq}I_1^{K\PP^3}\,,\\
    A_2(q)&=\frac{q\frac{d}{dq}C_1}{C_1}\,.
\end{align*}
We will define extra series $B_2(q)$, $B_4(q) \in \CC[[q]]$ in \eqref{SB}.
We also need new series $L_i(q)$ defined by roots of following degree $4$ polynomial in $\mathcal{L}$ for $i=0,1,2,3$:

$$(1-4^4q)\mathcal{L}^4-s_1\mathcal{L}^3+s_2\mathcal{L}^2-s_3\mathcal{L}+s_4\,,$$
with initial conditions,
$$L_i(0)=\lambda_i\,.$$
Here, $s_k$ is $k$-th elementary symmetric function in $\lambda_0,\dots,\lambda_3$.
Let $f_3$ be the polynomial of degree $3$ in variable $x$ over $\CC(\lambda_0,\dots,\lambda_3)$ defined by
$$f_3(x):=s_1 x^3-2s_2 x^2+3s_3 x-4 s_4\,.$$
The ring $$\mathds{G}_3:=\CC(\lambda_0,\dots,\lambda_3)[L_0^{\pm 1},\dots,L_3^{\pm 1},f(L_0)^{-\frac{1}{2}},\dots,f(L_3)^{-\frac{1}{2}}]$$ will play a basic role in our paper. Consider the free polynomial rings in the variables $A_2$, $B_2$, $B_4$ and $C_1^{-1}$ over $\mathds{G}_3$,
\begin{align}\label{LOR333}
    \mathds{G}_3[A_2,B_2,B_4,C_1^{\pm 1}]\,.
\end{align}
We have canonical map
\begin{align}\label{LOM2333}
   \mathds{G}_3[A_2,B_2,B_4,C_1^{\pm 1}]\rightarrow \CC(\lambda_0,\dots,\lambda_3)[[q]]
\end{align}
given by assigning the above defined series $A_2(q)$, $B_2(q)$, $B_4(q)$ and $C_1(q)$ to the variables $A_2$, $B_2$, $B_4$ and $C_1$ respectively. Therefore we can consider elements of the rings \eqref{LOR333} either as free polynomials in the variables $A_2$, $B_2$, $B_4$ and $C_1$ or as series in $q$.

Let $F(q)\in \CC(\lambda_0,\dots,\lambda_3)[[q]]$ be a series in $q$. When we write
$$F(q)\in \mathds{G}_3[A_2,B_2,B_4,C_1^{\pm 1}]\,,$$
we mean there is a cononical lift $F\in\mathds{G}_3[A_2,B_2,B_4,C_1^{\pm 1}]$ for which
$$F\rightarrow F(q)\in\CC(\lambda_0,\dots,\lambda_2)[[q]]$$
under the map \eqref{LOM2333}. The symbol $F$ without the argument $q$ is the canonical lift. The notation
$$F(q)\in\mathds{G}_3[A_2,B_2,B_4,C_1^{\pm 1}]$$
is parallel.

Let $T$ be the standard coordinate mirror to $t=\text{log}(q)$,
$$T=I^{K\PP^3}_1(q)\,.$$
Then $Q(q)=\text{exp}(T)$ is the mirror map.
\begin{Conj}\label{ooo333}
 For the stable quotient invariants of $K\PP^3$,
 \begin{itemize}
  \item[(i)] $\mathcal{F}^{\text{SQ}}_{g,a+b}[a,b](q)\in \mathds{G}_3[A_2,B_2,B_4,C_1^{\pm 1}]$ for $g\ge2$,
  \item[(ii)] $\mathcal{F}^{\mathsf{SQ}}_g$ is of degree at most $2(3g-3)$ with respect to $A_2$,
  \item[(iii)] $\frac{\partial^k \mathcal{F}^{\mathsf{SQ}}}{\partial T^k}(q),\in \mathds{G}_3[A_2,B_2,B_4,C_1^{\pm 1}]$ for $g\ge 1$ and $k\ge 1$.
 \end{itemize}
\end{Conj}
\noindent Here, $\mathcal{F}^{\mathsf{SQ}}_g=\mathcal{F}^{\mathsf{SQ}}_{g,0}[0,0]$.

\begin{Conj}\label{HAE333}
The holomorphic anomaly equations for the stable quotient invariants of $K\PP^3$ hold for $g\ge 2$:
\begin{multline*}\frac{L^2}{4 C_1}\frac{\partial \mathcal{F}^{\mathsf{SQ}}_g}{\partial A_2}+\frac{-2s_1 L^4-C_1(3B_2 L^2-s_1L^6)}{4 C_1^2}\frac{\partial \mathcal{F}^{\mathsf{SQ}}_g}{\partial B_4}=\\\sum^{g-1}_{i=1} \mathcal{F}^{\mathsf{SQ}}_{g-i,1}[0,1]\mathcal{F}^{\mathsf{SQ}}_{i,1}[1,0]+\mathcal{F}^{\mathsf{SQ}}_{g-1,1}[1,1]\,,\end{multline*}
$$\frac{2L^4}{C_1^2(L^4-1-2A_2)}\frac{\partial \mathcal{F}^{\mathsf{SQ}}_g}{\partial B_2}=\sum^{g-1}_{i=1} \frac{\partial\mathcal{F}^{\mathsf{SQ}}_{g-i}}{\partial T}\frac{\partial\mathcal{F}^{\mathsf{SQ}}_{i}}{\partial T}+\frac{\partial^2\mathcal{F}^{\mathsf{SQ}}_{g-1}}{\partial T^2}\,.$$
\end{Conj}

The derivative of $\mathcal{F}^{\mathsf{SQ}}_g$ with respect to $A_2$, $B_2$ and $B_4$ in the above equations is well-defined since
$$\mathcal{F}^{\mathsf{SQ}}_g\in \mathds{G}_3[A_2,B_2,B_4,C_1^{\pm 1}]$$
by part (i) of Conjecture \ref{HAE333}.

We will prove the following special cases of the conjectures in Section \ref{hafp}.

\begin{Thm}\label{MT1333}
Conjecture \ref{ooo333} holds for the choices of $\lambda_0,\dots,\lambda_3$ such that
\begin{align*}
    \lambda_i \ne \lambda_j \,\, \text{for}\,\,i \ne j\,,\\
    4s_2^2-s_1 s_3=0\,,\\
    2s_2^3-27 s_1^2s_4=0\,.
\end{align*}
\end{Thm}

\begin{Thm}\label{MT2333}
Conjecture \ref{HAE333} holds for the choices of $\lambda_0,\dots,\lambda_3$ such that
\begin{align*}
\lambda_i \ne \lambda_j \,\, \text{for}\,\,i \ne j\,,\\
    4s_2^2-s_1 s_3=0\,,\\
    2s_2^3-27 s_1^2s_4=0\,.
\end{align*}
\end{Thm}

For Calabi-Yau 3-folds, holomorphic anomaly equations were first discovered in physics (\cite{BCOV}). Also there were many studies to understand holomorphic anomaly equations mathematically (\cite{GJY,LP,LP2,ObPix}). But less is known for higher dimensional Calabi-Yau manifolds in physics. It might be interesting question to find physical arguments for the holomorphic anomaly equation for $K\PP^3$ proposed in our paper.

\subsection{Holomorphic anomaly for equivariant formal quintic invariants}

A particular twisted theory on $\PP^4$ related to the quintic 3-fold was introduced in \cite{LP2}. Let the algebraic torus
$$\mathsf{T}=(\com^*)^5$$
 act with the standard linearization
 on $\PP^4$ with weights
 $\lambda_0,\ldots,\lambda_4$
 on the vector space  $H^0(\PP^4,\mathcal{O}_{\PP^4}(1))$.

Let
\begin{equation}\label{gred}
\mathsf{C} \rightarrow \overline{M}_g(\mathbb{P}^4,d)
\, , \ \ \ 
f:\mathsf{C}\rightarrow \PP^4
\, , \ \ \ 
\mathsf{S}=f^*\mathcal{O}_{\PP^4}(-1) \rightarrow \mathsf{C}\, 
\end{equation}
be the universal curve,
the universal map,
and the universal bundle over
the moduli space of stable maps
--- all
equipped with canonical $\T$-actions.
%and let
%$$\mathsf{L}=\mathsf{S}^\star\, .$$
%The $\mathsf{T}$-equivariant integrals 
We define the {\em formal quintic invariants} by{\footnote{The negative
exponent denotes the dual:
$\mathsf{S}$ is a line bundle and $\mathsf{S}^{-5}=(\mathsf{S}^\star)^{\otimes 5}$.}}
\begin{equation}\label{fredfredfred}
\widetilde{N}_{g,d}^{\mathsf{GW}}= \int_{[\overline{M}_g(\mathbb{P}^4,d)]^{vir}} e(R\pi_*(\mathsf{S}^{-5}))\, , \   %[[\lambda_0,\lambda_1, \lambda_2,\lambda_3,\lambda_4]]
\end{equation}
where $e(R\pi_*(\mathsf{S}^{-5}))$
is the equivariant Euler class defined
{\em after} localization. More precisely, on each $\T$-fixed locus of
$\overline{M}_g(\mathbb{P}^4,d)$, both
$$R^0\pi_*(\mathsf{S}^{-5}) \ \ \ \text{and} \ \ \
R^1\pi_*(\mathsf{S}^{-5})$$
are vector bundles with moving weights, so
$$e(R\pi_*(\mathsf{S}^{-5})) = \frac{c_{\text{top}}(R^0\pi_*(\mathsf{S}^{-5}))}
{c_{\text{top}}(R^1\pi_*(\mathsf{S}^{-5}))}$$
is well-defined.
The integral \eqref{fredfredfred} is
homogeneous of degree 0 in localized
equivariant cohomology
and defines a rational function in $\lambda_i$, 
 $$\widetilde{N}_{g,d}^{\mathsf{GW}}\in {\mathbb{C}}(\lambda_0,\dots,\lambda_4)\,.$$
Let $g\geq 2$.
The associated Gromov-Witten generating series is 
$$\widetilde{\mathcal{F}}^{\mathsf{GW}}_g(Q)\, =\, 
\sum_{d=0}^\infty \widetilde{N}_{g,d}^{\mathsf{GW}} Q^d\, 
\in \, \mathbb{C}[[Q]] \, .$$
Let 
$$I^{\mathsf{Q}}_0(q)=\sum_{d=0}^\infty q^d \frac{(5d)!}{(d!)^5}\, ,  \ \ \ 
I^{\mathsf{Q}}_1(q)= \log(q)I^{\mathsf{Q}}_0(q) + 5 \sum_{d=1}^\infty q^d \frac{(5d)!}{(d!)^5} 
\left( \sum_{r=d+1}^{5d} \frac{1}{r}\right)\, .
$$
We {\em define} the generating series
of stable quotient invariants for
formal quintic theory by the
wall-crossing formula 
for the true quintic theory which has been recently proven by Ciocan-Fontanine and Kim in \cite{CKw},
\begin{equation}\label{jjed}
\widetilde{\mathcal{F}}_{g}^{\mathsf{GW}}(Q(q)) =
I^{\mathsf{Q}}_0(q)^{2g-2} \cdot \widetilde{\mathcal{F}}_{g}^{\mathsf{SQ}}(q)
%=\mathcal{F}_g^B(q)
\end{equation}
with respect to the true quintic mirror map
$$Q(q) = \exp\left(\frac{I^{\mathsf{Q}}_1(q)}
{I^{\mathsf{Q}}_0(q)}\right)\, = \, 
q\cdot \exp\left( \frac{5 \sum_{d=1}^\infty q^d \frac{(5d)!}{(d!)^5} 
\left( \sum_{r=d+1}^{5d} \frac{1}{r}\right)}{\sum_{d=0}^\infty\frac{(5d)!}{(d!)^5}} \right)
\, .$$
Denote the B-model side of \eqref{jjed} by
$$\widetilde{F}^{\mathsf{B}}_g(q)= I^{\mathsf{Q}}_0(q)^{2g-2} \widetilde{\mathcal{F}}^{\mathsf{SQ}}_g(q)\, .$$

%\begin{align*}\mathcal{F}^{\mathsf{SQ}}_g\, =\, 
%\sum_{d=0}^\infty \widetilde{N}_{g,d}^{\mathsf{SQ}} q^d\, 
%=\, \sum_{d=0}^\infty q^d
%\int_{[\overline{Q}_g(\mathbb{P}^4,d)]^{vir}} c_{\rm{top}}(R\pi_*(\mathsf{L}^5))\, \   
%\in\,
%\mathbb{Q}[[q]] .
%\end{align*}
In order to state the holomorphic anomaly equations,
we require several series in $q$. First, let
$$L(q) \, =\,   (1-5^5 q)^{-\frac{1}{5}}\,  = \, 1+625q+117185 q^2 +\ldots\, .$$
Let $\mathsf{D}=q\frac{d}{dq}$, and let
$$C_0(q)= I^{\mathsf{Q}}_0\, , \ \ \ 
C_1(q)= \mathsf{D} \left( \frac{I^{\mathsf{Q}}_1}
{I^{\mathsf{Q}}_0}\right)\, .$$
We define
\begin{eqnarray*}
%\, ,\\
%I_0(q)&=&
%1 +  \sum_{d=1}^\infty q^d \frac{(5d)!}{(d!)^5} \, , \\
%I_1(q)&=&
%I_0 \log (q) + 5 \sum_{d=1}^\infty q^d \frac{(5d)!}{(d!)^5} 
%\left( \sum_{r=d+1}^{5d} \frac{1}{r}\right)\, , \\
%%C_0(q)&=& I_0\,  ,\\
%C_1(q)&=& q \frac{d}{dq} \frac{I_1}{I_0}\, ,
%\\
K_2(q)&=& -\frac{1}{L^5} \frac{\DD C_0}{C_0}\,  ,
\\
A_2(q)&=& \frac{1}{L^5}\left( -\frac{1}{5}\frac{\DD C_1}{C_1}-\frac{2}{5}\frac{\DD C_0}{C_0}-\frac{3}{25}\right)\, ,\\
A_4(q) &=& \frac{1}{L^{10}} \Bigg(-\frac{1}{25}\left(
\frac{\DD C_0}{C_0}\right)^2-\frac{1}{25}
\left(\frac{\DD C_0}{C_0}\right)\left(\frac{\DD C_1}{C_1}\right)
\, \\
& &
+\frac{1}{25}\DD\left(\frac{\DD C_0}{C_0}\right)+\frac{2}{25^2} \Bigg)\,  ,\\
A_6(q) &=&\frac{1}{31250L^{15} }\Bigg( 4+125 \DD\left(\frac{\DD C_0}{C_0}\right)
+50\left(\frac{\DD C_0}{C_0}\right)
\left(1+10 \DD \left(\frac{\DD C_0}{C_0}
\right)\right)\,    \\
& & -5L^5\Bigg(1+10\left(\frac{\DD C_0}{C_0}\right)+25
\left(\frac{\DD C_0}{C_0}\right)^2+25\DD
\left(\frac{q\frac{d}{dq}C_0}{C_0}\right)\Bigg)\,  \\
& &+125\DD^2\left(\frac{\DD C_0}{C_0}\right)-125\left(\frac{\DD C_0}{C_0}\right)^2\Bigg(\left(\frac{\DD C_1}{C_1}\right)-1\Bigg) \Bigg)\, .
\end{eqnarray*}
For the full equivariant formal quintic theory, we need extra series $B_1,B_2,B_3,B_4\in \CC[[q]]$ obtained from $I$-function of quintic. We will give the exact definitions of these extra series in forthcoming paper \cite{HL2}. These series are closely related to the extra generators in \cite{GJY}, where the formal quintic theory was studied in genus 2 case with connections to real quintic theory.

We also need new series $L_i(q)$ defined by roots of following degree $5$ polynomial in $\mathcal{L}$ for $i=0,\dots,4$:

$$(1-5^5q)\mathcal{L}^5-s_1\mathcal{L}^4+s_2\mathcal{L}^3-s_3\mathcal{L}^2+s_4\mathcal{L}-s_5\,,$$
with initial conditions,
$$L_i(0)=\lambda_i\,.$$
Here $s_k$ is $k$-th elementary symmetric function in $\lambda_0,\dots,\lambda_4$.
Let $f_4$ be the polynomial of degree $4$ in variable $x$ over $\CC(\lambda_0,\dots,\lambda_4)$ defined by
$$f_4(x):=s_1x^4-2s_2x^3+3s_3x^2-4s_3x+5s_4\,.$$
The ring $$\mathds{G}_{\mathsf{Q}}:=\CC(\lambda_0,\dots,\lambda_4)[L_0^{\pm 1},\dots,L_4^{\pm 1},f_4(L_0)^{-\frac{1}{2}},\dots,f_4(L_4)^{-\frac{1}{2}}]$$ will play a basic role in formal quintic theory.

Let $T$ be the standard coordinate mirror to $t=\log(q)$,
\begin{align*} T= \frac{I^{\mathsf{Q}}_1(q)}{I^{\mathsf{Q}}_0(q)}\, .
\end{align*}
Then $Q(q)=\exp(T)$ is the mirror map. Let $$\mathds{G}_{\mathsf{Q}}[A_2,A_4,A_6,B_1,B_2,B_3,B_4,C_0^{\pm 1},C_1^{-1},K_2]$$ 
be the free
polynomial ring over $\mathds{G}_{\mathsf{Q}}$.

\begin{Conj} \label{ooo5} For the series 
$\widetilde{\mathcal{F}}_g^{\mathsf{B}}$
associated to the formal quintic, 
\begin{enumerate}
\item[(i)]
$\widetilde{\mathcal{F}}_g^{\mathsf{B}} (q) \in \mathds{G}_{\mathsf{Q}}[A_2,A_4,A_6,B_1,\dots,B_4,C_0^{\pm 1},C_1^{-1},K_2]$
for $g\geq 2$, \vspace{5pt}

\item[(ii)]
$\frac{\partial^k \widetilde{\mathcal{F}}_g^{\mathsf{B}}}{\partial T^k}(q) \in \mathds{G}_{\mathsf{Q}}[A_2,A_4,A_6,B_1,\dots,B_4,C_0^{\pm 1},C_1^{-1},K_2]$ for  $g\geq 1$, $k\geq 1$,
\vspace{5pt}

\item[(iii)]
${\frac{\partial^k \widetilde{\mathcal{F}}_g^{\mathsf{B}}}{\partial T^k}}$ is homogeneous 
with respect  to $C_1^{-1}$ 
of degree $k$.
\end{enumerate}
\end{Conj}

%\begin{Thm} \label{ooo5} For all %$g\geq 1$,
%$$\frac{\partial %%\widetilde{\mathcal{F}}_g^{\maths%f{SQ}}}{\partial T}(q) \in %\mathbb{C}[L^{\pm1}][A_2,A_4,A_6]\%, , $$
%and for all $g\geq 2$,
%$$\widetilde{\mathcal{F}}_g^{\math%sf{SQ}} (q) \in %%\mathbb{C}[L^{\pm1}][A_2,A_4,A_6]%%\, . $$
%\end{Thm}

\begin{Conj} \label{ttt5} The series $\widetilde{\mathcal{F}}^{\mathsf{B}}_g$
associated to the formal quintic satisfy some holomorphic anomaly equations which specialize to 
\begin{multline*}
\frac{1}{C_0^2C_1^2}\frac{\partial \widetilde{\mathcal{F}}_g^{\mathsf{B}}}{\partial{A_2}}-\frac{1}{5C_0^2C_1^2}\frac{\partial \widetilde{\mathcal{F}}_g^{\mathsf{B}}}{\partial{A_4}}K_2+\frac{1}{50C_0^2C_1^2}\frac{\partial \widetilde{\mathcal{F}}_g^{\mathsf{B}}}{\partial{A_6}}K_2^2
= \\
\frac{1}{2}\sum_{i=1}^{g-1} 
\frac{\partial \widetilde{\mathcal{F}}_{g-i}^{\mathsf{B}}}{\partial{T}}
\frac{\partial \widetilde{\mathcal{F}}_i^{\mathsf{B}}}{\partial{T}}
+
\frac{1}{2}\,
\frac{\partial^2 \widetilde{\mathcal{F}}_{g-1}^{\mathsf{B}}}{\partial{T}^2}\,,
\end{multline*}
\begin{align*}
\frac{\partial \widetilde{\mathcal{F}}_g^{\mathsf{B}}}{\partial K_2}=0\,,
\end{align*}
with the choice of $\lambda_i=\zeta_5^i$.
\end{Conj}

\noindent 
We expect holomorphic anomaly equations in Conjecture \ref{ttt5} to hold in the ring
\begin{align}\label{FQR}\mathds{G}_{\mathsf{Q}}[A_2,A_4,A_6,B_1,B_2,B_3,B_4,C_0^{\pm1},C_1^{-1},K_2]\, .\end{align}
\begin{Rmk} If we specialize $\lambda_i$ to (the power of) primitive fifth root of unity $\zeta_5$,
\begin{align}\label{spfq}\lambda_i=\zeta_5^i\,,\end{align}
the expected equations in Conjecture \ref{ttt5} exactly matches{\footnote{Our functions
$K_2$ and $A_{2k}$ 
 are normalized differently with respect to $C_0$ and $C_1$.
The dictionary to exactly match the notation of \cite[(2.52)]{ASYZ} is to 
multiply our $K_2$ by $(C_0C_1)^2$ and our $A_{2k}$ by $(C_0C_1)^{2k}$.}}
the
conjectural holomorphic anomaly equation
 \cite[(2.52)]{ASYZ} for
the true quintic theory and this was the main result in \cite{LP2}. Also the ring \eqref{FQR} can be reduced to the Yamaguchi-Yau ring introduced in \cite{YY} for the true quintic theory only with the choice of specialization \eqref{spfq}. This explains why the specialization \eqref{spfq} used in \cite{LP2} is the natural choice.

\end{Rmk}

\begin{Thm}\label{tttt}
Conjecture \ref{ooo5} holds for the choices of $\lambda_0,\dots,\lambda_4$ such that
\begin{align*}
      \lambda_i \ne \lambda_j\,\,\,\text{for}\,\,i \ne j\,,\\
    2 s_1 s_3=s_2^2\,,\\
    8 s_1^2 s_4=s_2^3\,,\\
    80 s_1^3 s_5=s_2^4\,.\,
\end{align*}
where $s_k$ is $k$-th elementary symmetric function in $\lambda_0,\dots,\lambda_4$.
\end{Thm}
\noindent Theorems \ref{tttt} will be proven and the precise form of holomorphic anomaly equations in Conjecture \ref{ttt5} will appear in \cite{HL2}.

\subsection{Acknowledgments}
I thank S. Guo, H. Iritani, F. Janda, B. Kim, A. Klemm, M. C.-C. Liu, G. Oberdieck, Y. Ruan and E. Scheidegger for disccusions over the years about the invariants of Calabi-Yau geometries and holomorphic anomaly equation.
I am especially grateful to  R. Pandharipande for the suggestion of this project and useful discussions. I thank F. Janda for correcting the statements of Conjecture \ref{ooo5} and Theorem \ref{tttt}. I was supported by the grant ERC-2012-AdG-320368-MCSK.

\section{Localization Graph}

\label{locq}
\subsection{Torus action}
Let $\mathsf{T}=(\com^*)^{n+1}$ act diagonally on the vector space $\mathbb{C}^{n+1}$
with weights
$$-\lambda_0, \ldots, -\lambda_n\, .$$
Denote the $\mathsf{T}$-fixed points of 
the induced $\mathsf{T}$-action on $\proj^n$ by
$$p_0, \ldots, p_{n}\, . $$
The weights of $\mathsf{T}$ on the tangent space $T_{p_j}(\proj^n)$ are
$$\lambda_j-\lambda_0, \ldots, \widehat{\lambda_j-\lambda_j}  ,\ldots, \lambda_j-\lambda_{n}\, .$$

There is an induced $\mathsf{T}$-action on 
the moduli space 
%of quasimaps{\footnote{We will also use the terminology
%of quasimaps for stable quotients since we require  results from \cite{CKg}.
%In fact, all the
%quasimap moduli spaces considered here will be 
%moduli spaces of stable quotients on projective spaces or local projective %spaces.}} 
$\overline{Q}_{g,k}(\proj^n,d)$.
The localization formula of \cite{GP} applied to the  virtual fundamental class 
$[\overline{Q}_{g,k}(\proj^n,d)]^{vir}$ will play a fundamental role our paper.
The $\mathsf{T}$-fixed loci are represented in terms of dual graphs,
and the contributions of the $\mathsf{T}$-fixed loci are given by
tautological classes. The formulas here are
standard, see \cite{KL,MOP}.

\subsection{Graphs}\label{grgr}
Let the genus $g$ and the number of markings $k$ for the moduli
space be in
the stable range
\begin{equation}\label{dmdm}
2g-2+k>0\, .
\end{equation}
We can organize the $\mathsf{T}$-fixed loci 
of $\overline{Q}_{g,k}(\proj^n,d)$
according to decorated graphs.
A {\em decorated graph} $\Gamma \in \mathsf{G}_{g,k}(\proj^n)$ consists 
of the data $(\mathsf{V}, \mathsf{E}, 
\mathsf{N}, \mathsf{g}, \mathsf{p} )$ where
\begin{enumerate}
 \item[(i)] $\mathsf{V}$ is the vertex set, 
 \item[(ii)] $\mathsf{E}$ is the edge set (including
 possible self-edges),
 \item[(iii)] $\mathsf{N} : \{1,2,..., k\} \rightarrow \mathsf{V}$ is the marking
 assignment,
   \item[(iv)] $\mathsf{g}: \mathsf{V} \rightarrow \ZZ_{\geq 0}$ is a genus
 assignment satisfying
 $$g=\sum_{v \in V} \mathsf{g}(v)+h^1(\Gamma)\, $$
and for which $(\mathsf{V},\mathsf{E},\mathsf{N},\mathsf{g})$ is stable graph{\footnote
{Corresponding to a stratum of the moduli space
of stable curves $\overline{M}_{g,n}$.}}, 
 \item[(v)] $\mathsf{p} : \mathsf{V} \rightarrow ({\PP ^n})^{\mathsf{T}}$ is an assignment of a $\mathsf{T}$-fixed point $\mathsf{p} (v)$ to each vertex $v \in \mathsf{V}$.
\end{enumerate}
The markings $\mathsf{L}=\{1,\ldots,k\}$ are often called {\em legs}.

To each decorated graph $\Gamma\in \mathsf{G}_{g,k}(\proj^n)$, we associate the set of fixed loci of  
$$\sum_{d\geq 0} \left[\overline{Q}_{g, k} (\proj^n, d)\right]^{\vir} q^d$$
with elements described as follows:
\begin{enumerate}
 
 \item[(a)] If $\{v_{i_1},\ldots,v_{i_j}\}=\{v\, |\, \mathsf{p}(v)=p_i\}$, then $f^{-1}(p_i)$ is a disjoint union of connected stable curves of genera $\mathsf{g}(v_{i_1}),\ldots, \mathsf{g}(v_{i_j})$ and finitely many points.
 
 %\item[(a)] If $\{v_{i_1},\ldots,v_{i_k}\}=\{v\, |\, \mathsf{p}(v)=p_i\}$, then $f^{-1}(p_i)$ is a disjoint union of connected stable curves of genera $\mathsf{g}(v_{i_1}),\ldots, \mathsf{g}(v_{i_k})$.
 
 \item[(b)] There is a bijective
  correspondence between the connected components of $C \setminus D$ and the set of edges and legs of $\Gamma$ respecting  vertex incidence where $C$ is domain curve and $D$ is union of all subcurves of $C$ which appear in (a). 
  
  %\item[(b)] There is a bijective correspondence between the connected components of $f^{-1}(\PP^m/\{p_0, \ldots, p_m\})$ and the edges of $\Gamma$ respecting  vertex incidence. 
\end{enumerate}
We write the localization formula as
$$\sum_{d\geq 0} \left[\overline{Q}_{g, k} (\PP^n, d)\right]^{\vir} q^d =
\sum_{\Gamma\in \mathsf{G}_{g,k}(\PP^n)} \text{Cont}_\Gamma\, .$$
While $\mathsf{G}_{g,k}(\proj^n)$ is a finite set,
each contribution $\text{Cont}_\Gamma$ is
a series in $q$ obtained from
an infinite sum over all edge possibilities (b).

\subsection{Unstable graphs}
The moduli spaces of stable quotients $$\overline{Q}_{0,2}(\proj^n,d) \ \ \
\text{and} \ \ \ \overline{Q}_{1,0}(\proj^n,d)$$
for $d>0$
are the only{\footnote{The moduli spaces
$\overline{Q}_{0,0}(\proj^n,d)$ and
$\overline{Q}_{0,1}(\proj^n,d)$
are empty by the definition of a stable
quotient.}}
cases where
the pair $(g,k)$ does 
{\em not} satisfy the Deligne-Mumford stability condition 
\eqref{dmdm}. 

An appropriate set of decorated graphs $\mathsf{G}_{0,2}(\PP^n)$
 is easily defined: The graphs
$\Gamma \in \mathsf{G}_{0,2}(\PP^n)$ all have 2 vertices
connected by a single edge. Each vertex carries a marking.
All  of the  conditions (i)-(v)
of Section \ref{grgr} are satisfied
except for the stability of $(\mathsf{V},\mathsf{E}, \mathsf{N},\gamma)$.
%and
%$\mathsf{G}_{1,0}(\PP^m)$ 
The  localization formula holds,
\begin{eqnarray}\label{ddgg}
\sum_{d\geq 1} \left[\overline{Q}_{0, 2} (\PP^n, d)\right]^{\vir} q^d &=&
\sum_{\Gamma\in \mathsf{G}_{0,2}(\proj^n)} \text{Cont}_\Gamma\,,
\end{eqnarray}
For $\overline{Q}_{1,0}(\proj^n,d)$, the matter
is more problematic --- usually a marking
is introduced to break the symmetry.
%\\ \nonumber
%\sum_{d\geq 1} \left[\overline{Q}_{1, 0} (\PP^m, %d)\right]^{\vir} q^d &=&
%\sum_{\Gamma\in \mathsf{G}_{1,0}(\proj^m)} \text{Cont}_\Gamma
%\end{eqnarray}
%are valid.

\section{Basic correlators}\label{bcbcbc}
\subsection{Overview}
We review here basic generating series in $q$ which 
arise in  the genus 0 theory of quasimap invariants. The series
will play a fundamental role in
the calculations of Sections \ref{hgi} - \ref{hafp}
related
to the holomorphic anomaly equation for $K\proj^2$.

We fix a torus action $\mathsf{T}=(\CC^*)^3$ on $\PP^2$ with
weights{\footnote{The associated weights on
$H^0(\PP^2,\mathcal{O}_{\PP^2}(1))$ are
$\lambda_0,\lambda_1,\lambda_2$
and so match the conventions of
Section \ref{twth}.}}
$$-\lambda_0, -\lambda_1, -\lambda_2$$
on the vector space $\mathbb{C}^3$.
The $\T$-weight on the fiber over
$p_i$ of the canonical
bundle 
\begin{equation}\label{pqq9}
\mathcal{O}_{\PP^2}(-3) \rightarrow \PP^2
\end{equation}
is $-3\lambda_i$.
The toric Calabi-Yau $K\PP^2$
is the total space of \eqref{pqq9}.

\subsection{First correlators}
We require several correlators defined 
via  the Euler class of the obstruction bundle,
$$e(\text{Obs})= e(R^1\pi_* \mathsf{S}^3)\, ,$$
associated to the $K\PP^2$ geometry
on the moduli space $\overline{Q}_{g,n}(\PP^2,d)$.
The first two are obtained from
standard stable quotient invariants.
For $\gamma_i \in H^*_{\T} (\PP^2)$, let
\begin{eqnarray*}
    \Big \langle \gamma_1\psi^{a_1}, ...,\gamma_n\psi^{a_n} \Big\rangle_{g,n,d}^{\mathsf{SQ}}&=& 
    \int_{[\overline{Q}_{g,n}(\PP^2,d)]^{\vir}} 
    e(\text{Obs})\cdot
    \prod_{i=1}^n \text{ev}_i^*(\gamma_i)\psi_i^{a_i},\\
      \Big\langle \Big\langle \gamma _1\psi  ^{a_1} , ..., \gamma _n\psi  ^{a_n} \Big\rangle \Big\rangle _{0, n}^{\mathsf{SQ}} 
    &=& \sum _{d\geq 0}\sum_{k\geq 0} \frac{q^{d}}{k!}
 \Big\lan    \gamma _1\psi  ^{a_1} , ..., \gamma _n\psi  ^{a_n} , t, ..., t  \Big\ran_{0, n+k, d}^{\mathsf{SQ}} , 
 \end{eqnarray*}
 where, in the second series,
 $t \in H_{\T}^* (\PP^2)$.
 We will systematically use the quasimap notation $0+$
for stable quotients,
\begin{eqnarray*}
    \Big \langle \gamma_1\psi^{a_1}, ...,\gamma_n\psi^{a_n} \Big\rangle_{g,n,d}^{0+}&=&
    \Big \langle \gamma_1\psi^{a_1}, ...,\gamma_n\psi^{a_n} \Big\rangle_{g,n,d}^{\mathsf{SQ}} \\
      \Big\langle \Big\langle \gamma _1\psi  ^{a_1} , ..., \gamma _n\psi  ^{a_n} \Big\rangle \Big\rangle _{0, n}^{0+} 
    &=& 
    \Big\langle \Big\langle \gamma _1\psi  ^{a_1} , ..., \gamma _n\psi  ^{a_n} \Big\rangle \Big\rangle _{0, n}^{\mathsf{SQ}}\, . 
 \end{eqnarray*}

\subsection{Light markings}\label{lightm}
Moduli of quasimaps can be considered with $n$ ordinary (weight 1) markings and $k$ light 
(weight $\epsilon$) markings{\footnote{See Sections 2 and 5 of \cite{BigI}.}},
$$\overline{Q}^{0+,0+}_{g,n|k}(\PP^2,d)\, .$$
Let $\gamma_i \in H^*_{\T} (\PP^2)$ be
equivariant cohomology classes, and
%\ot \QQ (\lambda )$, and
let
$$\delta _j \in H^*_{\T} ([\mathbb{C}^3/\com^* ])$$ 
be classes on the stack quotient. 
Following the notation of \cite{KL}, 
we define series for the $K\PP^2$
geometry,
%For $\gamma_i \in H^*_{\T} (Y)\ot \QQ (\lambda )$, $\tilde{t}, \delta _j \in H^*_{\T} ([V/\G ], \QQ )$ denote

\begin{multline*}
    \Big\lan \gamma _1\psi  ^{a_1} , \ldots, \gamma _n\psi  ^{a_n} ;  \delta _1, \ldots, \delta _k \Big\ran _{g, n|k, d}^{0+, 0+}  = \\
\int _{[\overline{Q}^{0+, 0+}_{g, n|k} (\PP^2, d)]^{\vir}} 
e(\text{Obs})\cdot
\prod _{i=1}^n \text{ev}_i^*(\gamma _i)\psi _i ^{a_i} 
\cdot \prod _{j=1}^k \widehat{\text{ev}}_j ^* (\delta _j)\, , 
\end{multline*}
\begin{multline*}
\Big \langle \Big\langle \gamma _1\psi  ^{a_1} , \ldots, \gamma _n\psi  ^{a_n} \Big\rangle\Big\rangle _{0, n}^{0+, 0+} 
= \\ \sum _{d\geq 0}\sum_{k\geq 0} \frac{q^{d}}{k!}
 \Big\lan    \gamma _1\psi  ^{a_1} , \ldots, \gamma _n\psi  ^{a_n} ; {t}, \ldots, {t}  
 \Big\ran_{0, n|k, d}^{0+, 0+} \, ,
 \end{multline*}
 where, in the second series,
 ${t} \in H_{\T}^* ([\mathbb{C}^3/\com^* ])$.
 
 For each $\T$-fixed point $p_i\in \PP^2$, let 
 $$e_i= e(T_{p_i}(\PP^2))\cdot(-3\lambda_i)$$
 be the equivariant Euler class of
 the tangent space of $K\PP^2$ at $p_i$. Let
 $$ \phi_i=\frac{-3\lambda_i \prod_{j \ne i}(H-\lambda_j)}{e_i}, \ \ \phi^i=e_i \phi_i\ \ \in H^*_{\T}(\PP^2)\, $$ be cycle classes. Crucial for us are the series
 \begin{align*}
\mathds{S}_i(\gamma) & = e_i \Big\langle \Big\langle  \frac{\phi _i}{z-\psi} , \gamma 
\Big\rangle \Big\rangle _{0, 2}^{0+,0+}\ \ ,  \\
\mathds{V}_{ij}  & =  
\Big\langle \Big\langle  \frac{\phi _i}{x- \psi } ,  \frac{\phi _j}{y - \psi } 
\Big\rangle \Big\rangle  _{0, 2}^{0+,0+}  \ \ . 
 \end{align*}
Unstable degree 0 terms are included by hand in the
above formulas. For $\mathds{S}_i(\gamma)$, the unstable degree 0 term is
$\gamma|_{p_i}$. For $\mathds{V}_{ij}$, the unstable degree 0 term is
$\frac{\delta_{ij}}{e_i(x+y)}$.

 We also write $$\mathds{S}(\gamma)=\sum_{i=0}^2 {\phi_i} \mathds{S}_i(\gamma)\, .$$ 
The series $\mathds{S}_i$ and $\mathds{V}_{ij}$
 satisfy the basic relation
\begin{equation}  \label{wdvv} e_i\mathds{V}_{ij} (x, y)e_j   = 
\frac{\sum _{k=0}^2 \mathds{S}_i (\phi_k)|_{z=x} \, \mathds{S}_j(\phi ^k )|_{z=y}}{x+ y}\,   \end{equation}
 proven{\footnote{In Gromov-Witten
 theory, a parallel relation is
 obtained immediately from the
 WDDV equation and the string equation.
 Since the map forgetting a point
 is not always well-defined for
 quasimaps, a different argument 
 is needed here \cite{CKg}}} in \cite{CKg}.
 
 Associated to each $\T$-fixed point $p_i\in \PP^2$,
 there is a special $\T$-fixed point locus, 
\begin{equation}\label{ppqq}
\overline{Q}^{0+, 0+}_{0, k|m} (\PP^2,d) ^{\T, p_i} \subset
\overline{Q}^{0+, 0+}_{0, k|m}(\PP^2, d)\, ,
\end{equation}
where all markings lie on a single connected
genus 0 domain component contracted to $p_i$.
Let $\text{Nor}$ denote the equivariant
normal bundle 
of $Q^{0+, 0+}_{0, n|k} (\PP^2,d) ^{\T, p_i}$
with respect to the embedding \eqref{ppqq}.
Define 
\begin{multline*}
\Big\lan \gamma _1\psi  ^{a_1} , \ldots, \gamma _n\psi  ^{a_n} ;  \delta _1, ..., \delta _k \Big\ran _{0, n|k, d}^{0+, 0+, p_i}  
=\\
\int _{[\overline{Q}^{0+, 0+}_{0, n|k} (\PP^2, d) ^{\T, p_i}]} 
\frac{e(\text{Obs})}{e(\text{Nor})}\cdot
\prod _{i=1}^n \text{ev}_i^*(\gamma _i)\psi _i ^{a_i} \cdot
\prod _{j=1}^k \widehat{\text{ev}}_j ^* (\delta _j) \, ,
\end{multline*}

\begin{multline*}
  \Big\langle \Big\langle
  \gamma _1\psi  ^{a_1} ,\ldots, \gamma _n\psi  ^{a_n} 
  \Big\rangle \Big\rangle  _{0, n}^{0+, 0+, p_i} =\\
 \sum _{d\geq 0}\sum_{k\geq 0} \frac{q^{d}}{k!}
 \Big\lan    \gamma _1\psi  ^{a_1} , \ldots, \gamma _n\psi  ^{a_n} ; {t}, \ldots, {t}  \Big \ran_{0, n|k, \beta}^{0+, 0+, p_i} \, .
\end{multline*}

\subsection{Graph spaces and I-functions}
\subsubsection{Graph spaces}
The {big I-function} is defined in \cite{BigI}
via the geometry of weighted quasimap graph spaces. 
We briefly summarize the constructions of \cite{BigI}
in the special case of 
%
%(In \cite{BigI}, they defined $(\theta,\epsilon)$-stable quasimap graph space. We summarize the results only for 
%
$(0+,0+)$-stability. The more general weightings
discussed in \cite{BigI} will not be
needed here.

As in Section \ref{lightm}, we consider the quotient
$$\com^3/\com^*$$
associated to $\PP^2$.
Following \cite{BigI},
there is a $(0+,0+)$-{\em stable quasimap graph space}
 \begin{equation}\label{xmmx}
     \mathsf{QG}_{g, n|k, d }^{0+,0+} ([\com^3/\com^*] ) \, .
     %=  
 %\mathsf{Q}_{g, m|k} ^{0+,0+} ([\CC^3\times\CC^2/ %\CC^* \times\CC^*], (d , 1)). 
 \end{equation}
A $\CC$-point of the graph space is described by data 
$$((C, {\bf x}, {\bf y}), (f,\varphi):C\lra [\CC^3/\CC^*]\times [\CC^2/\CC^*]).$$ 
By the definition of stability, $\varphi$ is a regular map to $$\PP^1=\CC^2/\!\!/\CC^*\, $$ of class $1$.
Hence, the domain curve $C$ has a distinguished irreducible component $C_0$ canonically isomorphic to $\PP ^1$ via $\varphi$. 
The {\em standard} $\CC ^*$-action, 
\begin{equation}\label{tt44}
t\cdot [\xi _0, \xi _1] = [t\xi _0, \xi _1], \, \, \text{ for } t\in \CC ^*, \, [\xi _0, \xi _1]\in \PP ^1,
\end{equation}
induces  a $\CC ^*$-action on the graph space.

The $\CC^*$-equivariant cohomology of a point is
a free algebra with generator $z$,
$$H^*_{\CC ^*} (\Spec (\CC )) = \QQ [z]\, .$$
Our convention is to define
$z$ as the $\CC^*$-equivariant first Chern class of the tangent line $T_0\PP ^1$ at $0\in\PP^1$ with respect to the
action \eqref{tt44},
$$z=c_1(T_0\PP ^1)\, .$$

The $\T$-action on $\com^3$ lifts to a $\T$-action on
the graph space \eqref{xmmx} which commutes with the
$\CC^*$-action obtained from the distinguished domain component.
As a result, we have a $\T\times \CC^*$-action on the graph space
and
 $\T\times\CC ^*$-equivariant evaluation morphisms
%\begin{align}\nonumber &\hat{\widetilde{ev}}_j :  QG_{g, m|k, \beta }^{\theta, %\ke} ([W/\G] ) \ra       [W/\G]\times \PP^1 , & j=1,\dots,k, \\ 
%        \nonumber     & \widetilde{ev} _i : QG_{g, m|k, \beta }^{\theta, \ke} ([W/\G] ) \ra       \WmodtG\times \PP^1 ,  & i=1,\dots,m, \end{align}   
%and
\begin{align}
\nonumber     &\text{ev}_i: \mathsf{QG}_{g, n|k, \beta }^{0+, 0+} ([\CC^3/\CC^*] ) \ra       \PP^2 ,  & i=1,\dots,n\, ,\\
\nonumber &\widehat{\text{ev}}_j:  \mathsf{QG}_{g, n|k, \beta }^{0+,0+} 
([\CC^3/\CC^*] ) \ra       [\CC^3/\CC^*] , & j=1,\dots,k\, .
\end{align}
Since a morphism $$f: C \ra [\CC^3/\CC^*]$$ 
is equivalent to the data of a 
principal $\G$-bundle $P$ on $C$ and a section $u$ of $P\times _{\CC^*} 
\CC^3$, 
there is a natural morphism $$C\ra E\CC^* \times _{\CC^*} \CC^3$$ and hence a pull-back map
 \[ f^*:  H^*_{\CC^*}([\CC^3/\CC^*])  \ra H^*(C). \] 
 The above construction applied
 to the universal curve over the moduli space 
 and the universal morphism to $[\CC^3/\CC^*]$ is $\T$-equivariant.
 Hence,
 %The evaluation maps are the compositions of the universal morphism with the %sections of the
%universal curve giving the markings and are $\T\times\CC ^*$-equivariant.
 we obtain a pull-back map 
 \[ \widehat{\text{ev}}_j^*: H^*_{\T}(\CC^3, \QQ)\otimes_\QQ \QQ[z] \ra H^*_{\T\times \CC^*} (\mathsf{QG}_{g, n|k, \beta }^{0+, 0+} 
 ([\CC^3/\CC^*] ) , \QQ ) \]  
 associated to the evaluation map $\widehat{\text{ev}}_j$.
 
% We identify as usual $H^*_{\T} ([W/\G], \QQ):= H^*_{\G\times \T}(W, \QQ) $.

%Now fix $(\theta,\ke)$ (including the cases $\theta=(0+)\cdot\theta_0$ and %$\ke=0+$) and consider the graph spaces 
%$QG_{0, 0|k, \beta }^{\theta, \ke} ([W/\G] )$. 

\subsubsection{{\em{I}}-functions}
The description of the fixed loci for the $\CC^*$-action 
on $$\mathsf{QG}_{g, 0|k, d }^{0+,0+} 
([\CC^3/\CC^*] )$$
is parallel to the description in \cite[\S4.1]{CKg0} for the unweighted case. 
In particular, there is a distinguished
subset $\F_{k,d}$ of the $\CC ^*$-fixed locus for which all the markings and the entire curve class $d$ lie  over $0 \in \PP ^1$. The locus
$\F_{k,d}$ comes with
a natural {\it proper} evaluation map $ev_{\bullet}$ obtained
from the generic point of $\PP ^1$:
\[ \text{ev}_\bullet:  \F_{k,d} \ra \CC^3/\!\!/\CC^* =\PP^2 .  \]
%When $k\ke+\beta(L_\theta) >1$, we have the identification
%\begin{equation*}F_{k,\beta}\cong Q_{0, 1|k}^{\theta, \ke} ([W/\G], \beta ),
%\end{equation*}
%with $\eb=ev_1$, the evaluation map at the weight $1$ marking.

We can explicitly write
\begin{equation*}\F_{k,d}\cong \F_d\times 0^k\subset \F_d\times (\PP^1)^k,
\end{equation*}
where $\F_d$ is the $\CC^*$-fixed locus in $\mathsf{QG}^{0+}_{0,0,d}([\CC^3/\CC^*])$ for which the class $d$ is concentrated over $0\in\PP^1$.  The locus $\F_d$ parameterizes
quasimaps of class $d$,
$$f:\PP^1\lra [\CC^3/\CC^*]\, ,$$ with a base-point of 
length $d$ at $0\in\PP^1$. The restriction of $f$ to $\PP^1\setminus\{0\}$ is a constant map to $\PP^2$ defining the evaluation
map $\text{ev}_\bullet$.

As in \cite{CK, CKg0,CKM}, we define the big $\mathds{I}$-function as the generating function for
the push-forward via $ev_\bullet$ of localization residue contributions of $\F_{k,d}$.
For ${\bf t}\in  
H^*_{\T} ([\CC^3/\CC^*], \QQ )\ot _{\QQ} \QQ[z]$, let
 \begin{align*} \mathrm{Res}_{\F_{k,d}}({\bf t}^k) &=
 \prod_{j=1}^k \widehat{\text{ev}}_j^*({\bf t})\, \cap\, \mathrm{Res}_{\F_{k,d}}[
 \mathsf{QG}_{g, 0|k, d }^{0+,0+} 
([\CC^3/\CC^*])
 ]^{\mathrm{vir}} \\
 &=\frac{\prod_{j=1}^k \widehat{\text{ev}}_j^*({\bf t})
 \cap [\F_{k,d}]^{\mathrm{vir}}}
 {\mathrm{e}(\text{Nor}^{\mathrm{vir}}_{\F_{k,d}})},
 \end{align*}
 where 
%$\iota _\beta : F_\beta \hookrightarrow \QGraphok$ is the inclusion, 
$\text{Nor}^{\mathrm{vir}}_{\F_{k,d}}$ is the virtual normal bundle.
%and $\mathrm{e}^{\CC^*}$ denotes the equivariant %Euler class.

\begin{Def}\label{Je}
 The big $\mathds{I}$-function for the $(0+,0+)$-stability condition,
 as a formal function in $\bf t$,
 is
\begin{equation*}
%\label{Je}
%\dsJ^{\theta, \ke } 
\mathds{I}
(q,{\bf t}, z)=\sum_{d\geq 0}\sum_{k\geq 0} \frac{q^d}{k!}
\text{\em ev}_{\bullet\, *}\Big(\mathrm{Res}_{\F_{k,d}}({\bf t}^k)
\Big)\, .
\end{equation*} 
\end{Def}

\subsubsection{Evaluations}
%Denote by $\mathds{I}$, the big $\mathds{J}^{\theta,\epsilon}$-function for the $(0+,0+)$-stability defined above.

%begin{align*}
 %   \mathds{I}:= \mathds{J}^{0+,0+}.
%end{align*}

Let $\widetilde{H}\in H^*_\T([\CC^3/\CC^*])$ and $H\in H^*_\T(\PP^2)$
denote the respective hyperplane classes. The $\mathds{I}$-function
of Definition \ref{Je} is evaluated in \cite{BigI}.

\begin{Prop} For ${\bf t}=t\widetilde{H} \in H^*_{\T} ([\CC^3/\CC^*], \QQ)$,
 \begin{align}\label{I_Hyper_P} 
\dsI({t}) = \sum _{d=0}^{\infty} q^d e^{t(H+dz)/z} \frac{ \prod _{k=0}^{3d-1}  (-3H - kz)}{\prod^2_{i=0}\prod _{k=1}^d (H-\lambda_i+kz)} . \end{align}

\end{Prop}

Observe that the $\mathds{I}$-function has following expandsion after restriction $t=0$,

$$\mathds{I}|_{t=0}=1+\frac{I_{1}H}{z}+\frac{I_{2,0}H^2+I_{2,1}(\lambda_0+\lambda_1+\lambda_2)H}{z^2}+\mathcal{O}(\frac{1}{z^3})\,,$$
where
\begin{align*}
    I_1(q)&=\sum_{d=1}^{\infty} 3 \frac{(3d-1)!}{(d!)^3}(-q)^d\,,\\
    I_{2,0}(q)&=\sum_{d=1}^{\infty} 3 \frac{(3d-1)!}{(d!)^3}\Big(3\text{Har}[3d-1]-3\text{Har}[d]\Big)(-q)^d\,,\\
    I_{2,1}(q)&=\sum_{d=1}^{\infty} 3 \frac{(3d-1)!}{(d!)^3}\text{Har}[d](-q)^d\,.
\end{align*}
Here $\text{Har}[d]:=\sum_{k=1}^d\frac{1}{k}$.

We return now to the functions  $\mathds{S}_i(\gamma)$
defined in Section \ref{lightm}.
Using Birkhoff factorization, an evaluation of
the series $\mathds{S}(H^j)$ can be obtained from the $\dsI$-function, see \cite{KL}:
\begin{align}
\nonumber\mathds{S}({1}) & = \mathds{I} \, , \\
\label{S1}\mathds{S}(H) & = \frac{  z\frac{d}{dt} \mathds{S}({1})}{  z\frac{d}{dt} \mathds{S}({1})|_{t=0,H=1,z=\infty}} \, , \\
\nonumber \mathds{S}(H^2) & = \frac{ z\frac{d}{dt} \mathds{S}(H)-(\lambda_0+\lambda_1+\lambda_2)N_2\mathds{S}(H)}{ \Big(z\frac{d}{dt} \mathds{S}(H)-(\lambda_0+\lambda_1+\lambda_2)N_2\mathds{S}(H)\Big)|_{t=0,H=1,z=\infty}}\, .
\end{align}
For a  series $F\in \CC[[\frac{1}{z}]]$, the specialization
$F|_{z=\infty}$ denotes constant term of $F$ with respect to $\frac{1}{z}$. Here, $N_2$ is series in $q$ defined by
$$N_2(q):=d\frac{q}{dq}(\frac{q\frac{d}{dq}I_{2,1}}{1+q\frac{d}{dq}I_{1,0}})\,.$$

\subsubsection{Further calculations}\label{furcalc}
Define small $I$-function 
%$$I(q)_\T
$$\overline{\mathds{I}}(q)
\in H^*_{\T}(\PP^2,\QQ)[[q]]$$ by the restriction
\begin{align*}
    \overline{\mathds{I}}(q)
%    I(q)_{\T}
    =\mathds{I}(q,{t})|_{t=0}\, .
\end{align*}
Define differential operators
$$\DD = q\frac{d}{dq}\, , \ \ \ M = H+ z \DD.$$
%Define $t$ by
%$$t:=e^{q}$$.
Applying $z\frac{d}{dt}$ to $\mathds{I}$ and then restricting
to $t=0$ has same effect as applying $M$ to 
$\overline{\mathds{I}}$
%$I_T$,
 \begin{align*}
     \left[\left(z\frac{d}{dt}\right)^k \mathds{I}\right]\Big|_{t=0} = M^k \, 
     \overline{\mathds{I}}\, .
     %I_T\, .
 \end{align*}
The function 
$\overline{\mathds{I}}$
%$I_{\T}$ 
satisfies following Picard-Fuchs equation
\begin{align}
\label{PF}\Big((M-\lambda_0)(M-\lambda_1)(M-\lambda_2)+3qM(3M+z)(3M+2z)\Big) 
\overline{\mathds{I}}=0
%I_\T=0
\end{align}
implied by the Picard-Fuchs equation for $\mathds{I}$,
\begin{multline*}
    \left(\prod_{j=0}^2\left(z\frac{d}{dt}-\lambda_j\right)+3q\left(z\frac{d}{dt}\right)\left(3\left(z\frac{d}{dt}\right)+z\right)
    \left( 3\left( z\frac{d}{dt}\right)+2z\right)\right)\mathds{I}=0\, .
\end{multline*}

The restriction
%$I_{\T}|_{H=\lambda_i}$ 
$\overline{\mathds{I}}|_{H=\lambda_i}$
admits following asymptotic form
\begin{align}
 \label{assym}
 %I_{\T}|_{H=\lambda_i}
 \overline{\mathds{I}}|_{H=\lambda_i}
 = e^{\mu_i/z}\left( R_{0,i}+R_{1,i} z+R_{2,i} z^2+\ldots\right)
\end{align}
with series 
$\mu_i,R_{k,i} \in \CC(\lambda_0,\lambda_1,\lambda_2)[[q]]$.

A derivation of \eqref{assym} is obtained in \cite{ZaZi} via  
the Picard-Fuchs equation \eqref{PF} for
%$I_{\T}|_{H=\lambda_i}$,  
$\overline{\mathds{I}}|_{H=\lambda_i}$.
The series
$\mu_i$ and 
$R_{k,i}$ are found by solving differential equations obtained from the coefficient of $z^k$. 
For example, 
\begin{eqnarray*}
    \lambda_i+ \DD\mu_i&=& L_i\, , \\
    R_{0,i}&=&\Big(\frac{\lambda_i\prod_{j \ne i}(\lambda_i-\lambda_j)}{f(L_i)}\Big)^{\frac{1}{2}}\, .
\end{eqnarray*}

Define the series $C_1$ and $C_2$ by the equations
\begin{align}
C_1 & = z\frac{d}{dt} \mathds{S}({1})|_{z=\infty,t=0,H=1}\, ,  \label{y999}\\
C_2 & =\Big( z\frac{d}{dt} \mathds{S}(H)-(\lambda_0+\lambda_1+\lambda_2)N_2\mathds{S}(H)\Big)|_{z=\infty,t=0,H=1}\, . \nonumber
\end{align}
The following relation was proven in \cite{ZaZi},
\begin{align}\label{c1c2l}
C_1^2 C_2 &= (1+27q)^{-1}\, .
\end{align}

%By definition, $\mathds{V}_{ii}$ is symmetric in %$x,y$. 

From the equations \eqref{S1} and \eqref{assym}, we can show the series $$\overline{\mathds{S}}_i({1})=\overline{\mathds{S}}({1})|_{H=\lambda_i}\,, \ \ \overline{\mathds{S}}_i(H)=
\overline{\mathds{S}}(H)|_{H=\lambda_i}\, , \ \ \overline{\mathds{S}}_i(H^2)=\overline{\mathds{S}}(H^2)|_{H=\lambda_i}$$ 
have the following asymptotic expansions:
\begin{align}\nonumber
\overline{\mathds{S}}_i({1}) & = e^{\frac{\mu_i}{z}} \Big(R_{00,i}+R_{01,i}z+R_{02,i} z^2+\ldots\Big) \, ,\\  \label{VS}
\overline{\mathds{S}}_i(H) & = e^{\frac{\mu_i}{z}} \frac{L_i}{C_1} \Big(R_{10,i}+R_{11,i}z+R_{12} z^2+\ldots\Big)\, ,\\ \nonumber
\overline{\mathds{S}}_i(H^2) & = e^{\frac{\mu_i}{z}} \frac{L_i^2}{C_1 C_2} \Big(R_{20,i}+R_{21,i}z+R_{22,i} z^2+\ldots\Big)\, .
 \end{align}
We follow here the normalization of \cite{ZaZi}. Note 
\begin{align*}
    R_{0k,i}=R_{k,i}.
\end{align*}
As in \cite[Theorem 4]{ZaZi}, we expect the following  constraints.
 
\begin{Conj}\label{RPoly}
 For all $k\geq 0$, we have
     $$R_{k,i} \in \mathds{G}_2\,.$$
\end{Conj}
Conjecture \ref{RPoly} is the main obstruction for the proof of Conjecture \ref{ooo} and \ref{HAE}. By the same argument of Section \ref{hafp}, we obtain the following result.
\begin{Thm}
Conjecture \ref{RPoly} implies Conjecture \ref{ooo} and \ref{HAE}.
\end{Thm}

By applying asymptotic expansions \eqref{VS} to \eqref{S1}, we obtain the following results.

\begin{Lemma}\label{RRR} We have
  \begin{align*}
      R_{1\,p+1,i}&=R_{0\,p+1,i}+\frac{\mathsf{D}R_{0\,p,i}}{L_i}\,,\\
      R_{2\,p+1,i}&=R_{1\,p+1,i}+\frac{\mathsf{D}R_{1\,p,i}}{L_i}+\Big(\frac{\mathsf{D}L_i}{L_i^2}-\frac{X}{L_i}\Big)-(\lambda_0+\lambda_1+\lambda_2)N_2\frac{R_{1\,k,i}}{L_i}\,,
  \end{align*}
with $X=\frac{\DD C_1}{C_1}$.  
\end{Lemma}

From Lemma \ref{RRR}, we obtain results for $\overline{\mathds{S}}({H})|_{H=\lambda_i}$
and $\overline{\mathds{S}}({H^2})|_{H=\lambda_i}$.
\begin{Lemma}\label{RPoly2} Suppose Conjecture \ref{RPoly} is true. Then for all $k\geq 0$, we have for all $k\geq 0$, 
 \begin{align*}
     &R_{1\,k,i} \in \mathds{G}_2\, ,\\
     &R_{2\,k,i} = Q_{2\,k,i} - \frac{R_{1\, k-1,i}}{L} X-(\lambda_0+\lambda_1+\lambda_2)N_2\frac{R_{1\,k,i}}{L_i}\, ,
 \end{align*}
 with $Q_{2\,k,i}\in \mathds{G}_2$.
\end{Lemma}

\subsection{Determining $\DD X$ and $N_2$}
The following relation was proven in \cite{LP}.
\begin{equation}\label{drule}
X^2-(L^3-1)X+\DD X-\frac{2}{9}(L^3-1)=0\, .
\end{equation}
 By the above result, the differential ring 
 \begin{equation}\label{ddd333}
 \mathds{G}_2[X,\DD X,\DD\DD X,\ldots]
 \end{equation}
 is just the polynomial ring $\mathds{G}_2[X]$.
Denote by $\text{Coeff}(x^iy^j)$ the coefficient of $x^iy^j$ in
$$\sum _{k=0}^2 e^{-\frac{\mu_i}{x}-\frac{\mu_i}{y}}\mathds{S}_i (\phi_k)|_{z=x} \, \mathds{S}_i(\phi ^k )|_{z=y}\,.$$
From \eqref{wdvv} and \eqref{VS}, we obtain the following equation.
$$\text{Coeff}(x^2)+\text{Coeff}(y^2)-\text{Coeff}(xy)=0\,.$$
Above equation immediately yields the following relation.
\begin{align}\label{NR}
    N_2=-\frac{1}{2}C_2+\frac{1}{2}L^3\,.
\end{align}

\section{Higher genus series on $\overline{M}_{g,n}$}\label{hgi}

 \subsection{Intersection theory on $\overline{M}_{g,n}$} \label{intmg}
 We review here the now standard method used by Givental \cite{Elliptic,SS,Book} to 
 express genus $g$ descendent correlators in terms of genus 0 data.
 
 Let $t_0,t_1,t_2, \ldots$ be formal variables. The series
 $$T(c)=t_0+t_1 c+t_2 c^2+\ldots$$   in the
additional variable $c$ plays a basic role. The variable $c$
will later be  replaced by the first Chern class $\psi_i$ of
 a cotangent line  over $\overline{M}_{g,n}$, 
 $$T(\psi_i)= t_0 + t_1\psi_i+ t_2\psi_i^2 +\ldots\, ,$$
 with the index $i$
 depending on the position of the series $T$ in the correlator.

Let $2g-2+n>0$.
For $a_i\in \mathbb{Z}_{\geq 0}$ and  $\gamma \in H^*(\overline{M}_{g,n})$, define the correlator 
%$\langle \psi^{a_1},\psi^{a_2},...,\psi^{a_n}\, |\, \gamma\rangle_{g,n}$ by
\begin{multline*}
    \lann \psi^{a_1},\ldots,\psi^{a_n}\, | \, \gamma\,  \rann_{g,n}=
    \sum_{k\geq 0} \frac{1}{k!}\int_{\overline{M}_{g,n+k}}
    \gamma \, \psi_1^{a_1}\cdots 
     \psi_n^{a_n}  \prod_{i=1}^k T(\psi_{n+i})\, . 
\end{multline*}
In the above summation,
the $k=0$ term is $$\int_{\overline{M}_{g,n}}\gamma\, \psi_1^{a_1}\cdots\psi_n^{a_n}\,.$$
We also need the following correlator defined for the unstable case,

$$\lan\lan 1,1 \ran\ran_{0,2}=\sum_{k > 0}\frac{1}{k!}\int_{\overline{M}_{0,2+k}}\prod_{i=1}^k T(\psi_{2+i})\,.$$

For formal variables $x_1,\ldots,x_n$, we also define the correlator
\begin{align}\label{derf}
\lannn \frac{1}{x_1-\psi},\ldots,\frac{1}{x_n-\psi}\, \Big| \, \gamma \, \rannn_{g,n}
\end{align}
in the standard way by expanding $\frac{1}{x_i-\psi}$ as a geometric series.

Denote by $\mathds{L}$ the differential operator 
\begin{align*}
        \mathds{L}\, =\, 
        \frac{\partial}{\partial t_0}-\sum_{i=1}^\infty t_i\frac{\partial}{\partial t_{i-1}}
        \, =\, \frac{\partial}{\partial t_0}-t_1\frac{\partial}{\partial t_0}-t_2\frac{\partial}{\partial t_1}-\ldots
        \, .
\end{align*}
 The string equation yields the following result.
 
\begin{Lemma} \label{stst} For $2g-2+n>0$, we have
$\mathds{L}\lann 1,\ldots,1\, | \, \gamma\, \rann_{g,n}=0$ and 
\begin{multline*}
\mathds{L} \lannn \frac{1}{x_1-\psi},\ldots,\frac{1}{x_n-\psi}\, \Big| \,\gamma \, 
\rannn_{g,n}= \\
 \left(\frac{1}{x_1}+\ldots +\frac{1}{x_n}\right)
 \lannn\frac{1}{x_1-\psi},\ldots \frac{1}{x_n-\psi}\, \Big| \, \gamma \, \rannn_{g,n}\, .
 \end{multline*}
\end{Lemma}

After the restriction $t_0=0$ and application of the dilaton equation,
the correlators are expressed in terms of finitely many integrals (by the
dimension constraint). For example,
\begin{eqnarray*}
    \lann 1,1,1\rann_{0,3}\, |_{t_0=0} &= &\frac{1}{1-t_1}\, ,\\
    \lann 1,1,1,1\rann_{0,4}\, |_{t_0=0}& =&\frac{t_2}{(1-t_1)^3}\, ,\\
    \lann 1,1,1,1,1\rann_{0,5}\, |_{t_0=0}&=&\frac{t_3}{(1-t_1)^4}+\frac{3 t_2^2}{(1-t_1)^5}\, ,\\
    \lann 1,1,1,1,1,1\rann_{0,6}\, |_{t_0=0}&=&\frac{t_4}{(1-t_1)^5}+\frac{10 t_2 t_3}{(1-t_1)^6}+\frac{15 t^3_2}{(1-t_1)^7}\, .
\end{eqnarray*}\\

We consider 
$\CC(t_1)[t_2,t_3,...]$
as $\ZZ$-graded ring over $\CC(t_1)$ with 
$$\text{deg}(t_i)=i-1\ \ \text{for $i\geq 2$ .}$$
Define a subspace of homogeneous elements by
$$\CC\left[\frac{1}{1-t_1}\right][t_2,t_3,\ldots]_{\text{Hom}} \subset 
\CC(t_1)[t_2,t_3,...]\, .
$$
We easily see 
$$\lann \psi^{a_1},\ldots,\psi^{a_n}\, | \, \gamma \, \rann_{g,n}\, |_{t_0=0}\ \in\
\CC\left[\frac{1}{1-t_1}\right][t_2,t_3,\ldots]_{\text{Hom}}\, .$$
%\subset
%\CC(t_1)[t_2,t_3,\ldots%]_{\text{Hom}} .$$
Using the leading terms (of lowest degree in $\frac{1}{(1-t_1)}$), we obtain the
following result.

\begin{Lemma}\label{basis}
The set of genus 0 correlators
 $$
 \Big\{ \, \lann 1,\ldots,1\rann_{0,n}\, |_{t_0=0} \, \Big\}_{n\geq  4} $$ 
freely generate the ring
 $\CC(t_1)[t_2,t_3,...]$ over $\CC(t_1)$.
\end{Lemma}

By  Lemma \ref{basis}, we can find a unique representation of $\lann \psi^{a_1},\ldots,\psi^{a_n}\rann_{g,n}|_{t_0=0}$
in the  variables
\begin{equation}\label{k3k3}
\Big\{\, \lann 1,\ldots,1\rann_{0,n}|_{t_0=0}\, \Big\}_{n\geq 3}\, .
\end{equation}
The $n=3$ correlator is included in the set \eqref{k3k3} to
capture the variable $t_1$.
For example, in $g=1$,
\begin{eqnarray*}
    \lann 1,1\rann_{1,2}|_{t_0=0}&=&\frac{1}{24}
    \left(\frac{\lann 1,1,1,1,1\rann_{0,5}|_{t_0=0}}{\lan 1,1,1\rann_{0,3}|_{t_0=0}}-\frac{\lann 1,1,1,1\rann^2_{0,4}|_{t_0=0}}{\lann 1,1,1\rann^2_{0,3}|_{t_0=0}}\right)\, ,\\
    \lann 1\rann_{1,1}|_{t_0=0}&=&\frac{1}{24}\frac{\lann 1,1,1,1\rann_{0,4}|_{t_0=0}}{\lann 1,1,1\rann_{0,3}|_{t_0=0}}
    \end{eqnarray*}
A more complicated example in $g=2$ is    
\begin{eqnarray*}
\lann \ \rann_{2,0}|_{t_0=0}&=& \ \ \frac{1}{1152}\frac{\lann 1,1,1,1,1,1\rann_{0,6}|_{t_0=0}}{\lann 1,1,1\rann_{0,3}|_{t_0=0}^2}\\
    & & -\frac{7}{1920}\frac{\lann 1,1,1,1,1\rann_{0,5}|_{t_0=0}\lann 1,1,1,1\rann_{0,4}|_{t_0=0}}{\lann 1,1,1\rann_{0,3}|_{t_0=0}^3}\\& &+\frac{1}{360}\frac{\lann 1,1,1,1\rann_{0,4}|_{t_0=0}^3}{\lann 1,1,1\rann_{0,3}
    |_{t_0=0}^4}\, .
\end{eqnarray*}

\begin{Def} 
For $\gamma \in H^*(\overline{M}_{g,k})$, let $$\pP^{a_1,\ldots,a_n,\gamma}_{g,n}(s_0,s_1,s_2,...)\in \QQ(s_0, s_1,..)$$ be 
the unique rational function satisfying the condition
$$\lann \psi^{a_1},\ldots,\psi^{a_n}\, |\, \gamma\, \rann_{g,n}|_{t_0=0}
=\pP^{a_1,a_2,...,a_n,\gamma}_{g,n}|_{s_i=\lann 1,\ldots,1\rann_{0,i+3}|_{t_0=0}}\, . $$
\end{Def}
 
\begin{Prop}\label{GR1} For $2g-2+n>0$,
we have
 $$\lann 1,\ldots,1\,|\, \gamma\, \rann_{g,n}
=\pP^{0,\ldots,0,\gamma}_{g,n}|_{s_i=\lann 1,\ldots,1\rann_{0,i+3}}\, . $$
\end{Prop} 

\begin{proof}
 Both sides of the equation satisfy the differential equation
 \begin{align*}
     \mathds{L}=0.
 \end{align*}
 By definition, both sides have the same initial conditions at $t_0=0$.
\end{proof}

\begin{Prop}\label{GR2} For $2g-2+n>0$,
 \begin{multline*}
     \lannn \frac{1}{x_1-\psi_1}, \ldots, \frac{1}{x_n-\psi_n}\, \Big| \, \gamma \, \rannn_{g,n}= \\
     e^{\lann 1,1\rann_{0,2}(\sum_i\frac{1}{x_i})}\sum_{a_1,\ldots,a_n}\frac{\pP^{a_1,\ldots,a_n,\gamma}_{g,n}|_{s_i=\lann 1,\ldots,1\rann_{0,i+3}}
     }{x_1^{a_1+1} \cdots x_n^{a_n+1}}.
 \end{multline*}
\end{Prop} 
 
 \begin{proof}
  Both sides of the equation satisfy differential equation
  \begin{align*}
      \mathds{L}-\sum_i\frac{1}{x_i}=0.
  \end{align*}
 Both sides have the same initial conditions at $t_0=0$.
 We use here
 %since
 %\begin{align*}
 %    <\frac{1}{x_1-\psi}, \frac{1}{x_2-\psi}, ..., %\frac{1}{x_n-\psi}>_{g,n}|_{t_0=0}=\frac{<\psi^{a_1},\psi^{a_2},...,%\psi^{a_n}>_{g,n}|_{t_0=0}}{x_1^{a_1} x_2^{a_2} .... x_n^{a_n}},
 %\end{align*} 
 %and
 %begin{align*}
     $$\mathds{L} \lann 1,1\rann_{0,2} =1\,, \ \ \ \  \lann 1,1\rann_{0,2}|_{t_0=0}=0\, .$$
 %\end{align*}
 There is no conflict here with Lemma
 \ref{stst} since $(g,n)=(0,2)$ is not
 in the stable range.
 \end{proof}

\subsection{The unstable case $(0,2)$}
The definition given in \eqref{derf}
of the correlator is valid
in the stable range $$2g-2+n>0\, .$$
The unstable case $(g,n)=(0,2)$ plays a
special role. We define
$$\lannn \frac{1}{x_1-\psi_1}, \frac{1}{x_2-\psi_2}\rannn_{0,2}$$
by 
adding the
degenerate term
$$\frac{1}{x_1+x_2}$$
to the terms obtained
by the 
 expansion of $\frac{1}{x_i-\psi_i}$ as 
 a geometric series.
 The degenerate term is associated
to the (unstable) moduli space
of genus 0 with 2 markings.

\begin{Prop}\label{GR22} We have
 \begin{equation*}
     \lannn \frac{1}{x_1-\psi_1}, \frac{1}{x_2-\psi_2} \rannn_{0,2}= 
     e^{\lann 1,1\rann_{0,2}\left(\frac{1}{x_1}+
     \frac{1}{x_2}\right)}\left(\frac{1}{x_1+x_2}\right)\, .
 \end{equation*}
\end{Prop} 
 
 \begin{proof}
  Both sides of the equation satisfy differential equation
  \begin{align*}
      \mathds{L}-\sum_{i=1}^2\frac{1}{x_i}=0.
  \end{align*}
 Both sides have the same initial conditions at $t_0=0$.
  %since
 %\begin{align*}
 %    <\frac{1}{x_1-\psi}, \frac{1}{x_2-\psi}, ..., %\frac{1}{x_n-\psi}>_{g,n}|_{t_0=0}=\frac{<\psi^{a_1},\psi^{a_2},...,%\psi^{a_n}>_{g,n}|_{t_0=0}}{x_1^{a_1} x_2^{a_2} .... x_n^{a_n}},
 %\end{align*} 
 %and
 %begin{align*
 \end{proof}

\subsection{Local invariants and wall crossing} 
%Let $(g,n)$ satisfy the Deligne-Mumford stability condition 
%$$2g-2+n>0\, .$$
The torus $\T$ acts on the moduli spaces
$\overline{M}_{g,n}(\PP^2,d)$  and
$\overline{Q}_{g,n}(\PP^2,d)$.
We consider here special localization contributions 
associated to the fixed points ${p}_i\in \PP^2$.

%Let 
%$$f: (C,p_1,\ldots,p_n) \rightarrow \PP^2$$
%be a genus $g$, $n$-pointed, degree $d$  stable map
%%to $\PP^2$ which is $\T$-fixed.
%Let 
%$$\pi: (C,p_1,\ldots,p_n) \rightarrow
%(C^{st},p_1,\ldots,p_n)$$
%be the Deligne-Mumford stabilization of the domain curve.

Consider first the moduli of stable maps.
Let
$$\overline{M}_{g,n}(\PP^2,d)^{\T,p_i}
\subset \overline{M}_{g,n}(\PP^2,d) $$
be the union of
 $\T$-fixed loci which parameterize stable maps
obtained by attaching $\T$-fixed rational tails to a genus $g$, $n$-pointed
Deligne-Mumford stable curve contracted
to the point $p_i\in\PP^2$.
Similarly, let 
$$\overline{Q}_{g,n}(\PP^2,d)^{\T,p_i}\subset
\overline{Q}_{g,n}(\PP^2,d)
$$
be the parallel $\T$-fixed locus
parameterizing stable quotients obtained
by attaching base points
to  a genus $g$, $n$-pointed
Deligne-Mumford stable curve contracted
to the point $p_i\in\PP^2$.

Let $\Lambda_i$ denote the localization of the ring
$$\CC[\lambda^{\pm 1}_0,\lambda^{\pm 1}_1,\lambda^{\pm 1}_2]$$ at 
the three tangent weights at $p_i\in \PP^2$.
Using the virtual
localization formula \cite{GP}, 
there exist unique series
$$S_{p_i}\in\Lambda_i[\psi][[Q]]$$ 
for which the localization contribution 
of the $\T$-fixed locus
$\overline{M}_{g,n}(\PP^2,d)^{\T,p_i}$
to the equivariant Gromov-Witten
invariants of $K\PP^2$
can be written as
\begin{multline*}
    \sum_{d=0}^\infty Q^d \int_{[\overline{M}_{g,n}(K\PP^2,d)^{\T,p_i}]^{\vir}}
    \psi_1^{a_1}\cdots\psi_n^{a_n}=\\
    \sum_{k=0}^\infty \frac{1}{k!} \int_{\overline{M}_{g,n+k}}
    {\mathsf{H}}_{g}^{p_i}\, \psi_1^{a_1}\cdots\psi_n^{a_n}\, \prod_{j=1}^k S_{p_i}(\psi_{n+j})\, .
\end{multline*}
Here, $\mathsf{H}_{g}^{p_i}$ is the standard vertex class, 
\begin{equation}\label{hhbb}
\frac{e(\mathbb{E}_g^*\otimes T_{p_i}(\PP^2))}{e(T_{p_i}(\PP^2))} \cdot \frac{e(\mathbb{E}_g^* \otimes(-3\lambda_i))}{(-3\lambda_i)}\, ,
\end{equation}
obtained the  Hodge bundle $\mathbb{E}_g\rightarrow \overline{M}_{g,n+k}$.

Similarly, the application of the
virtual localization formula to the moduli of stable
quotients yields classes
$$F_{p_i,k}\in H^*(\overline{M}_{g,n|k})\otimes_\CC\Lambda_i$$ 
for which the contribution of $\overline{Q}_{g,n}(\PP^2,d)^{T,p_i}$ is given by
\begin{multline*}
    \sum_{d=0}^\infty q^d \int_{[\overline{Q}_{g,n}(K\PP^2,d)^{\T,p_i}]^{\vir}}\psi_1^{a_1}\cdots
    \psi_n^{a_n}=\\
    \sum_{k=0}^\infty \frac{q^k}{k!} \int_{\overline{M}_{g,n|k}} \mathsf{H}_{g}^{p_i}\, \psi_1^{a_1}\cdots \psi_n^{a_n}\, F_{p_i,k}.
\end{multline*}
Here $\overline{M}_{g,n|k}$ is the moduli space of genus $g$ curves with markings
$$\{p_1,\cdots,p_n\}\cup\{\hat{p}_1\cdots\hat{p}_k \}\in C^{\text{ns}}\subset C$$
satisfying the conditions
\begin{itemize}
 \item[(i)] the points $p_i$ are distinct,
 \item[(ii)] the points $\hat{p}_j$ are distinct from the points $p_i$,
\end{itemize}
with stability given by the ampleness of 
$$\omega_C(\sum_{i=1}^m p_i+\epsilon\sum_{j=1}^k \hat{p}_j)$$
for every strictly positive $\epsilon \in \QQ$.

The Hodge class $\mathsf{H}_{g}^{p_i}$ is given again by
formula \eqref{hhbb} using the Hodge bundle $$\mathbb{E}_g\rightarrow \overline{M}_{g,n|k}\, .$$

\begin{Def}
 For $\gamma\in H^*(\overline{M}_{g,n})$, let
 \begin{eqnarray*}
     \lann \psi_1^{a_1},\ldots,\psi_n^{a_n}\, |\, \gamma\, \rann_{g,n}^{p_i,\infty}
     &=&
     \sum_{k=0}^\infty \frac{1}{k!}
     \int_{\overline{M}_{g,n+k}} \gamma \, \psi_1^{a_1}\cdots \psi_n^{a_n}\prod_{j=1}^k S_{p_i}(\psi_{n+j})\, ,\\
    \lann \psi_1^{a_1},\ldots,\psi_n^{a_n}\, |\, \gamma\, \rann_{g,n}^{p_i,0+}&=&
    \sum_{k=0}^\infty \frac{q^k}{k!} \int_{\overline{M}_{g,n|k}} \gamma \, \psi_1^{a_1}\cdots \psi_n^{a_n}\, F_{p_i,k}\, .
 \end{eqnarray*}
\end{Def}

\
\begin{Prop} [Ciocan-Fontanine, Kim \cite{CKg}] \label{WC} 
For $2g-2+n>0$,
we have the wall crossing relation
$$\lann \psi_1^{a_1},\ldots,\psi_n^{a_n}\, |\, \gamma\, \rann_{g,n}^{p_i,\infty}(Q(q))= \lann \psi_1^{a_1},\ldots,\psi_n^{a_n}\, |\, \gamma\rann_{g,n}^{p_i,0+}(q)$$
 where 
 $Q(q)$ is the mirror map
 $$Q(q)=\exp(I_1^{K\PP^2}(q))\, .$$
\end{Prop}

Proposition \ref{WC} is a consequence
of \cite[Lemma 5.5.1]{CKg}. The mirror
map here is the mirror map for
$K\PP^2$ discussed in Section \ref{holp2}.
 Propositions \ref{GR1} and \ref{WC} together yield 
 \begin{eqnarray*}
 \lann 1,\ldots,1\,  |\, \gamma \, \rann_{g,n}^{p_i,\infty}& =& \pP^{0,\ldots,0,\gamma}_{g,n}\big(\lann 1,1,1\rann_{0,3}^{p_i,\infty},\lann 1,1,1,1\rann _{0,4}^{p_i,\infty},\ldots\big)\, ,\\
 \lann 1,\ldots,1\,  |\, \gamma \, \rann_{g,n}^{p_i,0+}&=&\pP^{0,\ldots,0,\gamma}_{g,n}\big(\lann
 1,1,1\rann_{0,3}^{p_i,0+},\lann 1,1,1,1\rann _{0,4}^{p_i,0+},\ldots\big)\, .
 \end{eqnarray*}
Similarly, using Propositions \ref{GR2} and \ref{WC}, we obtain
\begin{multline*}
\lannn \frac{1}{x_1-\psi}, \ldots, \frac{1}{x_n-\psi}\, \Big| \, \gamma \, \rannn_{g,n}^{p_i,\infty}= \\
     e^{\lann 1,1\rann^{p_i,\infty}_{0,2}\left(\sum_i\frac{1}{x_i}\right)}\sum_{a_1,\ldots,a_n}\frac{\pP^{a_1,\ldots,a_n,\gamma}_{g,n}\big(\lann 1,1,1\rann_{0,3}^{p_i,\infty},\lann 1,1,1,1\rann_{0,4}^{p_i,\infty},\ldots \big)}{x_1^{a_1+1}\cdots x_n^{a_n+1}}\, ,
\end{multline*}
\begin{multline}\label{ppqqpp}
\lannn \frac{1}{x_1-\psi}, \ldots, \frac{1}{x_n-\psi}\, \Big| \, \gamma \, \rannn_{g,n}^{p_i,0+}= \\
     e^{\lann 1,1\rann^{p_i,0+}_{0,2}\left(\sum_i\frac{1}{x_i}\right)}\sum_{a_1,\ldots,a_n}\frac{\pP^{a_1,\ldots,a_n,\gamma}_{g,n}\big(\lann 1,1,1\rann_{0,3}^{p_i,0+},\lann 1,1,1,1\rann_{0,4}^{p_i,0+},\ldots \big)}{x_1^{a_1+1}\cdots x_n^{a_n+1}}\, .
\end{multline}

\section{Higher genus series on $K\PP^2$}\label{hgs}
\subsection{Overview}
We apply the localization strategy introduced first by Givental  \cite{Elliptic,SS,Book} for Gromov-Witten theory to the stable quotient invariants of local $\PP^2$. 
The contribution $\text{Cont}_\Gamma(q)$ 
discussed in Section \ref{locq} 
of a graph $\Gamma \in \mathsf{G}_{g}(\PP^2)$ 
can be separated into vertex and edge contributions.
We express the vertex and edge contributions in terms of
the series $\mathds{S}_i$ and $\mathds{V}_{ij}$ of Section \ref{lightm}.

\subsection{Edge terms}
Recall the definition{\footnote{We use
the variables $x_1$ and $x_2$ here instead
of $x$ and $y$.}}of $\mathds{V}_{ij}$
given in Section \ref{lightm},
\begin{equation}\label{dfdf6}
\mathds{V}_{ij}  =  
\Big\langle \Big\langle  \frac{\phi _i}{x- \psi } ,  \frac{\phi _j}{y - \psi } 
\Big\rangle \Big\rangle  _{0, 2}^{0+,0+}  \, .
\end{equation}
Let $\overline{\mathds{V}}_{ij}$ denote
the restriction of $\mathds{V}_{ij}$
to $t=0$.
Via formula \eqref{ddgg},
$\overline{\mathds{V}}_{ij}$ is a summation of contributions of fixed loci indexed by
a graph $\Gamma$ consisting of two vertices 
connected by a unique edge. 
Let $w_1$ and $w_2$ be 
$\T$-weights. Denote by $$\overline{{\mathds{V}}}_{ij}^{w_1,w_2}$$ the summation of contributions of $\T$-fixed loci with
tangent weights precisely $w_1$
and $w_2$ on the first rational components
which exit the vertex components over
$p_i$ and $p_j$.

The series $\overline{{\mathds{V}}}_{ij}^{w_1,w_2}$
includes {\em both} vertex and edge
contributions.
By definition \eqref{dfdf6} and the virtual localization formula, we find the
following relationship between
$\overline{\mathds{V}}_{ij}^{w_1,w_2}$
and the corresponding
pure edge contribution $\mathsf{E}_{ij}^{w_1,w_2}$,

\begin{eqnarray*}
    e_i\overline{\mathds{V}}_{ij}^{w_1,w_2}e_j
    &=& \lannn \frac{1}{w_1-\psi},\frac{1}{x_1-\psi}\rannn^{p_i,0+}_{0,2}\mathsf{E}_{ij}^{w_1,w_2}
    \lannn \frac{1}{w_2-\psi},\frac{1}{x_2-\psi}\rannn^{p_j,0+}_{0,2}\\
    &=&\frac{e^{\frac{\lann 1,1\rann^{p_i,0+}_{0,2}}{w_1}+\frac{\lann 1,1\rann^{p_j,0+}_{0,2}}{x_1}}}{w_1+x_1}
    \, \mathsf{E}^{w_1,w_2}_{ij}\, \frac{e^{\frac{\lann 1,1\rann^{p_i,0+}_{0,2}}{w_2}+\frac{\lann 1,1\rann^{p_j,0+}_{0,2}}{x_2}}}{w_2+x_2}
     \end{eqnarray*}
    
\begin{align*}        
    =\sum_{a_1,a_2}e^{\frac{\lann 1,1\rann^{p_i,0+}_{0,2}}{x_1}+\frac{\lann 1,1\rann^{p_i,0+}_{0,2}}{w_1}}e^{\frac{\lann 1,1\rann^{p_j,0+}_{0,2}}{x_2}+\frac{\lann 1,1\rann^{p_j,0+}_{0,2}}{w_2}}(-1)^{a_1+a_2} \frac{
    \mathsf{E}^{w_1,w_2}_{ij}}{w_1^{a_1}w_2^{a_2}}x_1^{a_1-1}x_2^{a_2-1}\, .
\end{align*}
After summing over all possible weights, we obtain
$$
    e_i\left(\overline{\mathds{V}}_{ij}-\frac{\delta_{ij}}{e_i(x+y)}\right)e_j=\sum_{w_1,w_2} e_i\overline{\mathds{V}}_{ij}^{w_1,w_2}e_j\, .$$
The above calculations immediately yield
the following result.
    
%    &=\sum_{w_1,w_2}<\frac{1}{w_1-\psi},\frac{1}{x_1-\psi}>E_{w_1,w_2}<\frac{1}{w_2-\psi},\frac{1}{x_2-\psi}>\\
%    &=\sum_{w_1,w_2}\frac{e^{\frac{<1,1>}{w_1}+\frac{<1,1>}{x_1}}}{w_1+x_1}E_{w_1,w_2}\frac{e^{\frac{<1,1>}{w_2}+\frac{<1,1>}{x_2}}}{w_2+x_2}\\
%    &=\sum_{w_1,w_2}\sum_{(a_1,a_2)}e^{\frac{<1,1>}{x_1}+\frac{<1,1>}{w_1}}e^{\frac{<1,1>}{x_2}+\frac{<1,1>}{w_2}}\frac{E_{w_1,w_2}}{w_1^{a_1}w_2^{a_2}}x_1^{a_1-1}y_1^{a_2-1}\\
%    &=\sum_{(a_1,a_2)}\sum_{w_1,w_2}e^{\frac{<1,1>}{x_1}+\frac{<1,1>}{w_1}}e^{\frac{<1,1>}{x_2}+\frac{<1,1>}{w_2}}\frac{E_{w_1,w_2}}{w_1^{a_1}w_2^{a_2}}x_1^{a_1-1}y_1^{a_2-1}.
%\end{align*}

\begin{Lemma}\label{Edge} We have
 \begin{multline*}
 \left[e^{-\frac{\lann1,1\rann^{p_i,0+}_{0,2}}{x_1}}
       e^{-\frac{\lann1,1\rann^{p_j,0+}_{0,2}}{x_2}}e_i\left(\overline{\mathds{V}}_{ij}-\frac{\delta_{ij}}{e_i(x+y)}\right)e_j\right]_{x_1^{a_1-1}x_2^{a_2-1}}=\\
       \sum_{w_1,w_2}
       e^{\frac{\lann1,1\rann^{p_i,0+}_{0,2}}{w_1}}e^{\frac{\lann1,1\rann^{p_j,0+}_{0,2}}{w_2}}(-1)^{a_1+a_2}\frac{\mathsf{E}_{ij}^{w_1,w_2}}{w_1^{a_1}w_2^{a_2}}\, .
 \end{multline*}
\end{Lemma}

\noindent The notation $[\ldots]_{x_1^{a_1-1}x_2^{a_2-1}}$
in Lemma \ref{Edge} denotes the coefficient of
 $x_1^{a_1-1}x_2^{a_2-1}$ in the series expansion
 of the argument.

\subsection{A simple graph}\label{simgr}
Before treating the general case, we present
the localization formula for a simple graph{\footnote{We
follow here the notation of Section \ref{locq}.}.
Let
$\Gamma\in \mathsf{G}_{g}(\PP^2)$ 
consist of two vertices   and one edge,
$$v_1,v_2\in \Gamma(V)\, , \ \ \ \ 
e\in \Gamma(E)\, $$
with genus and $\T$-fixed point assignments
$$\mathsf{g}(v_i)=g_i\, , \ \ \ \ \mathsf{p}(v_i)=p_i\, .$$

Let $w_1$ and $w_2$ be tangent
weights at the vertices $p_1$ and $p_2$
respectively. Denote by $\text{Cont}_{\Gamma,w_1,w_2}$
the summation of contributions to
\begin{equation}\label{zlzl}
\sum_{d>0} q^d\, \left[\overline{Q}_{g}(K\PP^2,d)\right]^{\vir}
\end{equation}
of $\T$-fixed loci with
tangent weights precisely $w_1$
and $w_2$ on the first rational components
which exit the vertex components over
$p_1$ and $p_2$.
We can express the localization formula for 
\eqref{zlzl} as
$$
\lannn \frac{1}{w_1-\psi}\, \Big|\, \mathsf{H}_{g_1}^{p_1}
\rannn_{g_1,1}^{p_1,0+}
\mathsf{E}^{w_1,w_2}_{12} \lannn\frac{1}{w_2-\psi}\, \Big|\, \mathsf{H}_{g_2}^{p_2}
\rannn_{g_2,1}^{p_2,0+} $$
which equals
$$\sum_{a_1,a_2} e^{\frac{\lann1,1\rann^{p_1,0+}_{0,2}}{w_1}}\frac{\ppl
{\psi^{a_1-1}} \, \Big|\, \mathsf{H}_{g_1}^{p_1}     \ppr_{g_1,1}^{p_1,0+}} {w_1^{a_1}} \mathsf{E}^{w_1,w_2}_{12} e^{\frac{\lann 1,1\rann^{p_2,0+}_{0,2}}{w_2}}\frac{\ppl {\psi^{a_2-1}} \, \Big|\, \mathsf{H}_{g_2}^{p_2}\ppr_{g_2,1}^{p_2,0+}}{w_2^{a_2}}
$$
%\begin{align*}
%    [\sum_d q^d \overline{Q}_{g}(K\PP^2,d)]^{\vir}_{\Gamma_{w_1,w_2}}&=H^{g_1,p_1}<\frac{1}{w_1-\psi}>E_{w_1,w_2}H^{g_2,p_2}<\frac{1}{w_2-\psi}>\\
%    &=\sum_{a_1,a_2} e^{\frac{<1,1>}{w_1}}H^{g_1,p_1}[\frac{\psi^{a_1-1}}{w_1^{a_1}}]E_{w_1,w_2}e^{\frac{<1,1>}{w_2}}H^{g_2,p_2}[\frac{\psi^{a_2-1}}{w_2^{a_2}}]
%\end{align*}
where $\mathsf{H}_{g_i}^{p_i}$ is
the Hodge class \eqref{hhbb}. We have used here
the notation
\begin{multline*}
\ppl
\psi^{k_1}_1, \ldots,\psi^{k_n}_n \, \Big|\, \mathsf{H}_{h}^{p_i}     \ppr_{h,n}^{p_i,0+} 
=\\
\pP^{k_1,\ldots,k_n,\mathsf{H}_{h}^{p_i}  }_{h,1}\big(\lann 1,1,1\rann_{0,3}^{p_i,0+},\lann 1,1,1,1\rann_{0,4}^{p_i,0+},\ldots \big)
\,
\end{multline*}
and applied \eqref{ppqqpp}.

% \begin{align*}
%     H^{g,p_i}:=\frac{\prod_{k \ne i} \prod_{j=0}^g (\lambda_i-\lambda_k-c_j)}{\prod_{k \ne i} (\lambda_i-\lambda_k)}\frac{-3\lambda_i}{\prod_{j=0}^g(-3\lambda_i-c_j)}.
% \end{align*}
After summing over all possible weights $w_1,w_2$ and
applying 
%\begin{align*}
%    [\sum_d q^d \overline{Q}_{g}(K\PP^2,d)]^{\vir}_{\Gamma}&=[\sum_d q^d \overline{Q}_{g}(K\PP^2,d)]^{\vir}_{\Gamma_{w_1,w_2}}\\
%    &=\sum_{w_1,w_2}H^{g_1,p_1}<\frac{1}{w_1-\psi}>E_{w_1,w_2}H^{g_2,p_2}<\frac{1}{w_2-\psi}>\\
%    &=\sum_{w_1,w_2}\sum_{a_1,a_2} e^{\frac{<1,1>}{w_1}}H^{g_1,p_1}[\frac{\psi^{a_1-1}}{w_1^{a_1}}]E_{w_1,w_2}e^{\frac{<1,1>}{w_2}}H^{g_2,p_2}[\frac{\psi^{a_2-1}}{w_2^{a_2}}]\\
%    &=\sum_{a_1,a_2}\sum_{w_1,w_2} H^{g_1,p_1}[\psi^{a_1-1}]H^{g_2,p_2}[\psi^{a_2-1}](e^{\frac{<1,1>}{w_1}}e^{\frac{<1,1>}{w_2}}\frac{E_{w_1,w_2}}{w_1^{a_1}w_2^{a_2}})
%\end{align*}
Lemma \ref{Edge}, we obtain the following result for the full contribution $$\text{Cont}_\Gamma = \sum_{w_1,w_2} \text{Cont}_{\Gamma,w_1,w_2}$$
of $\Gamma$ to $\sum_{d\geq 0} q^d \left[ \overline{Q}_{g}(K\PP^2,d)\right]^{\vir}$.

\begin{Prop} We have \label{propsim}
 \begin{multline*}
     \text{\em Cont}_{\Gamma}=
     \sum_{a_1,a_2>0}
     \ppl
{\psi^{a_1-1}}  \, \Big|\, \mathsf{H}_{g_1}^{p_i}\,     \ppr_{g_1,1}^{p_i,0+}
\ppl
{\psi^{a_2-1}}  \, \Big|\, \mathsf{H}_{g_2}^{p_j}\,     \ppr_{g_2,1}^{p_j,0+}\ \ \ \ \ \ \ \ \ \ \ \\
\ \ \ \ \ \ \ \ \ \ \cdot
     (-1)^{a_1+a_2}\left[e^{-\frac{\lann1,1\rann^{p_i,0+}_{0,2}}{x_1}}
       e^{-\frac{\lann1,1\rann^{p_j,0+}_{0,2}}{x_2}}e_i\left(\overline{\mathds{V}}_{ij}-\frac{\delta_{ij}}{e_i(x_1+x_2)}\right)e_j\right]_{x_1^{a_1-1}x_2^{a_2-1}}\, .
 \end{multline*}
\end{Prop}

\subsection{A general graph} We apply the argument of Section \ref{simgr}
to obtain a contribution formula for a general graph $\Gamma$.

Let $\Gamma\in \mathsf{G}_{g,0}(\PP^2)$ be a decorated graph as defined in Section \ref{locq}. The {\em flags} of $\Gamma$ are the 
half-edges{\footnote{Flags are either half-edges or markings.}}. Let $\mathsf{F}$ be the set of flags. 
Let
$$\mathsf{w}: \mathsf{F} \rightarrow \text{Hom}(\T, \com^*)\otimes_{\mathbb{Z}}{\mathbb{Q}}$$
be a fixed assignment of $\T$-weights to each flag.

We first consider the contribution $\text{Cont}_{\Gamma,\mathsf{w}}$ to 
$$\sum_{d\geq 0} q^d \left[\overline{Q}_g(K\PP^2,d)\right]^{\text{vir}}$$
of the $\T$-fixed loci associated $\Gamma$ satisfying
the following property:
the tangent weight on
the first rational component corresponding
to each $f\in \mathsf{F}$ is
exactly given by $\mathsf{w}(f)$.
We have 
\begin{equation}
    \label{s234}
    \text{Cont}_{\Gamma,\mathsf{w}} = \frac{1}{|\text{Aut}(\Gamma)|}
    \sum_{\mathsf{A} \in \ZZ_{> 0}^{\mathsf{F}}} \prod_{v\in \mathsf{V}} \text{Cont}^{\mathsf{A}}_{\Gamma,\mathsf{w}} (v)\prod _{e\in \mathsf{E}} {\text{Cont}}_{\Gamma,\mathsf{w}}(e)\, .
\end{equation}
 The terms on the  right side of \eqref{s234} 
require definition:
\begin{enumerate}
\item[$\bullet$] The sum on the right is over 
the set $\ZZ_{> 0}^{\mathsf{F}}$ of
all maps 
$$\mathsf{A}: \mathsf{F} \rightarrow \ZZ_{> 0}$$
corresponding to the sum over $a_1,a_2$ in
Proposition \ref{propsim}.
\item[$\bullet$]
For $v\in \mathsf{V}$ with 
$n$ incident
flags with $\mathsf{w}$-values $(w_1,\ldots,w_n)$ and
$\mathsf{A}$-values
$(a_1,a_2,...,a_n)$, 
\begin{align*}
    \text{Cont}^{\mathsf{A}}_{\Gamma,{\mathsf{w}}}(v)=
    \frac{\ppl
\psi_1^{a_1-1}, \ldots,
\psi_n^{a_n-1}
\, \Big|\, \mathsf{H}_{\mathsf{g}(v)}^{\mathsf{p}(v)}\,     \ppr_{\mathsf{g}(v),n}^{\mathsf{p}(v),0+}}
{w_1^{a_1} \cdots w_n^{a_n}}\, .
\end{align*}
\item[$\bullet$]
For $e\in \mathsf{E}$ with 
assignments $(\mathsf{p}(v_1), \mathsf{p}(v_2))$
for the two associated vertices{\footnote{In case $e$
is self-edge, $v_1=v_2$.}} and 
$\mathsf{w}$-values $(w_1,w_2)$ for the two associated flags,
    $$    
    \text{Cont}_{\Gamma,\mathsf{w}}(e)=
    e^{\frac{\lann1,1\rann^{\mathsf{p}(v_1),0+}_{0,2}}{w_1}}
    e^{\frac{\lann1,1\rann^{\mathsf{p}(v_2),0+}_{0,2}}{w_2}}
    \mathsf{E}^{w_1,w_2}_{\mathsf{p}(v_1),\mathsf{p}(v_2)}\, .$$
\end{enumerate}
The localization formula then yields \eqref{s234}
just as in the simple case of Section \ref{simgr}.

By summing the contribution \eqref{s234} of $\Gamma$ over
all the weight functions $\mathsf{w}$
and applying Lemma \ref{Edge}, we obtain
the following result which generalizes 
Proposition \ref{propsim}.

\begin{Prop}\label{VE} We have
 $$
 \text{\em Cont}_\Gamma
     =\frac{1}{|\text{\em Aut}(\Gamma)|}
     \sum_{\mathsf{A} \in \ZZ_{> 0}^{\mathsf{F}}} \prod_{v\in \mathsf{V}} 
     \text{\em Cont}^{\mathsf{A}}_\Gamma (v)
     \prod_{e\in \mathsf{E}} \text{\em Cont}^{\mathsf{A}}_\Gamma(e)\, ,
 $$
 where the vertex and edge contributions 
 with incident flag $\mathsf{A}$-values $(a_1,\ldots,a_n)$
 and $(b_1,b_2)$ respectively are
 \begin{eqnarray*}
    \text{\em Cont}^{\mathsf{A}}_\Gamma (v)&=&
    \ppl
\psi_1^{a_1-1}, \ldots,
\psi_n^{a_n-1}
\, \Big|\, \mathsf{H}_{\mathsf{g}(v)}^{\mathsf{p}(v)}\,
  \ppr_{\mathsf{g}(v),n}^{\mathsf{p}(v),0+}\,  ,\\
    \text{\em Cont}^{\mathsf{A}}_\Gamma(e)
    &=&
    (-1)^{b_1+b_2}\left[e^{-\frac{\lann1,1\rann^{\mathsf{p}(v_1),0+}_{0,2}}{x_1}}
       e^{-\frac{\lann1,1\rann^{\mathsf{p}(v_2),0+}_{0,2}}{x_2}}e_i\left(\overline{\mathds{V}}_{ij}-\frac{1}{e_i(x+y)}\right)e_j\right]_{x_1^{b_1-1}x_2^{b_2-1}}\, ,
%    (e^{-\frac{<1,1>}{x_1}}e^{-\frac{<1,1>}{x_2}}\mathds{V}_{e(1)e(2)})_{x_1^{b_1-1}x_2^{b_2-1}}.
 \end{eqnarray*}
where $\mathsf{p}(v_1)=p_i$ and $\mathsf{p}(v_2)=p_j$ in the second equation. 
\end{Prop}

\subsection{Legs} 
Let $\Gamma \in \mathsf{G}_{g,n}(\PP^2)$ be a decorated graph
with markings. While no markings are needed to define the
stable quotient invariants of $K\PP^2$, the contributions
of decorated graphs with markings will appear in the
proof of the holomorphic anomaly equation.
The formula for the contribution $\text{Cont}_\Gamma(H,\ldots,H)$
of $\Gamma$ to 
\begin{align*}
    \sum_{d\ge 0}q^d \prod_{j=0}^n \text{ev}^*(H)\cap\left[ \overline{Q}_{g,n}(K\PP^2,d)\right]^{\vir} 
\end{align*}
is given by the following result.
\begin{Prop}\label{VEL} We have
 \begin{multline*}
 \text{\em Cont}_\Gamma(H,\ldots,H)
     =\\\frac{1}{|\text{\em Aut}(\Gamma)|}
     \sum_{\mathsf{A} \in \ZZ_{>0}^{\mathsf{F}}} \prod_{v\in \mathsf{V}} 
     \text{\em Cont}^{\mathsf{A}}_\Gamma (v)
     \prod_{e\in \mathsf{E}} \text{\em Cont}^{\mathsf{A}}_\Gamma(e)
     \prod_{l\in \mathsf{L}} \text{\em Cont}^{\mathsf{A}}_\Gamma(l)\, ,
 \end{multline*}
 where the leg contribution 
% with $\mathsf{A}$-value $\mathsf{A}(l)$
 is 
 \begin{eqnarray*}
     \text{\em Cont}^{\mathsf{A}}_\Gamma(l)
    &=&
    (-1)^{\mathsf{A}(l)-1}\left[e^{-\frac{\lann1,1\rann^{\mathsf{p}(l),0+}_{0,2}}{z}}
       \overline{\mathds{S}}_{\mathsf{p}(l)}(H)\right]_{z^{\mathsf{A}(l)-1}}\, .
 \end{eqnarray*}
The vertex and edge contributions are same as before.
\end{Prop}

The proof of Proposition \ref{VEL} 
follows the vertex and edge analysis. We leave the
details as an exercise for the reader.
The parallel statement for Gromov-Witten theory
can be found in \cite{Elliptic, SS,Book}.

\section{Vertices, edges, and legs} \label{svel}
\subsection{Overview}
Using the results of Givental \cite{Elliptic,SS,Book} combined with wall-crossing  \cite{CKg}, we calculate here 
the vertex and edge contributions
 in terms of the function $R_k$ of Section \ref{furcalc}.

\subsection{Calculations in genus 0}
We follow the notation introduced in Section \ref{intmg}. Recall
the series
$$T(c)=t_0 +t_1 c+t_2 c^2+\ldots\, .$$ 

\begin{Prop} {\em (Givental \cite{Elliptic,SS,Book})} For $n\geq 3$, we have
\begin{multline*}
\lann
 1,\ldots,1\rann_{0,n}^{p_i,\infty} = \\
 (\sqrt{\Delta_i})^{2g-2+n}\left(\sum_{k\geq 0}\frac{1}{k!}\int_{\overline{M}_{0,n+k}}T(\psi_{n+1})\cdots T(\psi_{n+k})\right)\Big|_{t_0=0,t_1=0,t_{j\ge 2}=(-1)^j Q_{j-1,i}}
 \end{multline*}
 where the functions $\sqrt{\Delta_i},\,Q_{l,i}$ are defined by 
 \begin{align*}
     \overline{\mathds{S}}^{\infty}_i(1) =  e_i \Big\langle \Big\langle  \frac{\phi _i}{z-\psi} , 1 
\Big\rangle \Big\rangle _{0, 2}^{p_i,\infty}=\frac{e^{\frac{\lann1,1\rann^{p_i,\infty}_{0,2}}{z}}}{\sqrt{\Delta_i}}
\left( 1+\sum_{l=1}^\infty Q_{l,i} z^{l}\right)\, .
 \end{align*}
\end{Prop}

The existence of the above asymptotic expansion of $\overline{\mathds{S}}^\infty_i(1)$ can also be proven by the argument of  \cite[Theorem 5.4.1]{CKg0}.
Similarly, we have an asymptotic expansion of $\overline{\mathds{S}}_i(1)$,
\begin{align*}
    \overline{\mathds{S}}_i(1) =e^{\frac{\lann1,1\rann^{p_i,0+}_{0,2}}{z}}
\left( \sum_{l=0}^\infty R_{l,i} z^{l}\right)\, .
\end{align*}
By \eqref{VS}, we have
\begin{align*}
    \lann1,1\rann^{p_i,0+}_{0,2} = \mu_i.
\end{align*}

After applying the wall-crossing result of Proposition \ref{WC}, we \
obtain
\begin{eqnarray*}
    \lann
 1,\ldots,1\rann_{0,n}^{p_i,\infty}(Q(q)) &=&\lann
 1,\ldots,1\rann_{0,n}^{p_i,0+}(q), \\
 \overline{\mathds{S}}^{\infty}_i(1)(Q(q))&=&\overline{\mathds{S}}_i(1)(q),
\end{eqnarray*}
where $Q(q)$ is mirror map for $K\PP^2$ as before. By comparing asymptotic expansions of $\overline{\mathds{S}}^{\infty}_i(1)$ and $\overline{\mathds{S}}_i(1)$, we get
a wall-crossing relation between $Q_{l,i}$ and $R_{l,i}$,
\begin{align*}
 \sqrt{\Delta_i}(Q(q))&=\frac{1}{R_{0,i}(q)}\,,\\
 Q_{l,i}(Q(q))&=\frac{R_{l,i}(q)}{R_{0,i}(q)}\,\,\, \text{for}\,\,l\ge 1.
\end{align*}
We have proven the following result.

\begin{Prop}\label{q2q2} For $n\geq 3$, we have \label{zaa3}
\begin{multline*}
\lann
 1,\ldots,1\rann_{0,n}^{p_i,0+} = \\
 R_{0,i}^{2g-2+n}\left(\sum_{k\geq 0}\frac{1}{k!}\int_{\overline{M}_{0,n+k}}T(\psi_{n+1})\cdots T(\psi_{n+k})\right)\Big|_{t_0=0,t_1=0,t_{j\ge 2}=(-1)^j\frac{R_{j-1,i}}{R_{0,i}}}\, .
 \end{multline*}
% \begin{align*}
%     \lann
% 1,\ldots,1\rann_{0,k}^{p_i,0+} = %\sum_{n}(\int_{\overline{M}_{g,k+n}}T\ldots T)|_{t_0=0,t_1=0,t_{j\ge %2}=(-1)^j\frac{R_{j-1}}{\lambda_i^{j-1}}}
% \end{align*}
\end{Prop}

 Proposition \ref{q2q2} immediately implies the evaluation
\begin{equation} \label{fxxf}
\lann
 1,1,1\rann_{0,3}^{p_i,0+}=\frac{1}{R_{0,i}}\, .
 \end{equation}
Another simple consequence of Proposition \ref{zaa3} is the following 
 basic property.
\begin{Cor}\label{Poly} For $n\geq 3$, we have
 $$
 \lann
 1,\ldots,1\rann_{0,n}^{p_i,0+} \in \CC[R_{0,i}^{\pm 1},R_{1,i},R_{2,i},...]
 \,.$$
\end{Cor}

\subsection{Vertex and edge analysis}\label{VerEdge}
By Proposition \ref{VE}, we have decomposition of the
contribution to $\Gamma\in \mathsf{G}_{g}(\PP^2)$ to
the stable quotient theory of 
$K\PP^2$ 
into vertex terms and edge terms
$$
 \text{Cont}_\Gamma
     =\frac{1}{|\text{Aut}(\Gamma)|}
     \sum_{\mathsf{A} \in \ZZ_{> 0}^{\mathsf{F}}} \prod_{v\in \mathsf{V}} 
     \text{Cont}^{\mathsf{A}}_\Gamma (v)
     \prod_{e\in \mathsf{E}} \text{Cont}^{\mathsf{A}}_\Gamma(e)\, .
 $$

%\begin{eqnarray*}
%\sum_{d\geq0}  q^d \left[ %\overline{Q}_{g}(K\PP^2,d)\right]^{\vir}%_{\Gamma}=\\
%     &\frac{1}{Aut(\Gamma)}\sum_{I \in %\ZZ_{\ge 0}^{\Gamma(F)}} \prod_v %Cont^I_\Gamma (v)\prod _e %%Cont^I_\Gamma(e).
%\end{align*}

\begin{Lemma}\label{L1} Suppose Conjecture \ref{RPoly} is true. Then we have
    $$\text{\em Cont}^{\mathsf{A}}_\Gamma (v)\in \mathds{G}_2\,.$$ 
\end{Lemma}

\begin{proof} By Proposition \ref{VE}, 
$$\text{Cont}^{\mathsf{A}}_\Gamma (v) = 
    \ppl
\psi_1^{a_1-1}, \ldots,
\psi_n^{a_n-1}
\, \Big|\, \mathsf{H}_{\mathsf{g}(v)}^{\mathsf{p}(v)}\,
  \ppr_{\mathsf{g}(v),n}^{\mathsf{p}(v),0+}\, .$$
 The right side of the above
  formula is a polynomial 
  in the variables
% \begin{align*}
%    Cont^I_\Gamma(v)=H^{g_v,p_v}[\psi^{a_1-1},\psi%^{a_2-1},...,\psi^{a_n-1}]^{p_v,0+}_{g_v,n}, 
% \end{align*}
%$[\psi^{a_1-1},\psi^{a_2-1},...,\psi^{a_n-1}]^{p_v,0+}_{g_v,n}$ is polynomial
$$\frac{1}{\lann 1,1,1\rann^{\mathsf{p}(v),0+}_{0,3}}\ \ \ 
\text{and} \ \ \
 \Big\{ \, \lann 1,\ldots,1\rann^{\mathsf{p}(v),0+}_{0,n}\, |_{t_0=0} \, \Big\}_{n\geq  4}\, $$
 with coefficients in $\mathbb{C}(\lambda_0,\lambda_1,\lambda_2)$.
%<1,1,...,1>^{0+,p_i}_{0,n}|_{t_0=0} : n \ge 3\} 
The Lemma then follows from 
the evaluation \eqref{fxxf}, Corollary \ref{Poly},
and 
Conjecture \ref{RPoly}. 
\end{proof}

Let $e\in \mathsf{E}$ be an edge connecting the $\T$-fixed points $p_i, p_j \in \PP^2$. Let
the $\mathsf{A}$-values of the respective
half-edges be $(k,l)$.

\begin{Lemma}\label{L2} Suppose Conjecture \ref{RPoly} is true. Then we have
 $\text{\em Cont}^{\mathsf{A}}_\Gamma(e) \in \mathds{G}[X]$ and
 %[\lambda_0,\lambda_1,\lambda_2,\lambda_0^{-1},\lambda_1^{-1},\lambda_2^{-1}] and
 \begin{enumerate}
 \item[$\bullet$]
 the degree of $\text{\em Cont}^{\mathsf{A}}_\Gamma(e)$ with respect to $X$ is $1$,
 \item[$\bullet$]
 the coefficient of $X$ in 
 $\text{\em Cont}^{\mathsf{A}}_\Gamma(e)$
 is 
 $$(-1)^{k+l}\frac{3 L_i L_jR_{1\,k-1,i} R_{1\,l-1,j}}{L^3}\, .$$ 
% where $(k,l)$ is the component of $I$ coresponding %to $e$.
 \end{enumerate}
\end{Lemma}

\begin{proof}
 By Proposition \ref{VE}, 
$$\text{Cont}^{\mathsf{A}}_\Gamma (e) = 
    (-1)^{k+l}\left[e^{-\frac{\mu \lambda_i}{x}-\frac{\mu \lambda_j}{y}}e_i\left(\overline{\mathds{V}}_{ij}-\frac{\delta_{ij}}{e_i(x+y)}\right)e_j
    \right]_{x^{k-1} y^{l-1}}\, .$$
 Using also the equation
 \begin{align*}
     e_i \overline{\mathds{V}}_{ij} (x, y) e_j  = 
\frac{\sum _{r=0}^2 \overline{\mathds{S}}_i (\phi_r)|_{z=x} \, \overline{\mathds{S}}_j (\phi ^r )|_{z=y}}{x+ y}\, ,
 \end{align*}
we write $\text{Cont}^{\mathsf{A}}_\Gamma (e)$
as
 \begin{align*}
\left[(-1)^{k+l} e^{-\frac{\mu \lambda_i}{x}-\frac{\mu \lambda_j}{y}}\sum_{r=0}^2\overline{\mathds{S}}_i(\phi_r)|_{z=x}\, \overline{\mathds{S}}_j(\phi^r)|_{z=y}
\right]_{x^{k}y^{l-1}-x^{k+1}y^{l-2}+
%x^{b_1+3}y^{b_2-2}-....
\ldots +(-1)^{k-1} x^{k+l-1}}
\end{align*}
where the subscript signifies a (signed) sum
of the respective coefficients.
If we substitute the asymptotic expansions \eqref{VS} for
$$\overline{\mathds{S}}_i(1)\, , \ \
\overline{\mathds{S}}_i(H)\, , \ \
\overline{\mathds{S}}_i(H^2)
$$ in the above expression, the Lemma follows from Conjecture \ref{RPoly}, Lemma \ref{RPoly2} and \eqref{NR}.
\end{proof}

\subsection{Legs}
Using the contribution formula of Proposition \ref{VEL},
\begin{eqnarray*}
     \text{Cont}^{\mathsf{A}}_\Gamma(l)
    &=&
    (-1)^{\mathsf{A}(l)-1}\left[e^{-\frac{\lann1,1\rann^{\mathsf{p}(l),0+}_{0,2}}{z}}
       \overline{\mathds{S}}_{\mathsf{p}(l)}(H)\right]_{z^{\mathsf{A}(l)-1}}\, ,
 \end{eqnarray*}
 we easily conclude under the assumption of Conjecture \ref{RPoly}
 
 $$C_1 \cdot \text{Cont}^{\mathsf{A}}_\Gamma(l)\in
\mathds{G}_2\, .$$

\section{Holomorphic anomaly for $K\PP^2$}
\label{hafp}

\subsection{Proof of Theorem \ref{MT1}}

By definition, we have
\begin{equation}\label{ffww}
A_2(q)= \frac{1}{L^3}\left(3X
+1 -\frac{L^3}{2}\right)\, .
\end{equation}
Conjecture \ref{RPoly} was proven in Appendix for the choices of $\lambda_0,\lambda_1,\lambda_2$ such that 
$$(\lambda_0\lambda_1+\lambda_1\lambda_2+\lambda_2\lambda_0)^2-3\lambda_0\lambda_1\lambda_2(\lambda_0+\lambda_1+\lambda_2)=0\,.$$
Hence, statement (i),
$$\mathcal{F}_g^{\mathsf{SQ}} (q) \in \mathds{G}_2[A_2]\, ,$$
follows from Proposition \ref{VE}
and  Lemmas \ref{L1} - \ref{L2}.
Statement (ii),
$\mathcal{F}_g^{\mathsf{SQ}}$ has at most degree $3g-3$ with respect to $A_2$, holds since a stable graph of genus $g$ has at most $3g-3$ edges.
Since 
$$\frac{\partial}{\partial T} = \frac{q}{C_1}\frac{ \partial}{\partial q}\,, $$
statement (iii),
\begin{equation}\label{vvtt}
\frac{\partial^k \mathcal{F}_g^{\mathsf{SQ}}}{\partial T^k}(q) \in \mathds{G}_2[A_2][C_1^{-1}]\, ,
\end{equation}
follows since the ring
$$\mathds{G}_2[A_2]=\mathds{G}_2[X]$$
is closed under the action of the differential operator $$\DD=q\frac{\partial}{\partial q}\, $$
by \eqref{drule}.
The degree of $C_1^{-1}$ in \eqref{vvtt}
is $1$ which yields statement (iv).
\qed

\subsection{Proof of Theorem \ref{HAE}}
\label{prttt}

Let $\Gamma \in \mathsf{G}_{g}(\PP^2)$ be a decorated graph. Let us fix an edge $f\in\mathsf{E}(\Gamma)$:
\begin{enumerate}
\item[$\bullet$] if $\Gamma$ is connected after 
deleting $f$, denote the resulting graph by $$\Gamma^0_f\in \mathsf{G}_{g-1,2}(\PP^2)\, ,$$
\item[$\bullet \bullet$] if $\Gamma$ is disconnected after deleting $f$, denote the resulting two graphs by $$\Gamma^1_f\in \mathsf{G}_{g_1,1}(\PP^2) \ \ \ 
\text{and}\ \ \  \Gamma^2_f\in \mathsf{G}_{g_2,1}(\PP^2)$$
where $g=g_1+g_2$.
\end{enumerate}
%Note that $$\Gamma'_f \in G_{g-1,2}$$ and
%$$\Gamma^1_f\in G_{g_1,1}, \Gamma^2_f\in %%G_{g_2,1},$$
%for some $g_1,g_2$ such that $g_1+g_2=g$.
There is no canonical
order for the 2 new markings. 
We will always sum over the 2 labellings. So more precisely, the graph
$\Gamma^0_f$ in case $\bullet$
should be viewed as sum
of 2 graphs
$$\Gamma^0_{f,(1,2)} +
\Gamma^0_{f,(2,1)}\, .$$
Similarly, in case $\bullet\bullet$,
we will sum over the ordering of $g_1$ and $g_2$. As usual, the summation
will be later compensated by a factor of
$\frac{1}{2}$ in the formulas.

By Proposition \ref{VE}, we have
the following formula for the contribution 
of the graph $\Gamma$ to the stable quotient
theory of $K\PP^2$,
 $$
 \text{Cont}_\Gamma
     =\frac{1}{|\text{Aut}(\Gamma)|}
     \sum_{\mathsf{A} \in \ZZ_{\ge 0}^{\mathsf{F}}} \prod_{v\in \mathsf{V}} 
     \text{Cont}^{\mathsf{A}}_\Gamma (v)
     \prod_{e\in \mathsf{E}} \text{Cont}^{\mathsf{A}}_\Gamma(e)\, .
 $$

%Recall the following decomposition of $\sum_{d} %{[Q_{g, 0} (X, d)]^{\vir}_\Gamma q^d}$ into vertex %term and edge term
%\begin{align*}
%    &[\sum_d q^d \overline{Q}_{g}(K\PP^2,d)]^{\vir}_{\Gamma}=\\
%     &\frac{1}{Aut(\Gamma)}\sum_{I \in \ZZ_{\ge 0}^{\Gamma(F)}} \prod_v Cont^I_\Gamma (v)\prod _e Cont^I_\Gamma(e).
%\end{align*}

Let $f$ connect the $\T$-fixed points $p_i, p_j \in \PP^2$. Let
the $\mathsf{A}$-values of the respective
half-edges be $(k,l)$. By Lemma \ref{L2}, we have
\begin{equation}\label{Coeff}
\frac{\partial \text{Cont}^{\mathsf{A}}_\Gamma(f)}{\partial X} = (-1)^{k+l}\frac{3 L_i L_jR_{1\,k-1,i} R_{1\,l-1,j}}{L^3}\, .
\end{equation}

\noindent $\bullet$ If $\Gamma$ is connected after deleting $f$, we have
%by Lemma \ref{L2}, we obtain
\begin{multline*}
%\frac{1}{\text{Aut}(\Gamma)}&\sum_{I \in \ZZ^{\Gamma(F)}} \frac{d Cont^I_\Gamma(f)}{dX} \prod_v Cont^I_\Gamma (v)\prod_{e \ne f}{Cont^I_\Gamma(e)} = 
\frac{1}{|\text{Aut}(\Gamma)|}
     \sum_{\mathsf{A} \in \ZZ_{\ge 0}^{\mathsf{F}}} 
     \left(\frac{L^3}{3C^2_1}\right)
     \frac{\partial {\text{Cont}}^{\mathsf{A}}_\Gamma(f)}{\partial X} 
     \prod_{v\in \mathsf{V}} 
     \text{Cont}^{\mathsf{A}}_\Gamma (v)
     \prod_{e\in \mathsf{E},\, e\neq f} \text{Cont}^{\mathsf{A}}_\Gamma(e) \\=
\frac{1}{2} \,
\text{Cont}_{\Gamma^0_f}(H,H) \, .
\end{multline*}
The derivation is simply by using \eqref{Coeff} on the left
and Proposition \ref{VEL} on the right.
%Here, $\text{Cont}_{\Gamma^0_f}$ is the full contribution 
%of $\Gamma^0_f$ to 
%$$\sum_{d\geq 0} q^d %\left[\overline{Q}_{g-1,2}(K\PP^2,d)\right]^{\text{vir}}\, .$$

\vspace{5pt}
\noindent $\bullet\bullet$
If $\Gamma$ is disconnected after deleting $f$, we obtain
\begin{multline*}
\frac{1}{|\text{Aut}(\Gamma)|}
     \sum_{\mathsf{A} \in \ZZ_{\ge 0}^{\mathsf{F}}} 
     \left(\frac{L^3}{3C^2_1}\right)
     \frac{\partial {\text{Cont}}^{\mathsf{A}}_\Gamma(f)}{\partial X} 
     \prod_{v\in \mathsf{V}} 
     \text{Cont}^{\mathsf{A}}_\Gamma (v)
     \prod_{e\in \mathsf{E},\, e\neq f} \text{Cont}^{\mathsf{A}}_\Gamma(e)\\
=\frac{1}{2}\,
\text{Cont}_{\Gamma^1_f}(H) \,
\text{Cont}_{\Gamma^2_f}(H)\, 
\end{multline*}
by the same method. 
%The sum here is over all distributions
%of 
%where $\text{Cont}_{\Gamma^1_f}$ and $\text{Cont}_{\Gamma^2_f}$ 
%are the full contribution 
%of $\Gamma^1_f$ and $\Gamma^2_f$
%to 
%$$\sum_{d\geq 0} q^d %\left[\overline{Q}_{g_1,1}(K\PP^2,d)\right]^{\text{vir}}
%\ \ \ \text{and} \ \ \
%\sum_{d\geq 0} q^d %\left[\overline{Q}_{g_2,1}(K\PP^2,d)\right]^{\text{vir}}
%\, $$
%r%espectively.

By combining the above two equations for all 
the edges of all the graphs $\Gamma\in \mathsf{G}_g(\PP^2)$
and using the vanishing
\begin{align*}
\frac{\partial {\text{Cont}}^{\mathsf{A}}_\Gamma(v)}{\partial X}=0
\end{align*}
of Lemma \ref{L1}, we obtain
%Since
%\begin{align*}
% \frac{\partial \prod_e Cont^I_\Gamma(e)}{\partial X}=\frac{\partial Cont^I_\Gamma(f)}{\partial X}\prod_{e \ne f} Cont^I_\Gamma(e)+ othercombinations,
%\end{align*}
%combining above two equations for all possible edges, we conclude
\begin{multline}\label{greww}
\left(\frac{L^3} {3C^2_1}\right) \frac{\partial}{\partial X} 
 \lan  \ran^{\mathsf{SQ}}_{g,0}= \frac{1}{2}\sum_{i=1}^{g-1} \lan H\ran^{\mathsf{SQ}}_{g-i,1}
\lan H \ran^{\mathsf{SQ}}_{i,1} + \frac{1}{2} \lan H,H\ran^{\mathsf{SQ}}_{g-1,2}\, .
\end{multline}
We have followed here the notation of Section \ref{holp2}.
The equality \eqref{greww} holds in the ring $\mathds{G}_2[A_2,C_1^{-1}]$.

Since $A_2=\frac{1}{L^3}(3X+1-\frac{L^3}{2})$ and 
$\lan \, \ran^{\mathsf{SQ}}_{g,0}=\mathcal{F}_g^{\mathsf{SQ}}$,
the left side of \eqref{greww} is, by the chain rule,
$$\frac{1}{C_1^2} \frac{\partial \mathcal{F}_g^{\mathsf{SQ}}}{\partial A_2} \in \mathds{G}_2[A_2,C_1^{-1}]\, .$$ 
On the right side of \eqref{greww}, we have 
$$ \lan H  \ran^{\mathsf{SQ}}_{g-i,1}\, =\, \mathcal{F}_{g-i,1}^{\mathsf{SQ}}(q)\, =\,  \mathcal{F}^{\mathsf{GW}}_{g-i,1}(Q(q))\, ,$$
where the first equality is by definition and the second is by
wall-crossing \eqref{3456}. Then,
$$\mathcal{F}^{\mathsf{GW}}_{g-i,1}(Q(q))\ = \ \frac{\partial\mathcal{F}^{\mathsf{GW}}_{g-i}}{\partial T}(Q(q)) \ =\ 
\frac{\partial\mathcal{F}^{\mathsf{SQ}}_{g-i}}{\partial T}(q) 
$$
where the first equality is by the divisor equation in
Gromov-Witten theory and the second is again by wall-crossing
\eqref{3456}, so we conclude
$$ \lan H  \ran^{\mathsf{SQ}}_{g-i,1} =\frac{\partial\mathcal{F}^{\mathsf{SQ}}_{g-i}}{\partial T}(q)\, \in \mathbb{C}(\lambda_0,\lambda_1,\lambda_2)[[q]]\, .$$
Similarly, we obtain
\begin{eqnarray*}
 \lan H  \ran^{\mathsf{SQ}}_{i,1} &=&\frac{\partial\mathcal{F}^{\mathsf{SQ}}_{i}}{\partial T}(q)\, 
\, \in \mathbb{C}(\lambda_0,\lambda_1,\lambda_2)[[q]]\, ,
\\
 \lan H,H  \ran^{\mathsf{SQ}}_{g-1,2} &=&\frac{\partial^2\mathcal{F}^{\mathsf{SQ}}_{g-1}}{\partial T^2}(q)\,
\, \in \mathbb{C}(\lambda_0,\lambda_1,\lambda_2)[[q]]\, .
\end{eqnarray*}
Together, the above equations transform \eqref{greww} into 
exactly the holomorphic anomaly equation of Theorem \ref{MT2},
$$\frac{1}{C_1^2}\frac{\partial \mathcal{F}_g^{\mathsf{SQ}}}{\partial{A_2}}(q)
= \frac{1}{2}\sum_{i=1}^{g-1} 
\frac{\partial \mathcal{F}_{g-i}^{\mathsf{SQ}}}{\partial{T}}(q)
\frac{\partial \mathcal{F}_i^{\mathsf{SQ}}}{\partial{T}}(q)
+
\frac{1}{2}
\frac{\partial^2 \mathcal{F}_{g-1}^{\mathsf{SQ}}}{\partial{T}^2}(q)\,
$$
as an equality in $\mathbb{C}(\lambda_0,\lambda_1,\lambda_2)[[q]]$.

The series $L$ and $A_2$ are expected to be algebraically independent. Since
we do not have a proof
of the independence, to lift holomorphic anomaly equation to the equality
$$\frac{1}{C_1^2}\frac{\partial \mathcal{F}_g^{\mathsf{SQ}}}{\partial{A_2}}
= \frac{1}{2}\sum_{i=1}^{g-1} 
\frac{\partial \mathcal{F}_{g-i}^{\mathsf{SQ}}}{\partial{T}}
\frac{\partial \mathcal{F}_i^{\mathsf{SQ}}}{\partial{T}}
+
\frac{1}{2}
\frac{\partial^2 \mathcal{F}_{g-1}^{\mathsf{SQ}}}{\partial{T}^2}\,
$$
in the ring $\mathds{G}_2[A_2,C_1^{-1}]$, we must
prove
the equalities 
\begin{equation}\label{pp33p}
 \lan H  \ran^{\mathsf{SQ}}_{g-i,1} =\frac{\partial\mathcal{F}^{\mathsf{SQ}}_{g-i}}{\partial T}\,,  \ \ \ \ 
 \lan H  \ran^{\mathsf{SQ}}_{i,1} = \frac{\partial\mathcal{F}^{\mathsf{SQ}}_{i}}{\partial T}\, ,
\end{equation}
$$ \lan H,H  \ran^{\mathsf{SQ}}_{g-1,2}= \frac{\partial^2\mathcal{F}^{\mathsf{SQ}}_{g-1}}{\partial T^2}\,
$$
in the ring $\mathds{G}_2[A_2,C_1^{-1}]$.
The lifting follow from the argument in Section 7.3 in \cite{LP}.

We do not study the genus 1 unpointed series $\mathcal{F}^{\mathsf{SQ}}_1(q)$ in the paper, so we take
\begin{eqnarray*}
 \lan H  \ran^{\mathsf{SQ}}_{1,1} &=&\frac{\partial\mathcal{F}^{\mathsf{SQ}}_{1}}{\partial T}\, ,\\
 \lan H,H  \ran^{\mathsf{SQ}}_{1,2} &=&\frac{\partial^2\mathcal{F}^{\mathsf{SQ}}_{1}}{\partial T^2}\, .
\end{eqnarray*}
as definitions of the right side in the genus 1 case.
There is no difficulty in calculating these series explicitly
using Proposition \ref{VEL}.

\section{Holomorphic anomaly for $K\PP^3$}
\label{hafp333}

\subsection{Overview}
We fix a torus action $\mathsf{T}=(\CC^*)^4$ on $\PP^3$ with
weights{\footnote{The associated weights on
$H^0(\PP^3,\mathcal{O}_{\PP^3}(1))$ are
$\lambda_0,\dots,\lambda_3$
and so match the conventions of
Section \ref{twth}.}}
$$-\lambda_0, \dots, -\lambda_3$$
on the vector space $\mathbb{C}^4$.
The $\T$-weight on the fiber over
$p_i$ of the canonical
bundle 
\begin{equation}\label{pqq9333}
\mathcal{O}_{\PP^3}(-4) \rightarrow \PP^3
\end{equation}
is $-4\lambda_i$.
The toric Calabi-Yau $K\PP^3$
is the total space of \eqref{pqq9333}.
The basic generating series and other essential objects defined in Section \ref{bcbcbc} -- Section \ref{svel} can be defined\footnote{In fact the contents of Section \ref{bcbcbc} -- \ref{svel} can be stated universally for all $K\PP^n$.} similarly for $K\PP^3$. We will not repeat the definitions of these objects unless necessary.

\subsection{I-functions}
\subsubsection{Evaluations}
%Denote by $\mathds{I}$, the big $\mathds{J}^{\theta,\epsilon}$-function for the $(0+,0+)$-stability defined above.

%begin{align*}
 %   \mathds{I}:= \mathds{J}^{0+,0+}.
%end{align*}

Let $\widetilde{H}\in H^*_\T([\CC^4/\CC^*])$ and $H\in H^*_\T(\PP^3)$
denote the respective hyperplane classes. The $\mathds{I}$-function
of Definition \ref{Je} for $K\PP^3$ is evaluated in \cite{BigI}.

\begin{Prop} For ${\bf t}=t\widetilde{H} \in H^*_{\T} ([\CC^4/\CC^*], \QQ)$,
 \begin{align}\label{I_Hyper_Paaa} 
\dsI({t}) = \sum _{d=0}^{\infty} q^d e^{t(H+dz)/z} \frac{ \prod _{k=0}^{4d-1}  (-4H - kz)}{\prod^3_{i=0}\prod _{k=1}^d (H-\lambda_i+kz)} . \end{align}

\end{Prop}

We define the series $I_{i,j}$ by following expansion of the $\mathds{I}$-function after restriction $t=0$,

$$\mathds{I}|_{t=0}=1+\frac{I_{10}H}{z}+\frac{I_{20}H^2+I_{21}H}{z^2}+\frac{I_{30}H^3+I_{31}H^2+I_{32}H}{z^3}+\mathcal{O}(\frac{1}{z^4})\,.$$
For example,
\begin{align*}
    I_{10}(q)&=\sum_{d=1}^{\infty} 4 \frac{(4d-1)!}{(d!)^4}q^d\,\,\in \CC[[q]]\,,\\
    I_{20}(q)&=\sum_{d=1}^{\infty} 4 \frac{(4d-1)!}{(d!)^4}\Big(4\text{Har}[4d-1]-4\text{Har}[d]\Big)q^d\,\,\in\CC[[q]]\,,\\
    I_{21}(q)&=\sum_{d=1}^{\infty} 4s_1 \frac{(4d-1)!}{(d!)^4}\text{Har}[d]q^d\,\,\in\CC(\lambda_0,\dots,\lambda_3)[[q]]\,.
\end{align*}
Here $\text{Har}[d]:=\sum_{k=1}^d\frac{1}{k}$.

We return now to the functions  $\mathds{S}_i(\gamma)$
defined in Section \ref{lightm}. We define the following additional series in $q$:
\begin{align*}
    &C_1=1+\mathsf{D}I_{10}\,,\,\,
    J_{10}=\frac{I_{10}+\mathsf{D}I_{20}}{C_1}\,,\,\,J_{11}=\frac{\mathsf{D}I_{21}}{C_1}\,,\\
    &J_{20}=\frac{I_{20}+\mathsf{D}I_{30}}{C_1}\,,\,\,J_{21}=\frac{I_{21}+\mathsf{D}I_{31}}{C_1}\,,\,\,J_{22}=\frac{I_{22}+\mathsf{D}I_{32}}{C_1}\,,\\
    &C_2=1+\mathsf{D}J_{10}\,,\,\,K_{10}=\frac{J_{10}+\mathsf{D}J_{20}}{C_2}\,,\\
    &K_{11}=\frac{J_{11}+\mathsf{D}J_{21}-(\mathsf{D}J_{11})J_{10}}{C_2}\,,\,\,K_{12}=\frac{\mathsf{D}J_{22}-(\mathsf{D}J_{11})J_{11}}{C_2}\,,\\
    &C_3=1+\mathsf{D}K_{10}\,.
\end{align*}
Here, $\mathsf{D}=q\frac{d}{dq}$.
The following relations were proven in \cite{ZaZi},
\begin{align}\label{c1c2laaa}
C_2&=C_3\,,\\
\nonumber C_1^2 C_2^2 &= (1-4^4q)^{-1}\, .
\end{align}
Using Birkhoff factorization, an evaluation of
the series $\mathds{S}(H^j)$ can be obtained from the $\dsI$-function, see \cite{KL}:
\begin{align}
\nonumber\mathds{S}({1}) & = \mathds{I} \, , \\
\label{S1aaa}\mathds{S}(H) & = \frac{  z\frac{d}{dt} \mathds{S}({1})}{C_1} \, , \\
\nonumber \mathds{S}(H^2) & = \frac{ z\frac{d}{dt} \mathds{S}(H)-(\mathsf{D}J_{11})\mathds{S}(H)}{C_2}\, ,\\
\nonumber \mathds{S}(H^3) & =\frac{z\frac{d}{dt} \mathds{S}(H^2)-(\mathsf{D}K_{11})\mathds{S}(H^2)-(\mathsf{D}K_{12})\mathds{S}(H)}{C_3}\, .
\end{align}

\subsubsection{Further calculations}\label{furcalcaaa}
Define small $I$-function 
%$$I(q)_\T
$$\overline{\mathds{I}}(q)
\in H^*_{\T}(\PP^3,\QQ)[[q]]$$ by the restriction
\begin{align*}
    \overline{\mathds{I}}(q)
%    I(q)_{\T}
    =\mathds{I}(q,{t})|_{t=0}\, .
\end{align*}
Define differential operators
$$\DD = q\frac{d}{dq}\, , \ \ \ M = H+ z \DD.$$
%Define $t$ by
%$$t:=e^{q}$$.
Applying $z\frac{d}{dt}$ to $\mathds{I}$ and then restricting
to $t=0$ has same effect as applying $M$ to 
$\overline{\mathds{I}}$
%$I_T$,
 \begin{align*}
     \left[\left(z\frac{d}{dt}\right)^k \mathds{I}\right]\Big|_{t=0} = M^k \, 
     \overline{\mathds{I}}\, .
     %I_T\, .
 \end{align*}
The function 
$\overline{\mathds{I}}$
%$I_{\T}$ 
satisfies following Picard-Fuchs equation
\begin{align}
\label{PFaaa}\Big(\prod_{j=0}^3(M-\lambda_j)-4qM(4M+z)(4M+2z)(4M+3z)\Big) 
\overline{\mathds{I}}=0
%I_\T=0
\end{align}
implied by the Picard-Fuchs equation for $\mathds{I}$,
\begin{align*}
    \left(\prod_{j=0}^3\left(z\frac{d}{dt}-\lambda_j\right)-q\prod_{k=0}^3\left(4 z\frac{d}{dt}+kz\right)\right)\mathds{I}=0\, .
\end{align*}

The restriction
%$I_{\T}|_{H=\lambda_i}$ 
$\overline{\mathds{I}}|_{H=\lambda_i}$
admits following asymptotic form
\begin{align}
 \label{assymaaa}
 %I_{\T}|_{H=\lambda_i}
 \overline{\mathds{I}}|_{H=\lambda_i}
 = e^{\mu_i/z}\left( R_{0,i}+R_{1,i} z+R_{2,i} z^2+\ldots\right)
\end{align}
with series 
$\mu_i,R_{k,i} \in \CC(\lambda_0,\dots,\lambda_d)[[q]]$.

A derivation of \eqref{assymaaa} is obtained in \cite{ZaZi} via  
the Picard-Fuchs equation \eqref{PFaaa} for
%$I_{\T}|_{H=\lambda_i}$,  
$\overline{\mathds{I}}|_{H=\lambda_i}$.
The series
$\mu_i$ and 
$R_{k,i}$ are found by solving differential equations obtained from the coefficient of $z^k$. 
For example, 
\begin{eqnarray*}
    \lambda_i+ \DD\mu_i&=& L_i\, , \\
    R_{0,i}&=&\Big(\frac{\lambda_i\prod_{j \ne i}(\lambda_i-\lambda_j)}{f(L_i)}\Big)^{\frac{1}{2}}\, .
\end{eqnarray*}

From the equations \eqref{S1aaa} and \eqref{assymaaa}, we can show the series $$\overline{\mathds{S}}_i({1})=\overline{\mathds{S}}({1})|_{H=\lambda_i}\,, \ \ \overline{\mathds{S}}_i(H)=
\overline{\mathds{S}}(H)|_{H=\lambda_i}\, , \ \ \overline{\mathds{S}}_i(H^2)=\overline{\mathds{S}}(H^2)|_{H=\lambda_i}\, , \ \ \overline{\mathds{S}}_i(H^3)=\overline{\mathds{S}}(H^3)|_{H=\lambda_i}$$ 
have the following asymptotic expansions:
\begin{align}\nonumber
\overline{\mathds{S}}_i({1}) & = e^{\frac{\mu_i}{z}} \Big(R_{00,i}+R_{01,i}z+R_{02,i} z^2+\ldots\Big) \, ,\\  \label{VSaaa}
\overline{\mathds{S}}_i(H) & = e^{\frac{\mu_i}{z}} \frac{L_i}{C_1} \Big(R_{10,i}+R_{11,i}z+R_{12} z^2+\ldots\Big)\, ,\\ \nonumber
\overline{\mathds{S}}_i(H^2) & = e^{\frac{\mu_i}{z}} \frac{L_i^2}{C_1 C_2} \Big(R_{20,i}+R_{21,i}z+R_{22,i} z^2+\ldots\Big)\, ,\\ \nonumber
\overline{\mathds{S}}_i(H^3) & = e^{\frac{\mu_i}{z}} \frac{L_i^3}{C_1 C_2 C_3} \Big(R_{30,i}+R_{31,i}z+R_{32,i} z^2+\ldots\Big)\, .
 \end{align}
We follow here the normalization of \cite{ZaZi}. Note 
\begin{align*}
    R_{0k,i}=R_{k,i}.
\end{align*}
As in \cite[Theorem 4]{ZaZi}, we expect the following  constraints.
 
\begin{Conj}\label{RPolyaaa}
 For all $k\geq 0$, we have
     $$R_{k,i} \in \mathds{G}_3\,.$$
\end{Conj}
Conjecture \ref{RPolyaaa} is the main obstruction for the proof of Conjecture \ref{ooo333} and \ref{HAE333}. By the same argument of Section \ref{hafp333}, we obtain the following result.
\begin{Thm}
Conjecture \ref{RPolyaaa} implies Conjecture \ref{ooo333} and \ref{HAE333}.
\end{Thm}

By applying asymptotic expansions \eqref{VSaaa} to \eqref{S1aaa}, we obtain the following results.

\begin{Lemma}\label{RRRaaa} We have
  \begin{align*}
      R_{1\,p+1,i}&=R_{0\,p+1,i}+\frac{\mathsf{D}R_{0\,p,i}}{L_i}\,,\\
      R_{2\,p+1,i}&=R_{1\,p+1,i}-E_{11,i}R_{1\,k,i}+\frac{\mathsf{D}R_{1\,p,i}}{L_i}+\Big(\frac{\mathsf{D}L_i}{L_i^2}-\frac{A_2}{L_i}\Big)R_{1\,p,i}\,,\\
      R_{3\,p+1,i}&=R_{2\,p+1,i}-E_{21,i}R_{2\,k,i}-E_{22,i}R_{1\,k,i}+\frac{\mathsf{D}R_{2\,p,i}}{L_i}+\Big(2\frac{\mathsf{D}L_i}{L_i^2}-\frac{A_2}{L_i}-\frac{\frac{\mathsf{D}C_2}{C_2}}{L_i}\Big)R_{1\,p,i}
  \end{align*}
with 
\begin{align*}
    E_{11,i}=\frac{\mathsf{D}J_{11}}{L_i}\,,\,\,E_{21,i}=\frac{\mathsf{D}K_{11}}{L_i}\,,\,\,E_{22,i}=\frac{C_2}{L_i^2}\mathsf{D}K_{12}\,.
\end{align*}
\end{Lemma}

\subsection{Determining $\DD A_2$ and new series.}
The following relation was proven in \cite{LP}.
\begin{equation}\label{druleaaa}
A_2^2+(L^4-1)A_2+2 \DD A_2-\frac{3}{16}(L^4-1)=0\, .
\end{equation}
 By the above result, the differential ring 
 \begin{equation}\label{ddd333aaa}
 \mathds{G}_3[A_2,\DD A_2,\DD\DD A_2,\ldots]
 \end{equation}
 is just the polynomial ring $\mathds{G}[A_2]$.
The second equation in \eqref{c1c2laaa} yields the following relation.
\begin{align}\label{drule2aaa}
    2 A_2+ 2\frac{\mathsf{D}C_2}{C_2}=L^4-1\,.
\end{align}
Denote by $\text{Coeff}(x^iy^j)$ the coefficient of $x^iy^j$ in
$$\sum _{k=0}^3 e^{-\frac{\mu_i}{x}-\frac{\mu_i}{y}}\mathds{S}_i (\phi_k)|_{z=x} \, \mathds{S}_i(\phi ^k )|_{z=y}\,.$$
From \eqref{wdvv} and \eqref{VSaaa}, we obtain the following equations.
\begin{align*}
&\text{Coeff}(x^2)-\frac{1}{2}\text{Coeff}(xy)=0\,,\\
&\text{Coeff}(x^4)-\text{Coeff}(x^3  y)+\frac{1}{2}\text{Coeff}(x^2y^2)=0\,.
\end{align*}
Above equations immediately yields the following relations.
\begin{align}\label{NRaaa}
    \nonumber E_{11,i}=&\frac{E_{21,i}}{2}-\frac{s_1 L^2}{2C_1 L_i}+\frac{s_1 L^4}{2L_i}\,,\\
     E_{22,i}=&\frac{L^4(s_1^2(-3+2 C_1 L^2+C_1^2 L^4)-4s_2(-1+C_1^2))}{8C_1^2 L_i^2}\,\\
    \nonumber&\frac{s_1(-3L^2+C_1 L^4)}{4C_1 L_i}E_{21,i}-\frac{3}{8}E_{21}^2\,.
\end{align}

We define the series $B_2$ and $B_4$ which appeared in the introduction by
\begin{align}\label{SB}
    B_2&=L_i E_{21,i}\,,\\
    \nonumber B_4&=\mathsf{D} B_2\,.
\end{align}
Note that $B_2(q), \,B_4(q)\in\CC[[q]]$.

From Lemma \ref{RRRaaa} with the relations \eqref{druleaaa}, \eqref{drule2aaa} and \eqref{NRaaa}, we obtain results for $\overline{\mathds{S}}({H})|_{H=\lambda_i}$, $\overline{\mathds{S}}({H^2})|_{H=\lambda_i}$
and $\overline{\mathds{S}}({H^3})|_{H=\lambda_i}$.
\begin{Lemma}\label{RPoly2aaa} Suppose Conjecture \ref{RPolyaaa} is true. Then for all $k\geq 0$, we have for all $k\geq 0$, 
 \begin{align*}
     R_{1\,k,i}, \,R_{2\,k,i}\,,R_{3\,k,i} \in \mathds{G}_3[A_2,B_2,B_4,C_1^{\pm 1}]\,.
 \end{align*}
\end{Lemma}

\subsection{Vertex, edge, and leg analysis}
By parallel argument as in Section \ref{hgs}, we have decomposition of the
contribution to $\Gamma\in \mathsf{G}_{g,k}(\PP^3)$ to
the stable quotient theory of 
$K\PP^3$ 
into vertex terms, edge terms and leg terms
$$
 \text{Cont}_\Gamma
     =\frac{1}{|\text{Aut}(\Gamma)|}
     \sum_{\mathsf{A} \in \ZZ_{> 0}^{\mathsf{F}}} \prod_{v\in \mathsf{V}} 
     \text{Cont}^{\mathsf{A}}_\Gamma (v)
     \prod_{e\in \mathsf{E}} \text{Cont}^{\mathsf{A}}_\Gamma(e)\,\prod_{e\in \mathsf{L}} \text{Cont}^{\mathsf{A}}_\Gamma(l).
 $$

%\begin{eqnarray*}
%\sum_{d\geq0}  q^d \left[ %\overline{Q}_{g}(K\PP^2,d)\right]^{\vir}%_{\Gamma}=\\
%     &\frac{1}{Aut(\Gamma)}\sum_{I \in %\ZZ_{\ge 0}^{\Gamma(F)}} \prod_v %Cont^I_\Gamma (v)\prod _e %%Cont^I_\Gamma(e).
%\end{align*}

The following lemmas follow from the argument in Section \ref{svel}.
\begin{Lemma}\label{L1bbb} Suppose Conjecture \ref{RPolyaaa} is true. Then we have
    $$\text{\em Cont}^{\mathsf{A}}_\Gamma (v)\in \mathds{G}_3\,.$$ 
\end{Lemma}

Let $e\in \mathsf{E}$ be an edge connecting the $\T$-fixed points $p_i, p_j \in \PP^3$. Let
the $\mathsf{A}$-values of the respective
half-edges be $(k,l)$.

\begin{Lemma}\label{L2bbb} Suppose Conjecture \ref{RPolyaaa} is true. Then we have
 $$\text{\em Cont}^{\mathsf{A}}_\Gamma(e) \in \mathds{G}_3[A_2,B_2,B_4,C_1^{\pm 1}]\,.$$
\end{Lemma}

\begin{Lemma}\label{L3bbb} Suppose Conjecture \ref{RPolyaaa} is true. Then we have
 $$\text{Cont}^{\mathsf{A}}_\Gamma(l)\in
\mathds{G}_3[A_2,B_2,B_4,C_1^{\pm 1}]\, .$$
\end{Lemma} 

\subsection{Proof of Theorem \ref{MT1333}}

Conjecture \ref{RPolyaaa} can be proven for the choices of $\lambda_0,\dots,\lambda_3$ such that 
\begin{align*}
    \lambda_i \ne \lambda_j \,\, \text{for}\,\,i \ne j\,,\\
    4s_2^2-s_1 s_3=0\,,\\
    2s_2^3-27 s_1^2s_4=0\,.
\end{align*}
by the argument in Appendix. 
Hence, statement (i),
$$\mathcal{F}_{g,a+b}^{\mathsf{SQ}}[a,b] (q) \in \mathds{G}_3[A_2,B_2,B_4,C_1^{\pm 1}]\, ,$$
follows from the arguments in Proposition \ref{VE}
and  Lemmas \ref{L1bbb} - \ref{L3bbb}.
Statement (ii),
$\mathcal{F}_g^{\mathsf{SQ}}$ has at most degree $2(3g-3)$ with respect to $A_2$, holds since a stable graph of genus $g$ has at most $3g-3$ edges.
Since 
$$\frac{\partial}{\partial T} = \frac{q}{C_1}\frac{ \partial}{\partial q}\,, $$
statement (iii),
\begin{equation}\label{vvttbbb}
\frac{\partial^k \mathcal{F}_g^{\mathsf{SQ}}}{\partial T^k}(q) \in \mathds{G}[A_2,B_2,B_4,C_1^{\pm 1}]\, ,
\end{equation}
follows from divisor equation in stable quotient theory and statement (i).
\qed

\subsection{Proof of Theorem \ref{HAE333}: first equation.}
\label{prtttbbb}

Let $\Gamma \in \mathsf{G}_{g}(\PP^3)$ be a decorated graph. Let us fix an edge $f\in\mathsf{E}(\Gamma)$:
\begin{enumerate}
\item[$\bullet$] if $\Gamma$ is connected after 
deleting $f$, denote the resulting graph by $$\Gamma^0_f\in \mathsf{G}_{g-1,2}(\PP^3)\, ,$$
\item[$\bullet \bullet$] if $\Gamma$ is disconnected after deleting $f$, denote the resulting two graphs by $$\Gamma^1_f\in \mathsf{G}_{g_1,1}(\PP^3) \ \ \ 
\text{and}\ \ \  \Gamma^2_f\in \mathsf{G}_{g_2,1}(\PP^3)$$
where $g=g_1+g_2$.
\end{enumerate}
%Note that $$\Gamma'_f \in G_{g-1,2}$$ and
%$$\Gamma^1_f\in G_{g_1,1}, \Gamma^2_f\in %%G_{g_2,1},$$
%for some $g_1,g_2$ such that $g_1+g_2=g$.
There is no canonical
order for the 2 new markings. 
We will always sum over the 2 labellings. So more precisely, the graph
$\Gamma^0_f$ in case $\bullet$
should be viewed as sum
of 2 graphs
$$\Gamma^0_{f,(1,2)} +
\Gamma^0_{f,(2,1)}\, .$$
Similarly, in case $\bullet\bullet$,
we will sum over the ordering of $g_1$ and $g_2$. As usual, the summation
will be later compensated by a factor of
$\frac{1}{2}$ in the formulas.

By the argument in Section \ref{VerEdge}, we have
the following formula for the contribution 
of the graph $\Gamma$ to the stable quotient
theory of $K\PP^3$,
 $$
 \text{Cont}_\Gamma
     =\frac{1}{|\text{Aut}(\Gamma)|}
     \sum_{\mathsf{A} \in \ZZ_{\ge 0}^{\mathsf{F}}} \prod_{v\in \mathsf{V}} 
     \text{Cont}^{\mathsf{A}}_\Gamma (v)
     \prod_{e\in \mathsf{E}} \text{Cont}^{\mathsf{A}}_\Gamma(e)\, .
 $$

%Recall the following decomposition of $\sum_{d} %{[Q_{g, 0} (X, d)]^{\vir}_\Gamma q^d}$ into vertex %term and edge term
%\begin{align*}
%    &[\sum_d q^d \overline{Q}_{g}(K\PP^2,d)]^{\vir}_{\Gamma}=\\
%     &\frac{1}{Aut(\Gamma)}\sum_{I \in \ZZ_{\ge 0}^{\Gamma(F)}} \prod_v Cont^I_\Gamma (v)\prod _e Cont^I_\Gamma(e).
%\end{align*}

Let $f$ connect the $\T$-fixed points $p_i, p_j \in \PP^3$. Let
the $\mathsf{A}$-values of the respective
half-edges be $(k,l)$. Denote by $\mathds{D}_1$ the differentail operator
$$\frac{L^2}{4C_1}\frac{\partial}{\partial A_2}+\frac{-2s_1L^4-C_1(3B_2L^2-s_1L^6)}{4C_1^2}\frac{\partial }{\partial B_4}\,.$$
By Lemma \ref{RRRaaa} and the explicit formula for $\text{Cont}^{\mathsf{A}}_\Gamma (f)$ in Lemma \ref{L2}\footnote{Lemma \ref{L2} is stated for $K\PP^2$, but parallel statement holds for $K\PP^3$.}, we have
\begin{multline}\label{Coeffbbb}
\mathds{D}_1\,\text{Cont}^{\mathsf{A}}_\Gamma(f) =(-1)^{k+l}\Big(\frac{ L_i^2 L_jR_{2\,k-1,i} R_{1\,l-1,j}}{C_1^2 C_2}+\frac{ L_i L_j^2 R_{1\,k-1,i} R_{2\,l-1,j}}{C_1^2 C_2}\Big)\, .
\end{multline}

\noindent $\bullet$ If $\Gamma$ is connected after deleting $f$, we have
%by Lemma \ref{L2}, we obtain
\begin{multline*}
%\frac{1}{\text{Aut}(\Gamma)}&\sum_{I \in \ZZ^{\Gamma(F)}} \frac{d Cont^I_\Gamma(f)}{dX} \prod_v Cont^I_\Gamma (v)\prod_{e \ne f}{Cont^I_\Gamma(e)} = 
\frac{1}{|\text{Aut}(\Gamma)|}
     \sum_{\mathsf{A} \in \ZZ_{\ge 0}^{\mathsf{F}}} 
     \mathds{D}_1\,\text{Cont}^{\mathsf{A}}_\Gamma(f)
     \prod_{v\in \mathsf{V}} 
     \text{Cont}^{\mathsf{A}}_\Gamma (v)
     \prod_{e\in \mathsf{E},\, e\neq f} \text{Cont}^{\mathsf{A}}_\Gamma(e) \\=
\text{Cont}_{\Gamma^0_f}(H,H^2)+\text{Cont}_{\Gamma^0_f}(H^2,H) \, .
\end{multline*}
The derivation is simply by using \eqref{Coeffbbb} on the left
and the argument in Proposition \ref{VEL} on the right.
%Here, $\text{Cont}_{\Gamma^0_f}$ is the full contribution 
%of $\Gamma^0_f$ to 
%$$\sum_{d\geq 0} q^d %\left[\overline{Q}_{g-1,2}(K\PP^2,d)\right]^{\text{vir}}\, .$$

\vspace{5pt}
\noindent $\bullet\bullet$
If $\Gamma$ is disconnected after deleting $f$, we obtain
\begin{multline*}
\frac{1}{|\text{Aut}(\Gamma)|}
     \sum_{\mathsf{A} \in \ZZ_{\ge 0}^{\mathsf{F}}} 
     \mathds{D}_1\,\text{Cont}^{\mathsf{A}}_\Gamma(f) 
     \prod_{v\in \mathsf{V}} 
     \text{Cont}^{\mathsf{A}}_\Gamma (v)
     \prod_{e\in \mathsf{E},\, e\neq f} \text{Cont}^{\mathsf{A}}_\Gamma(e)\\
=\text{Cont}_{\Gamma^1_f}(H) \,
\text{Cont}_{\Gamma^2_f}(H^2)+\text{Cont}_{\Gamma^1_f}(H^2) \,
\text{Cont}_{\Gamma^2_f}(H)\,
\end{multline*}
by the same method.

By combining the above two equations for all 
the edges of all the graphs $\Gamma\in \mathsf{G}_g(\PP^3)$
and using the vanishing
\begin{align*}
\frac{\partial {\text{Cont}}^{\mathsf{A}}_\Gamma(v)}{\partial A_2}=0\,,\,\,\frac{\partial {\text{Cont}}^{\mathsf{A}}_\Gamma(v)}{\partial B_4}=0
\end{align*}
of Lemma \ref{L1bbb}, we obtain
%Since
%\begin{align*}
% \frac{\partial \prod_e Cont^I_\Gamma(e)}{\partial X}=\frac{\partial Cont^I_\Gamma(f)}{\partial X}\prod_{e \ne f} Cont^I_\Gamma(e)+ othercombinations,
%\end{align*}
%combining above two equations for all possible edges, we conclude
\begin{equation}\label{grewwbbb}
\mathds{D}_1\, 
 \lan  \ran^{\mathsf{SQ}}_{g,0}= \sum_{i=1}^{g-1} \lan H\ran^{\mathsf{SQ}}_{g-i,1}
\lan H^2 \ran^{\mathsf{SQ}}_{i,1} + \lan H,H^2\ran^{\mathsf{SQ}}_{g-1,2}\, .
\end{equation}
We have followed here the notation of Section \ref{holp22}.
The equality \eqref{grewwbbb} holds in the ring $\mathds{G}_3[A_2,B_2,B_4,C_1^{\pm 1}]$.

\subsection{Proof of Theorem \ref{HAE333}: second equation.}

By the argument in Section \ref{VerEdge}, we have
the following formula for the contribution 
of the graph $\Gamma$ to the stable quotient
theory of $K\PP^3$,
 $$
 \text{Cont}_\Gamma
     =\frac{1}{|\text{Aut}(\Gamma)|}
     \sum_{\mathsf{A} \in \ZZ_{\ge 0}^{\mathsf{F}}} \prod_{v\in \mathsf{V}} 
     \text{Cont}^{\mathsf{A}}_\Gamma (v)
     \prod_{e\in \mathsf{E}} \text{Cont}^{\mathsf{A}}_\Gamma(e)\, .
 $$

Let $f$ connect the $\T$-fixed points $p_i, p_j \in \PP^3$. Let
the $\mathsf{A}$-values of the respective
half-edges be $(k,l)$. Denote by $\mathds{D}_2$ the differential operator
$$\frac{2L^2}{C_1(L^4-1-2A_2)}\frac{\partial }{\partial B_2}\,.$$
By Lemma \ref{RRRaaa} and the explicit formula for $\text{Cont}^{\mathsf{A}}_\Gamma (f)$ in Lemma \ref{L2}\footnote{Lemma \ref{L2} is stated for $K\PP^2$, but parallel statement holds for $K\PP^3$.}, we have
\begin{equation}\label{Coefftttbbb}
\mathds{D}_2\,\text{Cont}^{\mathsf{A}}_\Gamma(f) = (-1)^{k+l}\frac{2L_i L_jR_{1\,k-1,i} R_{1\,l-1,j}}{C_1^2}\, .
\end{equation}

\noindent $\bullet$ If $\Gamma$ is connected after deleting $f$, we have
%by Lemma \ref{L2}, we obtain
\begin{multline*}
%\frac{1}{\text{Aut}(\Gamma)}&\sum_{I \in \ZZ^{\Gamma(F)}} \frac{d Cont^I_\Gamma(f)}{dX} \prod_v Cont^I_\Gamma (v)\prod_{e \ne f}{Cont^I_\Gamma(e)} = 
\frac{1}{|\text{Aut}(\Gamma)|}
     \sum_{\mathsf{A} \in \ZZ_{\ge 0}^{\mathsf{F}}} 
     \mathds{D}_2\,\text{Cont}^{\mathsf{A}}_\Gamma(f) 
     \prod_{v\in \mathsf{V}} 
     \text{Cont}^{\mathsf{A}}_\Gamma (v)
     \prod_{e\in \mathsf{E},\, e\neq f} \text{Cont}^{\mathsf{A}}_\Gamma(e) \\=
\text{Cont}_{\Gamma^0_f}(H,H) \, .
\end{multline*}
The derivation is simply by using \eqref{Coefftttbbb} on the left
and the arguments in Proposition \ref{VEL} on the right.
%Here, $\text{Cont}_{\Gamma^0_f}$ is the full contribution 
%of $\Gamma^0_f$ to 
%$$\sum_{d\geq 0} q^d %\left[\overline{Q}_{g-1,2}(K\PP^2,d)\right]^{\text{vir}}\, .$$

\vspace{5pt}
\noindent $\bullet\bullet$
If $\Gamma$ is disconnected after deleting $f$, we obtain
\begin{multline*}
\frac{1}{|\text{Aut}(\Gamma)|}
     \sum_{\mathsf{A} \in \ZZ_{\ge 0}^{\mathsf{F}}} 
     \mathds{D}_2\,\text{Cont}^{\mathsf{A}}_\Gamma(f) \prod_{v\in \mathsf{V}} 
     \text{Cont}^{\mathsf{A}}_\Gamma (v)
     \prod_{e\in \mathsf{E},\, e\neq f} \text{Cont}^{\mathsf{A}}_\Gamma(e)\\
=\text{Cont}_{\Gamma^1_f}(H) \,
\text{Cont}_{\Gamma^2_f}(H)\, 
\end{multline*}
by the same method. 
%The sum here is over all distributions
%of 
%where $\text{Cont}_{\Gamma^1_f}$ and $\text{Cont}_{\Gamma^2_f}$ 
%are the full contribution 
%of $\Gamma^1_f$ and $\Gamma^2_f$
%to 
%$$\sum_{d\geq 0} q^d %\left[\overline{Q}_{g_1,1}(K\PP^2,d)\right]^{\text{vir}}
%\ \ \ \text{and} \ \ \
%\sum_{d\geq 0} q^d %\left[\overline{Q}_{g_2,1}(K\PP^2,d)\right]^{\text{vir}}
%\, $$
%r%espectively.

By combining the above two equations for all 
the edges of all the graphs $\Gamma\in \mathsf{G}_g(\PP^3)$
and using the vanishing
\begin{align*}
\frac{\partial {\text{Cont}}^{\mathsf{A}}_\Gamma(v)}{\partial B_2}=0
\end{align*}
of Lemma \ref{L1bbb}, we obtain
%Since
%\begin{align*}
% \frac{\partial \prod_e Cont^I_\Gamma(e)}{\partial X}=\frac{\partial Cont^I_\Gamma(f)}{\partial X}\prod_{e \ne f} Cont^I_\Gamma(e)+ othercombinations,
%\end{align*}
%combining above two equations for all possible edges, we conclude
\begin{equation}\label{grewwtttbbb}
\mathds{D}_2
 \lan  \ran^{\mathsf{SQ}}_{g,0}=\sum_{i=1}^{g-1} \lan H\ran^{\mathsf{SQ}}_{g-i,1}
\lan H \ran^{\mathsf{SQ}}_{i,1} +  \lan H,H\ran^{\mathsf{SQ}}_{g-1,2}\, .
\end{equation}
We have followed here the notation of Section \ref{holp22}.
The equality \eqref{grewwtttbbb} holds in the ring $\mathds{G}_3[A_2,B_2,B_4,C_1^{\pm 1}]$.

On the right side of \eqref{grewwtttbbb}, we have 
$$ \lan H  \ran^{\mathsf{SQ}}_{g-i,1}\, =\, \mathcal{F}_{g-i,1}^{\mathsf{SQ}}[1,0](q)\, =\,  \mathcal{F}^{\mathsf{GW}}_{g-i,1}[1,0](Q(q))\, ,$$
where the first equality is by definition and the second is by
wall-crossing \eqref{34567}. Then,
$$\mathcal{F}^{\mathsf{GW}}_{g-i,1}[1,0](Q(q))\ = \ \frac{\partial\mathcal{F}^{\mathsf{GW}}_{g-i}}{\partial T}(Q(q)) \ =\ 
\frac{\partial\mathcal{F}^{\mathsf{SQ}}_{g-i}}{\partial T}(q) 
$$
where the first equality is by the divisor equation in
Gromov-Witten theory and the second is again by wall-crossing
\eqref{34567}, so we conclude
$$ \lan H  \ran^{\mathsf{SQ}}_{g-i,1} =\frac{\partial\mathcal{F}^{\mathsf{SQ}}_{g-i}}{\partial T}(q)\, \in \mathbb{C}(\lambda_0,\dots,\lambda_3)[[q]]\, .$$
Similarly, we obtain
\begin{eqnarray*}
 \lan H  \ran^{\mathsf{SQ}}_{i,1} &=&\frac{\partial\mathcal{F}^{\mathsf{SQ}}_{i}}{\partial T}(q)\, 
\, \in \mathbb{C}(\lambda_0,\dots,\lambda_3)[[q]]\, ,
\\
 \lan H,H  \ran^{\mathsf{SQ}}_{g-1,2} &=&\frac{\partial^2\mathcal{F}^{\mathsf{SQ}}_{g-1}}{\partial T^2}(q)\,
\, \in \mathbb{C}(\lambda_0,\dots,\lambda_3)[[q]]\, .
\end{eqnarray*}
Together, the above equations transform \eqref{grewwtttbbb} into 
exactly the second holomorphic anomaly equation of Theorem \ref{MT2333},
$$\frac{2L^4}{C_1^2(L^4-1-2A_2)}\frac{\partial \mathcal{F}^{\mathsf{SQ}}_g}{\partial B_2}=\sum^{g-1}_{i=1} \frac{\partial\mathcal{F}^{\mathsf{SQ}}_{g-i}}{\partial T}\frac{\partial\mathcal{F}^{\mathsf{SQ}}_{i}}{\partial T}+\frac{\partial^2\mathcal{F}^{\mathsf{SQ}}_{g-1}}{\partial T^2}\,.$$
as an equality in $\mathbb{C}(\lambda_0,\dots,\lambda_3)[[q]]$. To lift holomorphic anomaly equation to the equality
$$\frac{2L^4}{C_1^2(L^4-1-2A_2)}\frac{\partial \mathcal{F}^{\mathsf{SQ}}_g}{\partial B_2}=\sum^{g-1}_{i=1} \frac{\partial\mathcal{F}^{\mathsf{SQ}}_{g-i}}{\partial T}\frac{\partial\mathcal{F}^{\mathsf{SQ}}_{i}}{\partial T}+\frac{\partial^2\mathcal{F}^{\mathsf{SQ}}_{g-1}}{\partial T^2}\,$$
in the ring $\mathds{G}_3[A_2,B_2,B_4,C_1^{\pm 1}]$, we must
prove
the equalities 
\begin{equation}\label{pp33pbbb}
 \lan H  \ran^{\mathsf{SQ}}_{g-i,1} =\frac{\partial\mathcal{F}^{\mathsf{SQ}}_{g-i}}{\partial T}\,,  \ \ \ \ 
 \lan H  \ran^{\mathsf{SQ}}_{i,1} = \frac{\partial\mathcal{F}^{\mathsf{SQ}}_{i}}{\partial T}\, ,
\end{equation}
$$ \lan H,H  \ran^{\mathsf{SQ}}_{g-1,2}= \frac{\partial^2\mathcal{F}^{\mathsf{SQ}}_{g-1}}{\partial T^2}\,
$$
in the ring $\mathds{G}_3[A_2,B_2,B_4,C_1^{\pm 1}]$.
The lifting follow from the argument in Section 7.3 in \cite{LP}.

We do not study the genus 1 unpointed series $\mathcal{F}^{\mathsf{SQ}}_1(q)$ in the paper, so we take
\begin{eqnarray*}
 \lan H  \ran^{\mathsf{SQ}}_{1,1} &=&\frac{\partial\mathcal{F}^{\mathsf{SQ}}_{1}}{\partial T}\, ,\\
 \lan H,H  \ran^{\mathsf{SQ}}_{1,2} &=&\frac{\partial^2\mathcal{F}^{\mathsf{SQ}}_{1}}{\partial T^2}\, .
\end{eqnarray*}
as definitions of the right side in the genus 1 case.
There is no difficulty in calculating these series explicitly
using the argument in Proposition \ref{VEL}.

\section{Appendix}
\subsection{Overviews.}\label{LPPN}In section \ref{twth} the equivariant Gromov-Witten invariants of the local $\PP^n$ were defined,
\begin{align*}
    N_{g,d}^{\mathsf{GW}}=\int_{[\overline{M}_g(\PP^n,d)]^{\text{vir}}}e\Big(-R\pi_*f^* \mathcal{O}_{\PP^n}(-n-1)\Big)\,.
\end{align*}
We associate Gromov-Witten generating series by
$$\mathcal{F}^{\mathsf{GW},n}_g(Q)\, =\, 
\sum_{d=0}^\infty \widetilde{N}_{g,d}^{\mathsf{GW}} Q^d\, 
\in \, \CC(\lambda_0,\dots,\lambda_n)[[Q]] \, .$$
Motivated by mirror symmetry (\cite{ASYZ,BCOV,ObPix}), we can make the following predictions about the genus $g$ generating series $\mathcal{F}^{\mathsf{GW},n}_g$. 
\begin{itemize}
    \item[(A)] There exist a finitely generated subring 
    $$\mathbf{G}\in\CC(\lambda_0,\dots,\lambda_n)[[Q]]$$
    which contains $\mathcal{F}_g^{\mathsf{GW},n}$ for all $g$.
    \item[(B)] The series $\mathcal{F}_g^{\mathsf{GW},n}$ satisfy {\em holomorphic anomaly equations}, i.e. recursive formulas for the derivative of $\mathcal{F}_g^{\mathsf{GW},n}$ with respect to some generators in $\mathbf{G}$.
\end{itemize}

\subsubsection{$I$-function}
$I$-fucntion defined by
$$I_n=\sum_{d=0}^{\infty}\frac{\prod_{k=1}^{(n+1)d-1}(-(n+1)H-kz)}{\prod_{i=0}^n \prod_{k}^d(H+kz-\lambda_i)}q^d\in H_\mathsf{T}^*(\PP^n,\CC)[[q]]\,,$$
is the central object in the study of Gromov-Witten invariants of local $\PP^n$ geometry. See \cite{LP}, \cite{LP2} for the arguments. Several important properties of the function $I_n$ was studied in \cite{ZaZi} after the specialization
\begin{align}\label{sp}
    \lambda_i=\zeta_{n+1}^i\,
\end{align}
where $\zeta_{n+1}$ is primitive $(n+1)$-th root of unity. For the study of full equivariant Gromov-Witten theories, we extend the result of \cite{ZaZi} without the specialization \eqref{sp}.

\subsubsection{Picard-Fuchs equation and Birkhoff factorization}
Define differential operators
\begin{align*}
    \mathsf{D}=q \frac{d}{dq}\,,\,\,\,M=H+z\mathsf{D}\,.
\end{align*}
The function $I_n$ satisfies following Picard-Fuchs equation
\begin{align*}
    \Big(\prod_{i=0}^n\Big(M-\lambda_i\Big)-q\prod_{k=0}^{n}\Big(-(n+1)M-kz\Big)\Big)I_n=0\,.
\end{align*}

The restriction $I_n|_{H=\lambda_i}$ admits following asymptotic form
\begin{align}\label{asymp}
    I_n|_{H=\lambda_i}=e^{\frac{\mu}{z}}\Big( R_{0,i}+R_{1,i}z+R_{2,i}z^2+\dots \Big)
\end{align}
with series $\mu_i\,,\,R_{k,i}\in \CC(\lambda_0,\dots,\lambda_n)[[q]]$.

 A derivation of \eqref{asymp} is obtained from \cite[Theorem 5.4.1]{CKg0} and the uniqueness lemma in \cite[Section 7.7]{CKg0}. The series $\mu_i$ and $R_{k,i}$ are found by solving defferential equations obtained from the coefficient of $z^k$. For example,
 \begin{align*}\lambda_i+\mathsf{D}\mu_i= L_i\,,
 \end{align*}
where $L_i(q)$ is the series in $q$ defined by the root of following degree $(n+1)$ polynomial in $\mathcal{L}$
\begin{align*}
    \prod_{i=0}^{n}(\mathcal{L}-\lambda_i)-(-1)^{n+1}q \mathcal{L}^{n+1}\,.\,
\end{align*}
with initial conditions,
$$\mathcal{L}_i(0)=\lambda_i\,.$$

Let $f_n$ be the polynomial of degree $n$ in variable $x$ over $\CC(\lambda_0,\dots,\lambda_n)$ defined by
\begin{align*}
    f_n(x):=\sum_{k=0}^{n} (-1)^{k}k s_{k+1} x^{n-k}\,,
\end{align*}
where $s_k$ is $k$-th elementary symmetric function in $\lambda_0,\dots,\lambda_n$. The ring
\begin{align*}
    \mathds{G}_n:=\CC(\lambda_0,\dots,\lambda_n)[L_0^{\pm 1},\dots,L_n^{\pm 1},f_n(L_0)^{-\frac{1}{2}},\dots,f_n(L_n)^{-\frac{1}{2}}]
\end{align*}
will play a basic role. 

The following Conjecture was proven under the specializaiton \eqref{sp} in \cite[Theorem 4]{ZaZi}.

\begin{Conj}\label{MC}
 For all $k\ge 0$, we have
$$R_{k,i}\in \mathds{G}_n\,.$$
\end{Conj}

Conjecture \ref{MC} for the case $n=1$ will be proven in Section \ref{PP1}. Conjecture \ref{MC} for the case $n=2$ will be proven in Section \ref{PP2} under the specialization \eqref{spl2}.
In fact, the argument in Section \ref{PP2} proves Conjecture \ref{MC} for all $n$ under the specialization which makes $f_n(x)$ into power of a linear polynomial.

\subsection{Admissibility of differential equations}
Let $\mathsf{R}$ be a commutative ring. Fix a polynomial $f(x)\in \mathsf{R}[x]$. We consider a differential operator of {\em level} $n$ with following forms.

\begin{align}
    \mathcal{P}(A_{lp},f)[X_0,\dots,X_{n+1}]=\mathsf{D}X_{n+1}-\sum_{n\ge l\ge 0,\,p\ge 0}A_{lp} \mathsf{D}^p X_{n-l}\,,
\end{align}
where $\mathsf{D}:=\frac{d}{dx}$ and $A_{lp}\in\mathsf{R}[x]_f:= \mathsf{R}[x][f^{-1}]$. We assume that only finitely many $A_{lp}$ are not zero.

\begin{Def}\label{admissible}
Let $R_i$ be the solutions of the equations for $k\ge 0$,
\begin{align}\label{DO}
    \mathcal{P}(A_{lp},f)[X_{k+1},\dots,X_{k+n}]=0\,,
\end{align}
with $R_0=1$. We use the conventions $X_{i}=0$ for $i<0$.
We say differential equations \eqref{DO} is {\em admissible} if the solutions $R_k$ satisfies for $k \ge 0$,
$$R_k\in \mathsf{R}[x]_f\,.$$
\end{Def}

\begin{Rmk}
Note that the admissibility of $\mathcal{P}(A_{lp},f)$ in Definition \ref{admissible} do not depend on the choice of the solutions $R_k$.  
\end{Rmk}

\begin{Lemma}
Let $f$ be a degree one polynomial in $x$. Each $A\in\mathsf{R}[x]_f$ can be written uniquely as
$$A=\sum_{i\in\ZZ}a_i f^i\,$$
with finitely many non-zero $a_i\in\mathsf{R}$. We define {\em the order} $\text{Ord}(A)$ of $A$ with respect to $f$ by smallest $i$ such that $a_i$ is not zero. Then

\begin{align*}
    \mathcal{P}(A_{lp},f)[X_0,\dots,X_{n+1}]:=\mathsf{D}X_{n+1}-\sum_{n\ge l\ge 0,\,p\ge 0}A_{lp} \mathsf{D}^p X_{n-l}=0\,,
\end{align*}
is admissible if following condition holds:
\begin{align}\label{D1C}
   \nonumber \text{Ord}(A_{l0})&\le -2\,,\\
    \text{Ord}(A_{l1})&\le 0\,,\\
   \nonumber \text{Ord}(A_{lp})&\le p+1\,\,\,\,\text{for}\,\,p\ge2\,.
\end{align}
\end{Lemma}

\begin{proof}
The proof follows from simple induction argument.
\end{proof}

\begin{Lemma}
Let $f$ be a degree two polynomial in $x$. Denote by
$$\mathsf{R}_f$$
the subspace of $\mathsf{R}[x]_f$ generated by $f^{i}$ for $i\in \ZZ$. Each $A\in\mathsf{R}_f$ can be written uniquely as
$$A=\sum_{i\in\ZZ}a_i f^i\,$$
with finitely many non-zero $a_i\in\mathsf{R}$. We define {\em the order} $\text{Ord}(A)$ of $A\in\mathsf{R}_f$ with respect to $f$ by smallest $i$ such that $a_i$ is not zero. Then
\begin{align*}
    \mathcal{P}(A_{lp},f)[X_0,\dots,X_{n+1}]:=\mathsf{D}X_{n+1}-\sum_{n\ge l\ge 0,\,p\ge 0}A_{lp} \mathsf{D}^p X_{n-l}=0\,,
\end{align*}
is admissible if following condition holds:

\begin{align*}
    A_{lp}&=B_{lp}\,\,\,\,\,\,\,\,\,\,\,\,\,\,\,\,\,\,\text{if p is odd}\,,\\
    A_{lp}&=\frac{df}{dx}\cdot B_{lp}\,\,\,\,\,\text{if p is even}\,,
\end{align*}
where $B_{lp}$ are elements of $\mathsf{R}_f$ with
\begin{align*}
    \text{Ord}(B_{l0})&\le -2\,,\\
    \text{Ord}(B_{lp})&\le \Big[\frac{p-1}{2}\Big]\,\,\,\,\text{for}\,\,p\ge1\,.
\end{align*}

\end{Lemma}

\begin{proof}
Since $f$ is degree two polynomial in $x$, we have
$$\frac{d^2f}{dx^2}\,,\,(\frac{df}{dx})^2\in \mathsf{R}_f\,.$$
Then the proof of Lemma follows from simple induction argument.
\end{proof}

\subsection{Local $\PP^1$}\label{PP1}
\subsubsection{Overview}In this section, we prove Conjecture \ref{MC} for the case $n=1$.
Recall the $I$-function for $K\PP^1$,
 \begin{align} 
I_1(q) = \sum _{d=0}^{\infty}  \frac{ \prod _{k=0}^{2d-1}  (-2H - kz)}{\prod^1_{i=0}\prod _{k=1}^d (H-\lambda_i+kz)}q^d . \end{align}
The function $I_1$ satisfies following Picard-Fuchs equation
\begin{align}
\label{PFlp1}\Big((M-\lambda_0)(M-\lambda_1)-2qM(2M+z)\Big) 
I_1=0\,.
%I_\T=0
\end{align}
Recall the notation used in above equation,
\begin{align*}
    \mathsf{D}=q \frac{d}{dq}\,,\,\,\,M=H+z\mathsf{D}\,.
\end{align*}
The restriction $I_1|_{H=\lambda_i}$ admits following asymptotic form
\begin{align}
    I_1|_{H=\lambda_i}=e^{\mu_i/z}\left( R_{0,i}+R_{1,i} z+R_{2,i} z^2+\ldots\right)
\end{align}
with series $\mu_i, R_{k,i}\in \CC(\lambda_0,\lambda_1)[[q]]$. The series $\mu_i$ and $R_{k,i}$ are found by solving differential equations obtained from the coefficient of $z^k$ in \eqref{PFlp1}. For example, we have for $i=0,1$,
\begin{align*}
    &\lambda_i+\mathsf{D}\mu_i=L_i\,,\\
    &R_{0,i}=\Big(\frac{\lambda_i\prod_{j \ne i}(\lambda_i-\lambda_j)}{f_1(L_i)}\Big)^{\frac{1}{2}}\,,\\
\end{align*}
\begin{multline*}
     R_{1,i}=\Big(\frac{\lambda_i\prod_{j \ne i}(\lambda_i-\lambda_j)}{f(L_i)}\Big)^{\frac{1}{2}} \cdot\\
     \Big(\frac{-16 s_1^2 s_2^2 + 
   88 s_2^3 + (27 s_1^3 s_2 - 132 s_1 s_2^2) L_i + (-12 s_1^4 + 
    54 s_1^2 s_2) L_i^2}{24 s_1 (L_i s_1 - 2 s_2)^3}\\+\frac{12 \lambda_i^2 - 9 \lambda_i \lambda_{i+1} + \lambda_{i+1}^2}{24 (\lambda_i^3 - \lambda_i \lambda_{i+1}^2)}\Big)\,.
\end{multline*}
Here $s_1=\lambda_0+\lambda_1$ and $s_2=\lambda_0 \lambda_1$. In the above expression of $R_{1,i}$, we used the convention $\lambda_2=\lambda_0$.

\subsubsection{Proof of Conjecture \ref{MC}.}
We introduce new differential operator $\mathsf{D}_i$ defined by for $i=0,1$,
$$\mathsf{D}_i=(\mathsf{D}L_i)^{-1}\mathsf{D}\,. $$
By definition, $\mathsf{D}_i$ acts on rational functions in $L_i$ as the ordinary derivation with respect to $L_i$.
If we use following normalizations,
\begin{align*}
    R_{k,i}=f_1(L_i)^{-\frac{1}{2}} \Phi_{k,i}
\end{align*}
the Picard-Fuchs equation \eqref{PFlp2} yields the following differential equations,
\begin{align}\label{DFlp1}
    \mathsf{D}_i\Phi_{p,i}-A_{00,i}\Phi_{p-1,i}-A_{01,i}\mathsf{D}_i\Phi_{p-1,i}-A_{02,i}\mathsf{D}_i^2\Phi_{p-1,i}=0\,,
\end{align}
with 
\begin{align*}
    A_{00,i}&=\frac{-s_1^2 s_2^2 + (-s_1^3 s_2 + 8 s_1 s_2^2) L_i + (2 s_1^4 - 9 s_1^2 s_2) L_i^2}{4 (L_i s_1 - 2 s_2)^4}\,,\\
    A_{01,i}&=\frac{2 s_1 s_2^2 + (-s_1^2 s_2 - 8 s_2^2) L_i + (-s_1^3 + 10 s_1 s_2) L_i^2 - s_1^2 L_i^3}{2 (L_i s_1 - 2 s_2)^3}\,,\\
    A_{02,i}&=\frac{s_2^2 - 2 (s_1 s_2) L_i + (s_1^2 + s_2) L_i^2 - s_1 L_i^3}{(L_i s_1 - 2 s_2)^2}\,.
\end{align*}
Here $s_k$ is the $k$-th elementary symmetric functions in $\lambda_0,\lambda_1$. 
Since the differential equations \eqref{DFlp1} satisfy the condition \eqref{D1C}, we conclude the differential equations \eqref{DFlp1} is admissible.

\subsubsection{Gomov-Witten series.}
By the result of previous subsection, we obtain the following result which verifies the prediction (A) in Section \ref{LPPN}.
\begin{Thm}\label{GWPP1}For the Gromov-Witten series of $K\PP^1$, we have 
 $$\mathcal{F}^{\mathsf{GW},1}_g(Q(q)) \in \mathds{G}_1\,,$$
where $Q(q)$ is the mirror map of $K\PP^1$ defined by
$$Q(q):=q\cdot\text{exp}\Big(2\sum_{d=1}^\infty\frac{(2d-1)!}{(d!)^2}q^d\Big)\,.$$
\end{Thm}
\noindent Theorem \ref{GWPP1} follows from the argument in \cite{LP}. The prediction (B) in Section \ref{LPPN} is trivial statement for $K\PP^1$.

\subsection{Local $\PP^2$}\label{PP2}
\subsubsection{Overview}In this section, we prove Conjecture \ref{MC} for the case $n=2$ with following specializations,
$$\lambda_i \ne \lambda_j \,\, \text{for}\,\,i \ne j\,,\\$$
\begin{align}\label{spl2}
    (\lambda_0 \lambda_1+\lambda_1 \lambda_2+\lambda_2\lambda_0)^2-3\lambda_0\lambda_1\lambda_2(\lambda_0+\lambda_1+\lambda_2)=0\,.
\end{align}
For the rest of the section, the specialization \eqref{spl2} will be imposed. Recall the $I$-function for $K\PP^2$.
 \begin{align*}
I_2(q) = \sum _{d=0}^{\infty}  \frac{ \prod _{k=0}^{3d-1}  (-3H - kz)}{\prod^2_{i=0}\prod _{k=1}^d (H-\lambda_i+kz)}q^d . \end{align*}

The function $I_2$ satisfies following Picard-Fuchs equation
\begin{align}
\label{PFlp2}\Big((M-\lambda_0)(M-\lambda_1)(M-\lambda_2)+3qM(3M+z)(3M+2z)\Big) 
I_2=0
%I_\T=0
\end{align}
Recall the notation used in above equation,
\begin{align*}
    \mathsf{D}=q \frac{d}{dq}\,,\,\,\,M=H+z\mathsf{D}\,.
\end{align*}

The restriction $I_2|_{H=\lambda_i}$ admits following asymptotic form
\begin{align*}
    I_2|_{H=\lambda_i}=e^{\mu_i/z}\left( R_{0,i}+R_{1,i} z+R_{2,i} z^2+\ldots\right)
\end{align*}
with series $\mu_i, R_{k,i}\in \CC(\lambda_0,\lambda_1,\lambda_2)[[q]]$. The series $\mu_i$ and $R_{k,i}$ are found by solving differential equations obtained from the coefficient of $z^k$ in \eqref{PFlp2}. For example,
\begin{align*}
    &\lambda_i+\mathsf{D}\mu_i=L_i\,,\\
     &R_{0,i}=\Big(\frac{\lambda_i\prod_{j \ne i}(\lambda_i-\lambda_j)}{f_2(L_i)}\Big)^{\frac{1}{2}}\,.\\
\end{align*}

\subsubsection{Proof of Conjecture \ref{MC}.}
We introduce new differential operator $\mathsf{D}_i$ defined by
$$\mathsf{D}_i=(\mathsf{D}L_i)^{-1}\mathsf{D}\,. $$

If we use following normalizations,
\begin{align*}
    R_{k,i}=f_2(L_i)^{-\frac{1}{2}} \Phi_{k,i}
\end{align*}
the Picard-Fuchs equation \eqref{PFlp2} yields the following differential equations,
\begin{multline}\label{DFlp2}
    \mathsf{D}_i\Phi_{p,i}-A_{00,i}\Phi_{p-1,i}-A_{01,i}\mathsf{D}_i\Phi_{p-1,i}-A_{02,i}\mathsf{D}_i^2\Phi_{p-1,i}\\-A_{10,i}\Phi_{p-2,i}-A_{11,i}\mathsf{D}_i\Phi_{p-2,i}-A_{12,i}\mathsf{D}_i^2\Phi_{p-2,i}-A_{13,i}\mathsf{D}_i^3\Phi_{p-2,i}=0\,,
\end{multline}
with $A_{jl,i}\in\CC(\lambda_0,\lambda_1,\lambda_2)[L_i,f_2(L_i)^{-1}]$.
We give the exact values of $A_{jl,i}$ for reader's convinience.
\begin{multline*}
    A_{00,i}=\frac{s_1}{9(s_1 L_i-s_2)^5}\Big(s_1 s_2^3+(-4s_1^2s_2^2+3s_2^3)L_i\\+(-s_1^3 s_2+12s1 s_2^2)L_i^2+(11s_1^4-36s_1^2 s_2)L_i^3\Big)\,,
\end{multline*}

\begin{multline*}
    A_{01,i}=\frac{-s_1}{3(s_1 L_i-s_2)^4}\Big(s_2^3-4(s_1 s_2^2)L_i+(3s_1^2s_2+9s_2^2)L_i^2\\+(3s_1^3-21s_1 s_2)L_i^3+3s_1^2 L_i^4   \Big)\,,
\end{multline*}

\begin{multline*}
    A_{02,i}=\frac{-1}{3(s_1 L_i-s_2)^3}\Big( s_2^3-5(s_1 s_2^2)L_i+9s_1^2 s_2 L_i^2+(-6s_1^3-3s_1 s_2)L_i^3\\+6s_1^2 L_i^4 \Big)\,,
\end{multline*}

\begin{multline*}
    A_{10,i}=\frac{s_1^2 L_i}{27(s_1 L_i-s_2)^9}\Big( (8 s_1^2 s_2^5 - 
   21 s_2^6) + (-48 s_1^3 s_2^4 + 126 s_1 s_2^5) L_i + (120 s_1^4 s_2^3 \\- 
    315 s_1^2 s_2^4) L_i^2 + (-124 s_1^5 s_2^2 + 264 s_1^3 s_2^3 + 
    144 s_1 s_2^4) L_i^3 \\+ (12 s_1^6 s_2 + 153 s_1^4 s_2^2 - 
    432 s_1^2 s_2^3) L_i^4 + (60 s_1^7 - 342 s_1^5 s_2 \\+ 
    432 s_1^3 s_2^2) L_i^5 + (-33 s_1^6 + 108 s_1^4 s_2) L_i^6    \Big)\,,
\end{multline*}

\begin{multline*}
    A_{11,i}=\frac{-s_1 L_i}{27(s_1 L_i-s_2)^8}\Big( (8 s_1^2 s_2^5 - 
   21 s_2^6) + (-48 s_1^3 s_2^4 + 126 s_1 s_2^5) L_i\\ + (120 s_1^4 s_2^3 - 
    315 s_1^2 s_2^4) L_i^2 + (-124 s_1^5 s_2^2 + 264 s_1^3 s_2^3 + 
    144 s_1 s_2^4) L_i^3 \\+ (12 s_1^6 s_2 + 153 s_1^4 s_2^2 - 
    432 s_1^2 s_2^3) L_i^4 + (60 s_1^7 - 342 s_1^5 s_2\\ + 
    432 s_1^3 s_2^2) L_i^5 + (-33 s_1^6 + 108 s_1^4 s_2) L_i^6 \Big)\,
\end{multline*}

\begin{multline*}
    A_{12,i}=\frac{s_1}{9(s_1 L_i-s_2)^7}\Big(-s_2^6 + 9 s_1 s_2^5 L_i + (-32 s_1^2 s_2^4 - 9 s_2^5) L_i^2 \\+ (57 s_1^3 s_2^3 + 
    60 s_1 s_2^4) L_i^3 + (-48 s_1^4 s_2^2 - 
    171 s_1^2 s_2^3) L_i^4\\ + (9 s_1^5 s_2 + 237 s_1^3 s_2^2 + 
    27 s_1 s_2^3) L_i^5 + (9 s_1^6 - 144 s_1^4 s_2 - 
    90 s_1^2 s_2^2) L_i^6 \\+ (9 s_1^5 + 108 s_1^3 s_2) L_i^7 - 18 s_1^4 L_i^8\Big)\,,
\end{multline*}

\begin{multline*}
    A_{13,i}=-\frac{(3 L_i^2 s_1^2 - 3 L_i s_1 s_2 + s_2^2) (-3 L_i^3 s_1 + 3 L_i^2 s_1^2 - 3 L_i s_1 s_2 + 
   s_2^2)^2}{27 (s_1 L_i - s_2)^6}\,.
\end{multline*}
Here $s_k$ is the $k$-th elementary symmetric functions in $\lambda_0,\lambda_1,\lambda_2$. 
Since the differential equations \eqref{DFlp2} satisfy the condition \eqref{D1C}, we conclude that the differential equations \eqref{DFlp2} is admissible.

\end{document}